\documentclass[11pt,leqno]{amsart}
\usepackage{epsfig}
\usepackage{amssymb}
\usepackage{amscd}
\usepackage{enumitem}
\setenumerate[1]{label=(\alph*)}

\usepackage{amsthm}
\usepackage{newlfont}
\usepackage{amsmath}
\usepackage{amsfonts}
\usepackage[mathscr]{euscript}
\usepackage[ngerman,french,english]{babel}
\usepackage{mathrsfs}
\usepackage{fancybox}
\usepackage{nicefrac}
\usepackage{pifont}
\usepackage[all]{xy}
\usepackage{stmaryrd}
\usepackage{srcltx}
\usepackage{ulem}
\usepackage{accents}
\usepackage{leftidx}

\newcommand{\myxysize}{0.4cm}

\usepackage{hyperref}

\newtheorem{theorem}{Theorem}[section]
\newtheorem{lemma}[theorem]{Lemma}
\newtheorem{definition}[theorem]{Definition}
\newtheorem{proposition}[theorem]{Proposition}

\newtheorem{corollary}[theorem]{Corollary}
\newtheorem{remark}[theorem]{Remark}

\newtheorem{example}[theorem]{Example}
\newtheorem{examples}[theorem]{Examples}

\numberwithin{equation}{section}

\def\pre-tr{\operatorname{pre-tr}}

\newcommand{\Tot}{\operatorname{Tot}}

\newcommand{\Sh}{\operatorname{Sh}}

\newcommand{\res}{\operatorname{res}}

\newcommand{\Sym}{\operatorname{Sym}}

\newcommand{\Lie}{\operatorname{Lie}}

\newcommand{\id}{\operatorname{id}}

\DeclareMathAlphabet{\mathpzc}{OT1}{pzc}{m}{it}

\newcommand{\groundring}{{\mathsf{k}}}

\newcommand{\matfak}[4]{{\xymatrix@C3ex{{{#1}} \ar@<0.3ex>[r]^{{#2}} & {{#3}} \ar@<0.3ex>[l]^{{#4}}}}}

\newcommand{\op}{\operatorname}

\newcommand{\ra}{\rightarrow}

\newcommand{\la}{\leftarrow}

\newcommand{\sra}{\twoheadrightarrow}
\newcommand{\hra}{\hookrightarrow}

\newcommand{\xra}[1]{\xrightarrow{#1}}
\newcommand{\xla}[1]{\xleftarrow{#1}}

\newcommand{\sira}{\xra{\sim}}
\newcommand{\xsira}[1]{\xrightarrow[\sim]{#1}}

\newcommand{\sila}{\overset{\sim}{\leftarrow}}

\newcommand{\xsra}[1]{\overset{#1}{\twoheadrightarrow}}

\newcommand{\mar}{\ar@{|->}}
\newcommand{\sar}{\ar@{->>}}
\newcommand{\iar}{\ar@{^{(}->}}
\newcommand{\gar}{\ar@{=}}
\newcommand{\gleichar}{\ar@{}|{=}}
\newcommand{\congar}{\ar@{}|{\cong}}

\newcommand{\Bl}[1]{{\mathbb{#1}}}
\newcommand{\DZ}{\Bl{Z}}
\newcommand{\DN}{\Bl{N}}

\newcommand{\DR}{\Bl{R}}
\newcommand{\DC}{\Bl{C}}

\newcommand{\internalHom}{\ul{\op{Hom}}}

\newcommand{\calMod}{\mathcal{M}od}
\newcommand{\Yoneda}[1]{{\widehat{#1}}}

\newcommand{\Mono}{{\op{\textbf{Mono}}}}

\newcommand{\ol}[1]{{\overline{#1}}}
\newcommand{\ul}[1]{{\underline{#1}}}

\newcommand{\leftadjointtores}{\op{prod}}
\newcommand{\pro}{\leftadjointtores}

\newcommand{\Kokern}{\op{cok}}

\newcommand{\inv}{^{-1}}
\newcommand{\can}{\op{can}}

\newcommand{\Fib}{\operatorname{Fib}}

\newcommand{\trFib}{\operatorname{trFib}}

\newcommand{\cof}{\operatorname{cof}}

\newcommand{\Obj}{{\op{Obj}\;}}

\newcommand{\per}{\op{per}}

\newcommand{\dgcat}{\op{dgcat}}

\newcommand{\point}{\text{pt}}

\newcommand{\HoMC}{\op{Ho}}

\newcommand{\as}{{{\raisebox{-0.2mm}{\scriptsize{*}}\mkern-7mu\raisebox{-0.1mm}{\scriptsize{!}}\mkern2mu}}}

\newcommand{\tzmat}[4]{{\left[\begin{smallmatrix} {#1} & {#2} \\ {#3} & {#4} \end{smallmatrix}\right]}}

\newcommand{\tzquadmat}[4]{{\begin{smallmatrix} {#1} & {#2} \\ {#3} & {#4} \end{smallmatrix}}}

\newcommand{\tildew}[1]{\widetilde{#1}}

\newcommand{\comp}{\circ}

\newcommand{\Dbl}{\op{Dbl}}
\newcommand{\tot}{\op{tot}}

\newcommand{\opp}{{\op{op}}}

\newcommand{\gr}{\op{gr}}

\newcommand{\mathovalbox}[1]{{\text{\ovalbox{${#1}$}}}}

\newcommand{\define}[1]{{\textbf{#1}}}

\numberwithin{equation}{section}

\newcommand{\inj}{\op{inj}}

\newcommand{\hinj}{\op{h-inj}}

\newcommand{\sweet}{good{}}

\newcommand{\hy}{\text{-}}

\newcommand{\svek}[2]{{\left[\begin{smallmatrix} {#1} \\ {#2} \end{smallmatrix}\right]}}

\renewcommand{\phi}{\varphi}

\newcommand{\Cone}{\op{Cone}}

\newcommand{\ampl}{\op{ampl}}

\title{Smoothness of equivariant derived categories}

\author{Valery A.~Lunts \and Olaf M.~Schn{\"u}rer}

\address{
  Department of Mathematics\\
  Indiana University\\
  Rawles Hall\\
  831 East 3rd Street\\
  Bloomington, IN 47405\\
  USA
}

\email{vlunts@indiana.edu} 

\address{
  Mathematisches Institut\\ 
  Universit{\"a}t Bonn\\
  Endenicher Allee 60\\
  53115 Bonn\\
  Germany
}

\email{olaf.schnuerer@math.uni-bonn.de}
 	


\begin{document}

\begin{abstract}
  We introduce the notion of (homological) $G$-smoothness for a complex $G$-variety
  $X,$ where $G$ is a connected affine algebraic group. This is 
  based on the notion of smoothness for dg algebras and uses a
  suitable enhancement of the 
  $G$-equivariant derived category of $X.$
  If there are only finitely many $G$-orbits and all stabilizers are
  connected, we show that $X$ is $G$-smooth if and only if all orbits
  $\mathcal{O}$ satisfy $H^*(\mathcal{O}; \DR)=\DR.$  
  On the way we prove several results concerning smoothness of dg
  categories over a graded commutative dg ring.
\end{abstract}

\maketitle
\tableofcontents

\section{Introduction}
\label{sec:introduction}

We introduce the notion of (homological) $G$-smoothness of a complex
$G$-variety $X,$ where $G$ is a connected complex affine
algebraic group. The idea is to describe the $G$-equivariant morphism
$X \ra \point$ in terms of dg algebras and to use 
the notion of
smoothness defined for such algebras. Under suitable conditions we give
sufficient and necessary conditions for $G$-smoothness of $X.$ 
This is based on results on smoothness for dg $K$-algebras (or categories)
over a graded commutative dg ring $K$ that are of independent interest.

We first explain the definition and results concerning $G$-smoothness.
We assume in the following that $G$ acts on $X$
with finitely many orbits and that all stabilizer subgroups are
connected. We work with the $G$-equivariant bounded
constructible derived category $D^b_{G,c}(X)$ of sheaves of real
vector spaces on $X$ (see \cite{BL}).  Using a suitable enhancement we
find  
a dg $H_G(\point)$-algebra $A$ such that the perfect derived category
$\per(A)$ of dg $A$-modules is equivalent to $D^b_{G,c}(X).$ 
The structure morphism $H_G(\point) \ra A$ may be thought of as an
analog of the $G$-morphism $X \ra \point$ (recall that
$D^b_{G,c}(\point)$ and $\per(H_G(\point))$ are equivalent).
Slightly generalizing the standard definition we
say that $A$ is $H_G(\point)$-smooth if the diagonal bimodule $A$ is
in $\per(A \otimes^L_{H_G(\point)} A^\opp).$ Then we define $X$ to be
$G$-smooth if $A$ is $H_G(\point)$-smooth.
Our first main result shows that $G$-smoothness can be tested
on the orbits.

\begin{theorem}
  [see Theorem~\ref{t:X-G-smooth-iff-all-orbits-G-smooth}]
  \label{t:X-G-smooth-iff-all-orbits-G-smooth-intro}
  Under the above conditions, 
  $X$ is $G$-smooth if and only if all $G$-orbits in $X$ are
  $G$-smooth. 
\end{theorem}

Hence we need to understand when an orbit is $G$-smooth. Our second
main result gives a criterion answering this question.

\begin{theorem}
  [see Theorem~\ref{t:one-orbit-case-smoothness-and-quisos} and
  references there]
  \label{t:one-orbit-case-smoothness-and-quisos-intro}
  Let $\mathcal{O}=G/H$ where $G$ is as above and $H \subset G$ is a
  closed connected subgroup.
  Then the following conditions are equivalent:
  \begin{enumerate}
  \item 
    $\mathcal{O}$ is $G$-smooth.
  \item 
    $H_G(\point) \ra H_G(\mathcal{O})=H_H(\point)$ is an isomorphism.
  \item 
    $H_H(\point)$ is a smooth dg $H_G(\point)$-algebra.
  \item
    Any maximal compact subgroup of $H$ is a maximal compact subgroup of $G.$
  \item
    $H^*(\mathcal{O};\DR)=\DR.$ 
  \item 
    $\mathcal{O} \cong \DC^n$ as complex varieties for some $n \in \DN.$
  \end{enumerate}
\end{theorem}

There are some more equivalent conditions given in
Theorem~\ref{t:one-orbit-case-smoothness-and-quisos}
which are actually
needed for the proof of
Theorem~\ref{t:X-G-smooth-iff-all-orbits-G-smooth-intro}.

For example, if $G$ is reductive and $B \subsetneq G$ is a Borel
subgroup, then the flag variety $G/B$ is
$B$-smooth but not $G$-smooth.

Let us provide some details on how the above dg
$H_{G}(\point)$-algebra $A$ is defined.
We fix a
suitable universal $G$-principal fiber bundle $EG \ra BG.$ Among other
things, this means that 
there is a (version of the) de Rham sheaf $\Omega_{BG}$ of
dg algebras on $BG$ computing $H_G(\point).$ Let $c: X_G:= EG \times_G
X \ra BG$ be the obvious map.
We identify $D^b_{G,c}(X)$ with a full subcategory of the derived
category of dg $c^*(\Omega_{BG})$-modules (= sheaves of dg modules
over the sheaf $c^*(\Omega_{BG})$ of dg algebras) on $X_G.$ 
Using an injective model structure we find a dg enhancement
$\mathcal{E}$ of
$D^b_{G,c}(X)$ that consists of h-injective dg
$c^*(\Omega_{BG})$-modules. 
Let $E \in \mathcal{E}$ be an object that is a classical generator of
$D^b_{G,c}(X).$ Then in the obvious way $A:= \mathcal{E}(E,E)$ is a dg
$\Gamma(\Omega_{BG})$-algebra. By composing the structure morphism
with a quasi-isomorphism $H_G(\point) \ra \Gamma(\Omega_{BG})$ we
obtain a dg $H_G(\point)$-algebra $A$ that we can use for the definition
of $G$-smoothness of $X,$ as explained above.

The proofs of the above theorems rely on some results on dg
$K$-categories, where $K$ is a
graded commutative dg ring, that we present now.
A dg $K$-category is by definition a
category enriched in the (abelian symmetric) monoidal category 
$C(K)$ of dg $K$-modules. For example, a dg $K$-category with one
object is a dg $K$-algebra. 
The obvious generalization of a result of 
G.~Tabuada \cite{tabuada-model-structure-on-cat-of-dg-cats}
shows that the category of small dg $K$-categories carries a
cofibrantly generated model structure whose weak equivalences are the
quasi-equivalences. In 
particular, any dg $K$-category 
(or algebra) $\mathcal{A}$ has a
cofibrant replacement $Q(\mathcal{A}) \ra \mathcal{A}$ which allows us
to define $\mathcal{A} \otimes^L_K \mathcal{A}^\opp= Q(\mathcal{A}) \otimes_K
Q(\mathcal{A}^\opp).$
We say that $\mathcal{A}$ is $K$-smooth if the
diagonal bimodule $\mathcal{A}$ is in
$\per(\mathcal{A} \otimes^L_K \mathcal{A}^\opp).$
Here, if $\mathcal{B}$ is a dg $K$-category, $\per(\mathcal{B})$ is
the perfect derived category of dg $\mathcal{B}$-modules, which can
also be characterized as the full subcategory of the derived category
of dg $\mathcal{B}$-modules consisting of compact objects.
Instead of $Q(\mathcal{A}) \ra \mathcal{A}$ we
can take any trivial fibration $\mathcal{A}' \ra \mathcal{A}$ with
$\mathcal{A}'$ $K$-h-flat for testing $K$-smoothness of $\mathcal{A}.$
Then $\mathcal{A}$ is $K$-smooth if and only if the diagonal bimodule
$\mathcal{A}$ is in $\per(\mathcal{A}' \otimes_K \mathcal{A}'^\opp).$

We generalize 
and strengthen
two results of \cite{lunts-categorical-resolution}:
Theorem~\ref{t:smoothness-preserved-by-dg-Morita-equivalence} says
that $K$-smooth\-ness is invariant under dg Morita equivalence, i.\,e.\
if the derived categories of two dg $K$-categories $\mathcal{A}$ and
$\mathcal{B}$ are connected by a zig-zag of tensor equivalences, then
$\mathcal{A}$ is $K$-smooth if and only if $\mathcal{B}$ is $K$-smooth.
Theorem~\ref{t:diagonal-smooth-bimod-smooth-iff-extension-smooth}
shows the following:
If $R$ and $S$ are dg $K$-algebras (or categories) and $N$ is a dg $S
\otimes_K R^\opp$-module, then
the dg $K$-algebra (or category)
$
\Big[\begin{smallmatrix}
  S & 0\\
  N & R
\end{smallmatrix}\Big]
$
is $K$-smooth if and only if $R$ and $S$ are $K$-smooth and $N$ is
an object of $\per(S \otimes^L_K R^\opp).$
This result is the key to the proof of
Theorem~\ref{t:X-G-smooth-iff-all-orbits-G-smooth-intro}: A
decomposition of $X$ into an open $G$-orbit and its closed complement
provides a dg $H_{G}(\point)$-algebra of this form.

For the proof of
Theorem~\ref{t:one-orbit-case-smoothness-and-quisos-intro}
we need two more results. The first one
is Theorem~\ref{t:smoothness-and-base-change} and explains how $K$-smoothness
behaves with respect to a base change: Let $K \ra K'$ be a morphism of dg rings.
If a dg $K$-category $\mathcal{A}$ is $K$-smooth, 
then $\mathcal{A} \otimes_K^L K' := Q(\mathcal{A}) \otimes_K K'$ is
$K'$-smooth. Moreover, the converse is true if $K \ra K'$ is a
quasi-isomorphism.
The second result is the slightly technical criterion for $K$-smoothness
given in Proposition~\ref{p:smoothness-over-local-graded-finite-homological-dim-algebras}.
It is based on an easier criterion, stated in 
Proposition~\ref{p:homologically-positive-dg-algebra-plus-conditions-not-smooth}:
Let $T$ be a dg algebra over a field $\groundring$ that satisfies $H^i(T)=0$
for $i<0,$ $H^0(T)=\groundring$ and $H^i(T)=0$ for $i\gg 0.$ Then $T$
is $\groundring$-smooth if and only if $H(T)=\groundring.$


Let us finally give some advice to the reader. 
It might be helpful
to take section~\ref{sec:smoothn-equiv-derived-cats} as a roadmap
and
to look at the results on dg $K$-categories
and their smoothness in sections
\ref{sec:diff-grad-K-categories-and-their-module-categories} and
\ref{sec:smoothn-dg-K-categories} when needed.
In section~\ref{sec:smoothn-equiv-derived-cats} we proceed as follows.
After some preparations in
section~\ref{sec:sheaves-dg-modules-over-dg-algebras}  
we define $G$-smoothness of $X$ in
section~\ref{sec:enhancem-equiv-deriv-cats}.
In section~\ref{sec:smoothness-homogeneous-spaces}
(and \ref{sec:results-for-homogeneous-space}) we treat
homogeneous spaces and prove
Theorem~\ref{t:one-orbit-case-smoothness-and-quisos-intro}. The
main ingredient there from the dg side is
Proposition~\ref{p:smoothness-over-local-graded-finite-homological-dim-algebras}
mentioned above. Even though its hypotheses
may seem restrictive
it is surprisingly useful. In section~\ref{sec:reduct-homog-spac} we prove
Theorem~\ref{t:X-G-smooth-iff-all-orbits-G-smooth-intro}. Thanks
to the assumption that $X$ consists of finitely many $G$-orbits,
we can decompose $X$ into an open $G$-orbit $U$ and its closed
complement $F.$ On the dg side this decomposition gives rise to a
lower triangular 
$H_{G}(\point)$-algebra 
$
\Big[\begin{smallmatrix}
  S & 0\\
  N & R
\end{smallmatrix}\Big]
$
which is dg Morita equivalent to $A,$ and $S$ (resp.\ $R$) is 
$H_{G}(\point)$-smooth if and only if $F$ (resp.\ $U$) is
$G$-smooth. 
Using 
Theorems~\ref{t:smoothness-preserved-by-dg-Morita-equivalence} 
and \ref{t:diagonal-smooth-bimod-smooth-iff-extension-smooth}
explained above we then deduce
Theorem~\ref{t:X-G-smooth-iff-all-orbits-G-smooth-intro}.

\subsection*{Acknowledgments}
\label{sec:acknowledgements}

We would like to thank Hanspeter Kraft and Bertrand To{\"e}n for useful
correspondence. For helpful comments, mainly on model categories, we are
grateful to Tobias Dyckerhoff and Steffen Sagave.
We also thank the referee for suggestions and comments.

The first author was partially supported by the NSF grant 0901301.

The second author is grateful to the first author and his family for
their hospitality during his visit to Bloomington. He also thanks 
Indiana University and in particular the people in the math department
for their hospitality. He is thankful to the
Collaborative Research Center SFB Transregio 45 
of the German Science foundation for support.
He was partially supported by the priority program SPP 1388 of the
German Science foundation.

\section{Differential graded \texorpdfstring{$K$}{K}-categories}
\label{sec:diff-grad-K-categories-and-their-module-categories}

This section generalizes in a straightforward manner well known
results from dg (= differential graded) 
categories over a commutative ring $\groundring$ to dg categories over
a graded commutative dg algebra $K$; for example we describe the
projective model structure on the category of modules over such a
dg $K$-category, and we equip the category $\dgcat_K$ of small dg
$K$-categories with a model structure (following G.~Tabuada 
\cite{tabuada-model-structure-on-cat-of-dg-cats}). The reader who is
familiar with the 
usual theory will find no surprises and is advised to pass directly to 
Section~\ref{sec:smoothn-dg-K-categories}.

Let $\groundring$ be a commutative (associative unital) ring and
$K$ a graded commutative dg (= differential
($\DZ$-)graded)
($\groundring$-)algebra, i.\,e.\
$K=\bigoplus_{p \in \DZ} K^p$ is a graded (associative unital)
$\groundring$-algebra (the structure morphism $\groundring\ra K$ lands in $K^0$ and in the
center of $K$) endowed with a $\groundring$-linear differential $d=(d^p:K^p \ra
K^{p+1})_{p \in \DZ}$ of degree one such that
$d(kl)=d(k)l+(-1)^{|k|}kd(l)$ for all elements $k,$ $l \in K$
with $k$ of degree $|k|$ (here and in the following we use the
convention that elements are assumed to be homogeneous if their degree
appears in a formula); the assumption that $K$ is graded commutative
means that $kl=(-1)^{|k||l|}lk$ for all $k,$ $l \in K.$
For example, $K$ could be $\groundring$ viewed as a dg algebra
concentrated in degree zero. We fix $\groundring$ and $K$ for the rest
of this article.

\subsection{Dg \texorpdfstring{$K$}{K}-categories}
\label{sec:dg-K-categories}

Let $C(K)$ be the abelian $\groundring$-linear category of (right) dg $K$-modules
(morphisms are $K$-linear, preserve the degree and commute with the respective
differentials). Given dg $K$-modules $M,$ $N,$ we can view $N$ as a left dg
$K$-module via $k.n:=(-1)^{|k||n|}nk$ and obtain the tensor product $M
\otimes_K N$ which is again an a dg $K$-module. In fact $C(K)$ becomes
a symmetric monoidal category in the obvious way, which is moreover 
closed:
Given any $M \in C(K),$ the functor $(? \otimes_K M)$ has an obvious
right adjoint denoted $\internalHom(M,?),$ i.\,e. $(C(K))(N \otimes_K M, P)=
(C(K))(N, \internalHom(M,P))$ naturally in $N$ and $P.$

This enables us to speak about dg $K$-categories 
(:= $C(K)$-(enriched )categories), dg $K$-functors (:=
$C(K)$-functors), and dg $K$-natural transformations (:=
$C(K)$-natural transformations), see \cite{kelly-enriched}.

To an arbitrary dg $K$-category $\mathcal{A}$ we can associate two
$\groundring$-linear categories, namely the category
$Z^0(\mathcal{A})$ and the homotopy category $[\mathcal{A}].$ They
have the same objects as $\mathcal{A},$ but their morphisms spaces are
given by the cocycles 
$(Z^0(\mathcal{A}))(A,A')= Z^0(\mathcal{A}(A,A))$ of degree zero
and by the cohomology classes 
$[\mathcal{A}](A,A')= H^0(\mathcal{A}(A,A'))$ of degree zero.

An example is the dg $K$-category $\calMod(K)$ of dg
$K$-modules. It has the same objects as $C(K),$ and its morphism
spaces are given by $(\calMod(K))(M, N)= \internalHom(M, N)$ where
$M,$ $N$ are dg $K$-modules.
Note that
\begin{equation*}
  (C(K))(M,N) = (C(K))(K, (\calMod(K))(M,N)) = Z^0((\calMod(K))(M,N)).
\end{equation*}
The first equality says that the underlying 
category of the dg $K$-category $\calMod(K)$ is $C(K),$ 
and then the second equality says that $C(K) = Z^0(\calMod(K)).$

\subsection{Module categories}
\label{sec:module-categories}

Let $\mathcal{A}$ be a small dg $K$-category.
A (right) dg $\mathcal{A}$-module $M$ is a dg $K$-functor
$M: \mathcal{A}^\opp \ra \calMod(K),$ 
where $\mathcal{A}^\opp$ is the opposite dg $K$-category.
More explicitly, such a functor is given by dg $K$-modules $M(A),$ for $A \in
\mathcal{A},$ and morphisms
\begin{equation*}
  M(A) \otimes_K \mathcal{A}(A',A) \ra M(A')
\end{equation*}
in $C(K),$ for $A,$ $A' \in \mathcal{A},$ that make the obvious
diagrams encoding unitality and associativity commutative.
We denote the category of dg $\mathcal{A}$-modules whose morphisms are
the dg $K$-natural transformations by $C(\mathcal{A}).$ This is an
abelian $\groundring$-linear category having all small limits and
colimits; we explain 
in Remark~\ref{rem:dependence-on-K}
below that it is essentially independent of $K.$

Again there is a dg $K$-category $\calMod(\mathcal{A})$ whose
underlying category is $C(\mathcal{A}).$ 
Let $M,$ $N$ be dg $\mathcal{A}$-modules. Then 
$(\calMod(\mathcal{A}))(M,N)$ is defined to be the
dg $K$-module
\begin{equation*}
  \Big\{(f(A)) \in \prod_{A \in \mathcal{A}} (\calMod(K))(M(A), N(A)) \mid
  \text{
    $N(a) f(A')= f(A'') M(a)$ for all 
    $a \in \mathcal{A}^\opp(A',A'')$
  }
  \Big\}. 
\end{equation*}
Similar as above we have
\begin{equation*}
  C(\mathcal{A})(M,N) 
  = Z^0((\calMod(\mathcal{A}))(M,N)).
\end{equation*}

We may consider $K$ as a dg $K$-category consisting of one object whose  
endomorphisms are $K.$ Then the definitions of $C(K)$ and $\calMod(K)$
are consistent with their previous definitions.

Any object $A \in \mathcal{A}$ gives rise to the dg
$\mathcal{A}$-module $\Yoneda{A} := \mathcal{A}(?,A)$ represented by $A.$
If $M$ is any dg $\mathcal{A}$-module, the map
\begin{equation}
  \label{eq:yoneda-dgA}
  (\calMod(\mathcal{A}))(\Yoneda{A}, M) \sira M(A),\quad
  f \mapsto (f(A))(\id_A),
\end{equation}
is an isomorphism in $C(K),$ the Yoneda-isomorphism. Taking degree
zero cocycles gives the isomorphism
\begin{equation*}
  (C(\mathcal{A}))(\Yoneda{A}, M) \sira Z^0(M(A)).
\end{equation*}

The dg $K$-functor 
\begin{equation*}
  \mathcal{A} \ra \calMod(\mathcal{A}), \quad
  A \mapsto \Yoneda{A}= \mathcal{A}(?,A),
\end{equation*}
is full and faithful by \eqref{eq:yoneda-dgA} and called the
Yoneda embedding.

Whenever we work with module categories over a dg $K$-category in the
following we implicitly assume that the given dg $K$-category is
small.

\subsection{Homotopy categories and derived categories}
\label{sec:homotopy-categories-and-derived-categories}

We have seen above that $Z^0(\calMod(\mathcal{A}))=C(\mathcal{A}).$ We
define $\mathcal{H}(\mathcal{A}):=[\calMod(\mathcal{A})]$ and call it
the homotopy category of dg $\mathcal{A}$-modules. There is an obvious
functor $C(\mathcal{A}) \ra \mathcal{H}(\mathcal{A}).$

In the usual way (cf.\ e.\,g.\ \cite[Ch.~10]{BL}) we equip
$\mathcal{H}(\mathcal{A})$ with the structure of a triangulated
category: One defines the
translation or shift functor $[1]$ on $\calMod(\mathcal{A})$ (and
$C(\mathcal{A}),$ $\mathcal{H}(\mathcal{A})$), and the
cone $\Cone(f)$ of a morphism $f:M \ra N$ in $C(\mathcal{A})$; this
cone fits into an obvious diagram $M \xra{f} N \ra \Cone(f) \ra [1]M$
in $C(\mathcal{A})$ called a standard triangle. 
We define a (distinguished) triangle in $\mathcal{H}(\mathcal{A})$ to
be a candidate triangle isomorphic to the image of a standard
triangle.
Then $\mathcal{H}(\mathcal{A})$ with the shift functor $[1]$ and this
class of triangles is a triangulated category.

A morphism $f: M \ra N$ in $C(\mathcal{A})$ (or
$\mathcal{H}(\mathcal{A})$) induces 
in the obvious way a morphism $H(f): H(M) \ra H(N)$
on cohomology. We call $f$ a quasi-isomorphism if $H(f)$ is an
isomorphism.

A dg $\mathcal{A}$-module $M$ is called acyclic if all $M(A)$ have
vanishing cohomology, i.\,e.\ $H(M(A))=0,$ for all $A \in
\mathcal{A}^\opp.$
Then a morphism in $\mathcal{H}(\mathcal{A})$ is a quasi-isomorphism
if and only if its cone is acyclic; here we mean by the cone of a
morphism $f$ in a triangulated category the third object
in a triangle whose first morphism is $f$ (it is well
defined up to isomorphism).

The derived category $D(\mathcal{A})$ of dg $\mathcal{A}$-modules is
defined to be the Verdier quotient of $\mathcal{H}(\mathcal{A})$ by
the (thick) triangulated subcategory of all acyclic dg
$\mathcal{A}$-modules. Note that $\mathcal{H}(\mathcal{A})$ and
$D(\mathcal{A})$ have all small coproducts and products, and the
functor $\mathcal{H}(\mathcal{A}) \ra 
D(\mathcal{A})$ commutes with coproducts and products.

A dg $\mathcal{A}$-module $P$ is called h-projective
if all morphisms $P \ra N$ in $C(\mathcal{A})$ with acyclic $N$
are homotopic to zero, 
i.\,e.\ $(\mathcal{H}(\mathcal{A}))(P, N)=0.$
For example all $[i]\Yoneda{A}$ for $A \in \mathcal{A}$ and $i \in
\DZ$ are h-projective
since
\begin{equation}
  \label{eq:yoneda-dgA-H0A}
  (\mathcal{H}(\mathcal{A}))([i]\Yoneda{A}, N) \sira H^{-i}(N(A)). 
\end{equation}
by \eqref{eq:yoneda-dgA}.

If $P$ and $M$ are dg $\mathcal{A}$-modules it is easy to see that the
canonical morphism 
\begin{equation}
  \label{eq:h-projective-and-morphisms-DA}
  (\mathcal{H}(\mathcal{A}))(P, M) \ra (D(\mathcal{A}))(P,M)
\end{equation}
is an isomorphism if $P$ is h-projective.

We define $\per(\mathcal{A})$ to be the
smallest strict full triangulated subcategory of $D(\mathcal{A})$ that
contains all dg $\mathcal{A}$-modules $\Yoneda{A},$ for $A \in
\mathcal{A},$ and is closed under summands. This category has an
alternative description. 
Let $D(\mathcal{A})^c$ be the full subcategory of $D(\mathcal{A})$
consisting of compact objects, i.\,e.\ objects $M$ such that
$(D(\mathcal{A}))(M,?)$ commutes with all coproducts.
Let $\mathcal{E}$ be the set of all objects
$[i]\Yoneda{A} \in D(\mathcal{A}),$ for $A \in \mathcal{A}$ and $i
\in \DZ.$
From 
\eqref{eq:h-projective-and-morphisms-DA} and \eqref{eq:yoneda-dgA-H0A}
we deduce that $\mathcal{E}$ consists of compact objects, and moreover
that $\mathcal{E}$ generates $D(\mathcal{A}).$
The arguments of
\cite{neeman-connection-TTYBR} (cf.\
\cite[Thm.~2.1.2]{bondal-vdbergh-generators}) 
show that 
\begin{equation}
  \label{eq:perA-equal-compactDA}
  \per(\mathcal{A})=D(\mathcal{A})^c
\end{equation}
and that
$D(\mathcal{A})$ is the 
smallest strict full triangulated subcategory of $D(\mathcal{A})$ that
contains $\mathcal{E}$
and is closed with respect to the formation of arbitrary
$D(\mathcal{A})$-coproducts, i.\,e.\ the
localizing subcategory of $D(\mathcal{A})$ generated by $\mathcal{E}$
is all of $D(\mathcal{A}).$

\begin{remark}
  \label{rem:dependence-on-K}
  Let $Z \ra K$ be a morphism of graded commutative dg algebras (for
  example the structure morphism $\groundring \ra K$).
  Let $\mathcal{A}$ be a dg $K$-category. We define 
  $\res^K_Z(\mathcal{A})$ to be the dg $Z$-category which is obtained from
  $\mathcal{A}$ by the obvious restriction along $Z \ra K.$

  Then one checks that the obvious restriction functor
  \begin{equation}
    \label{eq:restriction-from-K-to-Z}
    \res^\mathcal{A}_{\res^K_Z \mathcal{A}}: 
    C(\mathcal{A}) \sira C(\res^K_Z(\mathcal{A})) 
  \end{equation}
  is an isomorphism of $\groundring$-linear categories.
  This is just an elaborate version of the following fact.
  If $R' \ra R$ is a morphism of rings, and $A$ is an $R$-algebra,
  then the module categories of $A$ as an $R$-algebra and as an
  $R'$-algebra coincide (and only depend on the ring underlying $A$).

  The above isomorphism \eqref{eq:restriction-from-K-to-Z} in fact
  comes from an isomorphism
  $\res^K_Z(\calMod(\mathcal{A})) \sira \calMod(\res^K_Z \mathcal{A})$
  of dg $Z$-categories. Similarly, we have isomorphisms
  $\mathcal{H}(\mathcal{A}) \sira \mathcal{H}(\res^K_Z(\mathcal{A}))$
  and
  $D(\mathcal{A}) \sira D(\res^K_Z(\mathcal{A}))$
  of $\groundring$-linear categories.
\end{remark}

\subsection{Projective model structure for dg \texorpdfstring{$\mathcal{A}$}{A}-modules}
\label{sec:proj-model-structure-dg-A-modules}

We refer to \cite{hovey-model-categories} 
(and \cite[App.]{lurie-higher-topos})
for the language of model
categories. However we do not assume that 
functorial factorizations are part of a model structure.
Since all model categories we consider will be cofibrantly generated
we can fix such factorizations whenever convenient. 

We use the following terminology.
If (P) is a property of objects in a model category, we define a
\define{(P) resolution} (of an object $X$) to be a trivial fibration
whose domain has property (P) (and whose codomain is $X$).
It follows for example from the definition of a model category that
any object has a cofibrant resolution. 

Let $\mathcal{A}$ be a dg $K$-category.
For $A \in \mathcal{A}$ and $n \in \DZ$ define dg
$\mathcal{A}$-modules $S_{n,A}:= [n]\Yoneda{A}$ and
$D_{n,A}:= \Cone(\id_{S_{n,A}}).$ There are obvious morphisms 
$\iota_{n,A}: S_{n,A} \ra D_{n,A}$ in $C(\mathcal{A}).$ 
Define the following sets of morphisms in $C(\mathcal{A})$:
\begin{align*}
  I &:= \{S_{n,A} \xra{\iota_{n,A}} D_{n,A} \mid A \in \mathcal{A}, n \in \DZ\},\\
  J &:= \{0 \ra D_{n,A} \mid A \in \mathcal{A}, n \in \DZ\}.
\end{align*}

\begin{theorem}
  \label{t:CA-cofib-gen-model-cat}
  Let $\mathcal{A}$ be a (small) dg $K$-category.
  The category $C(\mathcal{A})$ can be equipped with the structure of a
  cofibrantly generated model category whose 
  weak equivalences are the quasi-isomorphisms and whose fibrations
  are the epimorphisms.
  One can take $I$ as the set of generating cofibrations and $J$ as
  the set of generating trivial cofibrations.
  
  We call this model structure on $C(\mathcal{A})$ the
  \define{projective} model structure.
\end{theorem}

\begin{remark}
  Theorem~\ref{t:CA-cofib-gen-model-cat}
  seems to be well known, at least for $K = \groundring$
  (cf.\ \cite[Thm.~3.2]{keller-on-dg-categories-ICM} or
  \cite[3.2]{toen-lectures-dg-cats})
  and in the dg algebra
  case \cite[11.2.6]{fresse-modules-over-operads}.
  Presumably one can deduce its existence also from
  \cite[App.~3]{lurie-higher-topos} and even see that it is a
  $C(K)$-model structure.
  Our approach is elementary and essentially follows
  \cite[Section~2.3]{hovey-model-categories}. 
\end{remark}

\begin{proof}
  Let $\mathcal{W}$ be the class of quasi-isomorphisms in $C(\mathcal{A}).$
  Adapting the method of \cite[Section~2.3]{hovey-model-categories} (and using
  the notation explained there) one
  proves that
  $J\hy\inj$ consists precisely of epimorphisms,
  that $I \hy\inj = \mathcal{W} \cap J\hy\inj,$ 
  that projective objects of $C(\mathcal{A})$ are acyclic
  and that $J\hy\cof$ consists precisely of (split) monomorphisms with
  cokernel a projective object of $C(\mathcal{A}),$ so in particular
  $J\hy\cof \subset \mathcal{W}.$
  Then application of 
  \cite[Thm~2.1.19 and Lemma~2.1.10]{hovey-model-categories} shows the
  result.
\end{proof}

Note that any object of $C(\mathcal{A})$ is fibrant.
Examples of cofibrant objects are the objects $\Yoneda{A}$ and their
shifts (take the pushout of a map in $I$ along the morphism to the
zero object). 

The proof of
\cite[2.3.9]{hovey-model-categories},
adapted\footnote{
  For this we need the following result whose proof is similar to the proof
  of \cite[2.3.6]{hovey-model-categories}:
  Let $C$ be a cofibrant dg $\mathcal{A}$-module. 
  Then given any epimorphism $p: M \ra N$ in $C(\mathcal{A})$ and any
  morphism $f: C \ra N$ in $\calMod(\mathcal{A})$ of degree zero
  (i.\,e.\ a
  morphism of graded
  $\mathcal{A}$-modules),
  there is a morphism $\hat{f}:C \ra M$ in $\calMod(\mathcal{A})$ of
  degree zero such that $p \hat{f} =f.$
}
to our setting,
shows that the cofibrations are precisely the monomorphisms with cofibrant cokernel.
This fact and the trivial fact that cofibrations are closed under
composition shows the following two lemmata.

\begin{lemma}
  \label{l:mapping-cones-of-morphisms-between-cofibrants}  
  Let $f: M \ra N$ be a morphism in $C(\mathcal{A})$ between cofibrant
  objects. Then the canonical morphism $N \ra \Cone(f)$ is a
  cofibration, and in particular $\Cone(f)$ is cofibrant.
\end{lemma}

\begin{lemma}
  \label{l:extension-of-cofibrant-dgAmodules}
  Let $M' \hra M \sra M''$ be a short exact sequence in
  $C(\mathcal{A})$ and assume that $M'$ and $M''$ are cofibrant.
  Then the inclusion $M' \hra M$ is a cofibration
  and $M$ is cofibrant.
  (In fact this short exact sequence is isomorphic
  to the standard short exact sequence $M' \hra \Cone(\tau) \sra M''$
  for some morphism $\tau:[-1]M'' \ra M'$ in $C(\mathcal{A}).$)
\end{lemma}

\begin{lemma}
  [{cf.~\cite[Lemma~2.3.8]{hovey-model-categories}}]
  \label{l:cofobj-is-h-proj-dgA}
  Any cofibrant object of $C(\mathcal{A})$ is h-projective.
\end{lemma}

\begin{proof}
  Let $C \in C(\mathcal{A})$ be cofibrant.
  Let $f: C \ra N$ be a morphism in $C(\mathcal{A})$ and assume that
  $N$ is acyclic. Then the obvious epimorphism $p:[-1]\Cone(\id_N) \ra
  N$ is a quasi-isomorphism and hence a trivial 
  fibration. Since $C$ is cofibrant there is a lift $h: C \ra
  [-1]\Cone(\id_N)$ such that $ph=f.$ This lift has the form $h=
  \svek{D}{f}
  $ 
  and commutes with the differential.
  This implies that $f=d_N(-D)+(-D)d_C$ and hence $f=0$ in $\mathcal{H}(\mathcal{A}).$
\end{proof}

Denote by $C(\mathcal{A})_{cf}$
(resp.~$\mathcal{H}(\mathcal{A})_{cf}$) the full subcategory of
$C(\mathcal{A})$ (resp.~$\mathcal{H}(\mathcal{A})$) consisting of
cofibrant (and fibrant) objects.
From Lemma~\ref{l:mapping-cones-of-morphisms-between-cofibrants} 
we see that $\mathcal{H}(\mathcal{A})_{cf}$ is a triangulated
subcategory of $\mathcal{H}(\mathcal{A})$ (non-strict in general).  
Since any object of $C(\mathcal{A})$ has a cofibrant resolution, 
Lemma~\ref{l:cofobj-is-h-proj-dgA} and
\eqref{eq:h-projective-and-morphisms-DA}
immediately imply that the canonical triangulated functor
\begin{equation}
  \label{eq:homotopy-cat-cofibrants-isom-derived-cat}
  \mathcal{H}(\mathcal{A})_{cf} \sira D(\mathcal{A})
\end{equation}
is an equivalence.
We fix for any dg $\mathcal{A}$-module $M$ a cofibrant (and hence
h-projective) resolution $p(M) \ra M$ (we could even assume that $p:
C(\mathcal{A}) \ra C(\mathcal{A})_{cf}$ is a functor).
Then $M \mapsto p(M)$ extends to a functor
\begin{equation}
  \label{eq:homotopy-cat-cofibrants-isom-derived-cat-quasi-inverse}
  p: D(\mathcal{A}) \ra \mathcal{H}(\mathcal{A})_{cf}
\end{equation}
which is quasi-inverse to 
\eqref{eq:homotopy-cat-cofibrants-isom-derived-cat}.
We will use $p$ for (left-)deriving certain functors.

A dg $\mathcal{A}$-module $F$ is \define{free} if it is isomorphic in
$C(\mathcal{A})$ to a coproduct of shifts of objects $\Yoneda{A},$ where
$A$ varies in $\mathcal{A}.$ A dg $\mathcal{A}$-module $F$ is called \define{semi-free}
(cf.\ \cite[13.1, 14.8]{drinfeld-dg-quotients})
if it can be represented as the union of an increasing sequence of
dg $\mathcal{A}$-submodules $F_i$ (where $i \in \DN$) such that $F_0=0$ and
each quotient $F_{i+1}/F_{i}$ is a free dg $\mathcal{A}$-module.

\begin{lemma}
  \label{l:cofibrant-dg-A-modules}
  \rule{1mm}{0mm}
  \begin{enumerate}
  \item 
    \label{enum:dg-A-semifree-Icell-cofib}
    All semi-free dg $\mathcal{A}$-modules are 
    cofibrant.

  \item 
    \label{enum:dg-A-cofib-retract-semifree}
    Every cofibrant
    dg $\mathcal{A}$-module
    is a retract of a semi-free dg $\mathcal{A}$-module.
  \end{enumerate}
\end{lemma}

\begin{proof}
  \ref{enum:dg-A-semifree-Icell-cofib}
  This follows from Lemma~\ref{l:extension-of-cofibrant-dgAmodules}
  since free dg $\mathcal{A}$-modules are cofibrant (and cofibrations
  are closed under transfinite compositions). 


  \ref{enum:dg-A-cofib-retract-semifree}
  Let $C \in C(\mathcal{A})$ be cofibrant.
  The obvious variation of \cite[Lemma~13.3, cf.\ 14.8]{drinfeld-dg-quotients}
  shows that there is a surjective quasi-isomorphism (= trivial
  fibration)
  $f:F \ra C$ where $F$ is a semi-free dg $\mathcal{A}$-module.
  Since $C$ is cofibrant, $\id_C$ factors through the trivial fibration
  $f,$ and hence $C$ is a retract of $F.$
\end{proof}

\subsection{Tensor product and flatness}
\label{sec:tens-prod-flatn-1}

Let $\mathcal{A},$ $\mathcal{B}$ and $\mathcal{C}$ be dg
$K$-categories, and let 
$X=\leftidx{_\mathcal{B}}{X}{_\mathcal{A}}$ be a dg $\mathcal{A}
\otimes_K \mathcal{B}^\opp$-module
and 
$Y=\leftidx{_\mathcal{C}}{X}{_\mathcal{B}}$ a dg $\mathcal{B}
\otimes_K \mathcal{C}^\opp$-module. Their tensor product is the 
dg $\mathcal{A} \otimes_K \mathcal{C}^\opp$-module 
$Y \otimes_\mathcal{B} X$ defined in the usual way (cf.\
\cite[6.1]{Keller-deriving-dg-cat} or
\cite[14.3]{drinfeld-dg-quotients}). This construction yields a dg
$K$-functor
\begin{equation*}
  (? \otimes_{\mathcal{B}} ?):
  \calMod(\mathcal{B} \otimes \mathcal{C}^\opp) \otimes_K
  \calMod(\mathcal{A} \otimes \mathcal{B}^\opp) \ra
  \calMod(\mathcal{A} \otimes \mathcal{C}^\opp).
\end{equation*}

A dg $\mathcal{A}$-module $M$ is \define{$\mathcal{A}$-homotopically-flat} (abbreviated
\define{$\mathcal{A}$-h-flat}) (cf.\
\cite[14.7]{drinfeld-dg-quotients})
if $M \otimes_\mathcal{A} X$ is acyclic whenever $X$ is an acyclic
dg $\mathcal{A}^\opp$-module.
An equivalent condition is that $(M \otimes_\mathcal{A} ?)$ preserves
quasi-isomorphisms, i.\,e.\ whenever $f: X \ra Y$ in $C(\mathcal{A}^\opp)$
is a quasi-isomorphism, then $\id_M \otimes_\mathcal{A} f: M
\otimes_\mathcal{A} X \ra M \otimes_\mathcal{A} Y$ is a
quasi-isomorphism.

Note that this also defines the notion of $K$-h-flatness by
considering $K$ as a dg $K$-category with one object.

Note that a dg $\mathcal{A}^\opp$-module $M$ is
$\mathcal{A}^\opp$-h-flat if and only if
$(? \otimes_\mathcal{A} M)$ preserves acyclics.

If $M$ is a dg $\mathcal{A} \otimes_K \mathcal{B}^\opp$-module we say that
it is $\mathcal{A}$-h-flat if each
$M(?,B)$ (for $B \in \mathcal{B}^\opp$)
is $\mathcal{A}$-flat, i.\,e. if $(M \otimes_\mathcal{A} ?)$ maps acyclic dg
$\mathcal{A}^\opp$-modules to acyclic dg $\mathcal{B}^\opp$-modules,
and we define $\mathcal{B}^\opp$-h-flatness similarly.

For example, all $\Yoneda{A}$ (for $A \in \mathcal{A}$) are
$\mathcal{A}$-h-flat, since $\Yoneda{A} \otimes_\mathcal{A} X = X(A)$
canonically.
More generally, the following is true.

\begin{lemma}
  \label{l:cofibrant-dg-A-module-is-A-h-flat}
  Any cofibrant dg $\mathcal{A}$-module is $\mathcal{A}$-h-flat.  
  In particular any cofibrant dg $K$-module is $K$-h-flat. 
\end{lemma}

\begin{proof}
  A cofibrant dg $\mathcal{A}$-module is a summand of a semi-free dg
  $\mathcal{A}$-module (Lemma~\ref{l:cofibrant-dg-A-modules}), and
  semi-free modules are obviously $\mathcal{A}$-h-flat
  (cf.\ \cite[14.8]{drinfeld-dg-quotients}).
\end{proof}

\begin{lemma}
  \label{l:qisos-between-Khflats-remain-qisos}
  Let $f: M \ra N$ be a quasi-isomorphism in $C(K)$ and assume that $M$
  and $N$ are $K$-h-flat. If $L$ is any dg $K$-module, then
  $f \otimes_K \id_L: M \otimes_K L \ra N \otimes_K L$ is a quasi-isomorphism.
\end{lemma}

\begin{proof}
  Let $r: L' \ra L$ be a $K$-h-flat resolution of $L$ (e.\,g.\ a
  cofibrant resolution, cf.\
  Lemma~\ref{l:cofibrant-dg-A-module-is-A-h-flat}).
  In the commutative diagram
  \begin{equation*}
    \xymatrix@C1.6cm{
      {M \otimes_K L'} 
      \ar[r]^-{f\otimes_K \id_{L'}} 
      \ar[d]_-{\id_M \otimes_K r} &
      {N \otimes_K L'} 
      \ar[d]^-{\id_N \otimes_K r} \\
      {M \otimes_K L} 
      \ar[r]^-{f\otimes_K \id_{L}} &
      {N \otimes_K L,} 
    }
  \end{equation*}
  $f\otimes_K \id_L$ is a quasi-isomorphism since the other three
  morphisms are quasi-isomorphisms. 
\end{proof}

\begin{proposition}
  \label{p:cofibrant-bimodule-remains-cofibrant-if-insert-cofib-cat}
  Let $\mathcal{B}$ be a dg $K$-category with cofibrant morphism
  spaces.
  \begin{enumerate}
  \item 
    \label{enum:cofibrant-module-remains-cofibrant-if-evaluate}
    If $M$ is a cofibrant dg $\mathcal{B}$-module, then $M(B)$ is
    a cofibrant dg $K$-module, for any $B \in \mathcal{B},$ and is in
    particular $K$-h-flat.
  \end{enumerate}
  Let $\mathcal{A}$ be a dg $K$-category.
  \begin{enumerate}[resume]
  \item 
    \label{enum:cofibrant-bimodule-remains-cofibrant-if-partially-evaluate}
    Let $\mathcal{R}=\mathcal{A}\otimes_K \mathcal{B}^\opp$ or
    $\mathcal{R}=\mathcal{A}\otimes_K \mathcal{B}.$
    Let $X$ be a cofibrant 
    dg $\mathcal{R}$-module. Then $X^B:=X(?,B)$ is
    a cofibrant dg $\mathcal{A}$-module, for every $B \in
    \mathcal{B}.$
    In particular any cofibrant dg $\mathcal{A} \otimes_K
    \mathcal{B}^\opp$-module is $\mathcal{A}$-h-flat.
  \end{enumerate}
  Assume that  $\mathcal{A}$ has cofibrant morphism spaces.
  \begin{enumerate}[resume]
  \item 
    \label{enum:cofibrant-bimodule-remains-cofibrant-if-evaluate}
    If $X$ is a cofibrant dg $\mathcal{A}
    \otimes_K\mathcal{B}^\opp$-module, then $X(A,B)$ is a cofibrant dg
    $K$-module, for all $A \in \mathcal{A}$ and $B \in \mathcal{B},$
    and is in particular $K$-h-flat.
  \end{enumerate}
\end{proposition}

\begin{proof}
  The assertions concerning h-flatness follow from 
  Lemma~\ref{l:cofibrant-dg-A-module-is-A-h-flat}.

  \ref{enum:cofibrant-bimodule-remains-cofibrant-if-partially-evaluate}
  implies
  \ref{enum:cofibrant-module-remains-cofibrant-if-evaluate}:
  Take $\mathcal{A}=K$ and $\mathcal{R}=\mathcal{A} \otimes_K
  \mathcal{B}=\mathcal{B}.$

  \ref{enum:cofibrant-bimodule-remains-cofibrant-if-partially-evaluate}
  and \ref{enum:cofibrant-module-remains-cofibrant-if-evaluate} imply
  \ref{enum:cofibrant-bimodule-remains-cofibrant-if-evaluate}:
  Obvious.

  We need to prove
  \ref{enum:cofibrant-bimodule-remains-cofibrant-if-partially-evaluate}.
  We only consider the case 
  $\mathcal{R}=\mathcal{A}\otimes_K \mathcal{B}^\opp,$ the case
  $\mathcal{R}=\mathcal{A}\otimes_K \mathcal{B}$ follows by
  considering $\mathcal{B}^\opp$ instead of $\mathcal{B}.$

  Any cofibrant dg $\mathcal{A}\otimes_K \mathcal{B}^\opp$-module is a
  summand of a semi-free dg 
  $\mathcal{A} \otimes_K \mathcal{B}^\opp$-module by
  Lemma~\ref{l:cofibrant-dg-A-modules}, and cofibrant dg
  $\mathcal{A}$-modules are stable under summands.
  Hence we can assume that $X$ is a semi-free dg $\mathcal{A}\otimes_K
  \mathcal{B}^\opp$-module. Then there is an increasing filtration $0=F_0 \subset
  F_1 \subset \dots$ of $X$ such that $X= \bigcup_{i \in \DN} F_i$ and
  each quotient $F_{i+1}/F_i$ is a free dg $\mathcal{A} \otimes_K
  \mathcal{B}^\opp$-module. 

  Let $B \in \mathcal{B}.$ Evaluation at $B$ is exact and hence yields an
  increasing filtration $0 = F_0^B \subset F_1^B \subset \dots$ of $X^B$ such that
  $X^B=\bigcup_{i \in \DN}F_i^B.$ 
  An obvious induction using
  Lemma~\ref{l:extension-of-cofibrant-dgAmodules}
  and the fact that
  cofibrations are closed under transfinite compositions
  shows that it is sufficient to show that all 
  $F_{i+1}^B/F_i^B$ are cofibrant dg $\mathcal{A}$-modules.

  Fix $i \in \DN.$
  Then $F_{i+1}/F_i$ is isomorphic to a coproduct of shifts of objects
  $\Yoneda{R}$ for $R \in \mathcal{R}.$
  Evaluation at $B$ yields that
  $F_{i+1}^B/F_i^B$ is isomorphic to a coproduct of shifts of objects
  $\Yoneda{R}^B$ for $R \in \mathcal{R}.$

  Hence it is sufficient to show that $\Yoneda{(A_0,B_0)}^B$ is
  cofibrant for any $(A_0, B_0) \in \mathcal{A} \otimes_K
  \mathcal{B}^\opp.$
  Since $\mathcal{B}(B_0, B)$ is a cofibrant dg $K$-module, it is a
  direct summand of a semi-free dg $K$-module $G.$
  We can write $G$ as the union/colimit of a sequence $0=G_0
  \subset G_1 \subset \dots$ of dg $K$-submodules
  such that all quotients $G_{j+1}/G_j$ are free dg $K$-modules.
  Fix isomorphisms
  $G_{j+1}/G_j \cong \bigoplus_{l \in L^{(j+1)}}[n^{(j+1)}_l]K$ for
  suitable index sets $L^{(j+1)}.$

  For $A \in \mathcal{A}$ we have
  \begin{equation*}
    \Yoneda{(A_0,B_0)}^B(A)= (\mathcal{A}\otimes_K
    \mathcal{B}^\opp)((A,B), (A_0,B_0))
    = \mathcal{A}(A,A_0) \otimes_K \mathcal{B}(B_0, B)
  \end{equation*}
  which is a direct summand of $\mathcal{A}(A, A_0) \otimes_K G.$
  Hence $\Yoneda{(A_0,B_0)}^B$ is a direct summand of
  $\Yoneda{A_0} \otimes_K G = \mathcal{A}(?, A_0) \otimes_K G$
  and it is enough to show that the
  latter is a cofibrant dg $\mathcal{A}$-module.

  Since 
  $G_j \subset G_{j+1}$ splits in 
  graded $K$-modules, $\mathcal{A}(?,A_0)\otimes_K G$ is the
  union/colimit of its submodules
  \begin{equation*}
    0 = \mathcal{A}(?,A_0) \otimes_K G_0 \subset
    \mathcal{A}(?,A_0) \otimes_K G_1 \subset \dots
  \end{equation*}
  and the subquotients are isomorphic to 
  \begin{equation*}
    \mathcal{A}(?,A_0) \otimes_K (G_{j+1}/G_j)
    \cong \mathcal{A}(?,A_0) \otimes_K \bigoplus_{l \in L^{(j+1)}}
      [n^{(j+1)}_l] K =
    \bigoplus_{l \in L^{(j+1)}} [n^{(j+1)}_l] \mathcal{A}(?,A_0);
  \end{equation*}
  these claims can be checked by plugging in $A \in \mathcal{A}.$
  
  Since coproducts of shifts of $\mathcal{A}(?, A_0)=\Yoneda{A_0}$ are
  cofibrant,
  an obvious induction using 
  Lemma~\ref{l:extension-of-cofibrant-dgAmodules}
  (and a (countable) transfinite composition)
  shows that
  $\Yoneda{A_0} \otimes_K G$ is a
  cofibrant dg $\mathcal{A}$-module.
\end{proof}

\subsection{Standard functors and constructions}
\label{sec:stand-funct-constr}

We discuss some standard constructions and refer the reader to 
\cite[Section 6]{Keller-deriving-dg-cat} and
\cite{keller-on-dg-categories-ICM} for more details.

Let $\mathcal{A},$ $\mathcal{B}$ be dg $K$-categories, and let 
$X=\leftidx{_\mathcal{B}}{X}{_\mathcal{A}}$
be an $\mathcal{A} \otimes_K \mathcal{B}^\opp$-module.

\subsubsection{Hom and tensor}
\label{sec:hom-tensor}
This datum gives rise to a pair $(T_X, H_X)$ of adjoint dg $K$-functors  
\begin{equation}
  \label{eq:adj-tensor-hom-calModlevel}
  \xymatrix{
    {\calMod(\mathcal{B})}
    \ar@/^/[r]^{T_X}
    &
    {\calMod(\mathcal{A})}
    \ar@/^/[l]^{H_X}
  }
\end{equation}
where $T_X := (? \otimes_\mathcal{B} X),$ and
$H_X(M)$ is defined by
\begin{equation}
  \label{eq:def-HX}
  (H_X(M))(B):= (\calMod(\mathcal{A}))(X(?, B), M)
\end{equation}
for $B \in \mathcal{B}$ with the obvious action morphisms.

\subsubsection{Dual bimodule}
\label{sec:dual-module}
Following \cite[6.2]{Keller-deriving-dg-cat} we define
the dg $\mathcal{B} \otimes_K \mathcal{A}^\opp$-module $X^\perp=
\leftidx{_\mathcal{A}}{(X^\perp)}{_\mathcal{B}}$ by
\begin{equation}
  \label{eq:def-dual-bimodule}
  X^\perp(B,A):= (\calMod(\mathcal{A}))(X(?,B),\Yoneda{A})
\end{equation}
for $(B,A) \in \mathcal{B} \otimes_K \mathcal{A}^\opp$ with obvious
action morphisms. Observe that there is a canonical
transformation
$\tau: T_{X^\perp} = (? \otimes_\mathcal{A} X^\perp) \ra H_X$ of dg
$K$-functors.

\subsubsection{Left derived tensor product}
\label{sec:left-derived-tensor}

We define the triangulated
functor 
$LT_{X}:= (?
  \otimes_\mathcal{B}^L X)$ to be the composition 
\begin{equation}
  \label{eq:def-derived-tensor-product}
  D(\mathcal{B}) \xsira{p} \mathcal{H}(\mathcal{B})_{cf} \xra{T_X}
  \mathcal{H}(\mathcal{A}) \ra D(\mathcal{A}),
\end{equation}
where the first arrow is the equivalence 
\eqref{eq:homotopy-cat-cofibrants-isom-derived-cat-quasi-inverse}.
Note that $LT_X$ preserves all (small) coproducts.
We call any functor $D(\mathcal{B}) \ra D(\mathcal{A})$ of this form a
tensor functor. A tensor equivalence is a tensor functor that is an equivalence.

\subsubsection{Restriction and extension of scalars}
\label{sec:restr-prod}

Let $F: \mathcal{B} \ra \mathcal{A}$ be a dg $K$-functor of
dg $K$-categories. 
Taking $X=\mathcal{A}$ in
\eqref{eq:adj-tensor-hom-calModlevel} defines the extension of scalars
dg $K$-functor
$F^*:=\pro_\mathcal{B}^\mathcal{A}:= T_{\mathcal{A}}=(? \otimes_\mathcal{B} \mathcal{A})$
and its right adjoint $H_{\mathcal{A}}$ which is canonically
isomorphic to the obvious restriction of scalars dg $K$-functor
$F_*:=\res_\mathcal{B}^\mathcal{A}.$
Obviously $\res_\mathcal{B}^\mathcal{A}$ preserves acyclics and
descends to a triangulated functor
\begin{equation}
  \label{eq:restriction-derived}
  \res_\mathcal{B}^\mathcal{A}:D(\mathcal{A}) \ra
  D(\mathcal{B}).  
\end{equation}
It has $L\pro_\mathcal{B}^\mathcal{A}:=
LT_\mathcal{A}$ as a left adjoint, and this functor preserves compact
objects since its right adjoint commutes with coproducts.
If $F$ is a quasi-equivalence (as defined
below, see \ref{enum:quasi-equi-qiso-homs} and
\ref{enum:quasi-equi-H0-essentially-epi}),
then \eqref{eq:restriction-derived}
is an equivalence and induces an equivalence 
\begin{equation}
  \label{eq:restriction-equi-per-along-quequi}
  \res_\mathcal{B}^\mathcal{A}:\per(\mathcal{A}) \sira \per(\mathcal{B}).    
\end{equation}
This essentially follows from the results explained around
\eqref{eq:perA-equal-compactDA} and the fact that 
$L\pro_\mathcal{B}^\mathcal{A}$ commutes with coproducts and maps
$\Yoneda{B}$ to (an object isomorphic to) $\Yoneda{F(B)}.$

\subsection{Model structure on the category of dg
  \texorpdfstring{$K$}{K}-categories} 
\label{sec:model-struct-categ}

G.~Tabuada defines in
\cite{tabuada-model-structure-on-cat-of-dg-cats} a model structure
on the category 
of small dg categories.
We discuss a small generalization of
this result.

Note that any dg $K$-functor $F: \mathcal{A} \ra \mathcal{B}$ induces
a functor $[F]: [\mathcal{A}] \ra [\mathcal{B}]$ on homotopy categories. 
A dg $K$-functor $F:\mathcal{A} \ra \mathcal{B}$ is a
\define{quasi-equivalence} if 
\begin{enumerate}[label=(qe{\arabic*})]
\item 
  \label{enum:quasi-equi-qiso-homs}
  for all objects $a_1, a_2 \in \mathcal{A},$ the morphism
  $F: \mathcal{A}(a_1, a_2) \ra \mathcal{B}(Fa_1, Fa_2)$
  is a quasi-isomorphism, and
\item 
  \label{enum:quasi-equi-H0-essentially-epi}
  the functor $[F]: [\mathcal{A}] \ra[\mathcal{B}]$ is
  essentially surjective (i.\,e.\ surjective on isoclasses of
  objects).
\end{enumerate}
If \ref{enum:quasi-equi-qiso-homs} holds, then
\ref{enum:quasi-equi-H0-essentially-epi} is equivalent to the condition
that $[F]: [\mathcal{A}] \ra[\mathcal{B}]$ is an equivalence.

Denote by $\dgcat_K$ the category of small
dg $K$-categories: Its objects are small dg $K$-categories, and its morphisms
are dg $K$-functors.

Let $\Fib$ be the class of all morphisms $G:\mathcal{X} \ra
\mathcal{Y}$ in $\dgcat_K$ that 
satisfy the following two conditions:
\begin{enumerate}[label=(Fib{\arabic*})]
\item 
  \label{enum:fib-dgcat-epi-on-homs}
  for all objects $x_1, x_2$ of $\mathcal{X},$ the morphism
  $\mathcal{X}(x_1,x_2) \ra \mathcal{Y}(Gx_1, Gx_2)$ is surjective, and
\item 
  \label{enum:fib-dgcat-lifting-isos-in-homotopy-cat}
  for all objects $x_1$ in $\mathcal{X}$ and each isomorphism $v: Gx_1
  \ra y$ in $[\mathcal{Y}]$ there is (an object $x_2$ in
  $\mathcal{X}$) and an isomorphism $u: x_1 \ra x_2$ in
  $[\mathcal{X}]$ such that $[G](u)=v$ (so $Gx_2=y$).
\end{enumerate}

\begin{theorem}
  [{cf.~\cite{tabuada-model-structure-on-cat-of-dg-cats}}]
  \label{t:dgcatK-cofib-gen-model-cat}
  The category $\dgcat_K$ can be equipped with the structure of a
  cofibrantly generated model category whose 
  weak equivalences are the quasi-equivalences and whose fibrations
  are $\Fib.$
\end{theorem}

In the following, whenever we use model-theoretic terminology, we
assume that all dg $K$-categories involved are small.

\begin{proof}
  The proof of
  \cite[Thm.~2.1]{tabuada-model-structure-on-cat-of-dg-cats}
  generalizes in the obvious way to this setting.
\end{proof}

Note that any object of $\dgcat_K$ is fibrant.
Moreover, the proof describes the class $\trFib$ of trivial
fibrations (= morphisms that are weak equivalences and fibrations) as
follows.
A morphism
$G:\mathcal{X} \ra \mathcal{Y}$ in $\dgcat_K$ is a trivial fibration
if and only if it is surjective in the following two senses:
\begin{enumerate}[label=(trFib{\arabic*})]
\item 
  \label{enum:surj-on-homs-surj-qiso}
  for all objects $x_1, x_2 \in \mathcal{X},$ the morphism
  $G: \mathcal{X}(x_1, x_2) \ra \mathcal{Y}(Gx_1, Gx_2)$
  is a surjective quasi-isomorphism, and
\item
  \label{enum:surj-epi-on-sets-of-objects}
  $G$ induces a surjection from the set of objects of
  $\mathcal{X}$ onto the set of objects of $\mathcal{Y}.$
\end{enumerate}

A dg $K$-category $\mathcal{F}$ is \define{discrete} if $\mathcal{F}(X,X)=K$
for all $X \in \mathcal{F}$ and $\mathcal{F}(X,Y)=0$ for all $X, Y \in
\mathcal{F}$ with $X\not=Y.$
A dg $K$-category $\mathcal{F}$ is \define{semi-free} (cf.\
\cite[13.4]{drinfeld-dg-quotients}) if it can be represented 
as the union of an increasing sequence of dg $K$-subcategories 
$\mathcal{F}_i$ (where $i \in \DN$)
such that $\mathcal{F}_0$ is a discrete dg $K$-category
and each $\mathcal{F}_i$ (for $i > 0$) as a graded $K$-category (= category
enriched over the symmetric monoidal category of graded $K$-modules) is freely 
generated over $\mathcal{F}_{i-1}$ by a family of homogeneous
morphisms $f_\alpha$ whose differentials $df_\alpha$ are morphisms in
$\mathcal{F}_{i-1}$ (in particular all $\mathcal{F}_i$ have the same objects).

\begin{lemma}
  \label{l:cofibrant-dg-K-categories}
  \rule{1mm}{0mm}
  \begin{enumerate}
  \item 
    \label{enum:dg-K-semifree-Icell-cofib}
    All semi-free dg $K$-categories are 
    cofibrant.
  \item 
    \label{enum:dg-K-cofib-retract-semifree}
    Every cofibrant dg $K$-category 
    is a retract of a semi-free dg $K$-category.
  \end{enumerate}
\end{lemma}

\begin{proof}
  \ref{enum:dg-K-semifree-Icell-cofib} 
  A semi-free dg $K$-category
  has the required lifting property with respect to trivial
  fibrations, as follows from (the obvious variation of)
  \cite[13.6]{drinfeld-dg-quotients}.

  \ref{enum:dg-K-cofib-retract-semifree}
  Let $\mathcal{C}$ be a cofibrant dg $K$-category.
  By the dg $K$-version of
  \cite[Lemma~13.5]{drinfeld-dg-quotients}
  there exists a semi-free dg $K$-category $\mathcal{F}$ 
  and a trivial fibration $\mathcal{F} \ra \mathcal{C}$ (which can be
  assumed to be the identity on objects). Since $\mathcal{C}$ is
  cofibrant $\id_\mathcal{C}$ factors as $\mathcal{C} \ra \mathcal{F} \ra
  \mathcal{C}.$
\end{proof}

\begin{definition}
  \label{d:K-flat-dgK-cat}
  A dg $K$-category $\mathcal{A}$ is called \define{$K$-h-flat} if all
  morphisms spaces $\mathcal{A}(A,A'),$ for $A,$ $A' \in \mathcal{A},$
  are $K$-h-flat.
\end{definition}

\begin{lemma}
  \label{l:cofibrant-dg-K-cat-has-cofibrant-morphism-spaces}
  Cofibrant dg $K$-categories have cofibrant morphism spaces. In
  particular, they are $K$-h-flat, by Lemma~\ref{l:cofibrant-dg-A-module-is-A-h-flat}.
\end{lemma}

In the proof we will use the following obvious criterion for semi-freeness
(generalized from \cite[13.1.(2)]{drinfeld-dg-quotients} to our setting).
Let $M$ be a dg $K$-module. 
Then $M$ is semi-free if (and only if) the following
condition is satisfied:
$M$ has a homogeneous $K$-basis $B$ (as a graded $K$-module) with
the following property: for a subset $S \subset B$ let $\delta(S)$
be the smallest subset $T \subset B$ such that $d_M(S)$ is contained
in the $K$-linear span of $T$; then for every $b \in B$ there is an
$n \in \DN$ such that $\delta^n(\{b\})=\emptyset.$

\begin{proof}
  Obviously retracts of dg $K$-categories with cofibrant morphism
  spaces have cofibrant morphism spaces. By
  Lemma~\ref{l:cofibrant-dg-K-categories}
  it is therefore sufficient to show that any semi-free dg
  $K$-category $\mathcal{F}$ has cofibrant morphism spaces.
  We even show that all $\mathcal{F}(A,A')$ are semi-free dg
  $K$-modules (for $A,$ $A' \in \mathcal{F}$). 

  Choose an exhausting increasing filtration 
  $(\mathcal{F}_i)_{i \in \DN}$ of $\mathcal{F}$
  such that $\mathcal{F}_0$ is a discrete dg $K$-category and
  $\mathcal{F}_{i+1}=\mathcal{F}_i\langle P_i\rangle$ 
  (as a graded $K$-category) where $P_i$ is a set of homogeneous
  arrows (= morphisms) such that the morphism $dp$ is a
  cocycle in $\mathcal{F}_i$ for each $p \in P_i.$

  This implies that the underlying graded $K$-category of 
  $\mathcal{F}$ is freely generated by the arrows $\bigcup_{i \in \DN}
  P_i.$ Hence the set $B(A,A')$
  of all paths in these arrows starting at $A$ and ending at $A'$ is a
  homogeneous
  $K$-basis of $\mathcal{F}(A,A').$ We claim that this basis
  satisfies the above criterion ensuring semi-freeness.

  For $i \in \DN$ denote the set of all
  paths in the arrows $\bigcup_{s < i} P_s$ by $B_i,$ and by
  $B_i(X, X')$ the subset of paths that start at $X$ and end
  at $X'$ (for $X,$ $X' \in \mathcal{F}_i$), so
  $B_i = \coprod_{X, X' \in \mathcal{F}_{j}} B_{j}(X, X').$
  Then $B_i(X, X')$ is is a
  homogeneous $K$-basis of 
  $\mathcal{F}_i(X, X').$ 
  We fix $i$ and can assume by induction that
  $B_i(X,X')$ satisfies the above criterion for all $X, X' \in
  \mathcal{F}.$ 
  We need to show that $B_{i+1}(A, A')$ satisfies this criterion.

  Let $j \in \DN.$ Note that $B_{j}$ is stable by composition (=
  concatenation of paths). 
  If $S$ and $T$ are subsets of $B_{j},$
  then
  obviously $\delta(S \cup T) =
  \delta(S) \cup \delta(T).$ Moreover the Leibniz
  rule $d(st)= d(s)t \pm sd(t)$ shows that $\delta(ST) \subset
  \delta(S)T \cup S\delta(T),$ and hence
  $\delta^n(ST) \subset \bigcup_{i=0}^n
  \delta^{n-i}(S)\delta^{i}(T).$
  
  For any $(a:X \ra Y) \in P_i$ we know that $d(a) \subset
  \mathcal{F}_i(X,Y),$ hence $\delta(a) \subset T$ for some finite
  subset of $B_i.$ By induction we know that $\delta^n(T)=\emptyset$
  for $n$ big enough and hence $\delta^{n+1}(\{a\})=\emptyset.$
    
  Any element $b$ of $B_{i+1}$ is a finite product of elements $a$ (in
  $B_i$ or in $P_i$) for which we already know that
  $\delta^n(\{a\})=\emptyset$ for $n\gg 0.$ Induction over the number
  of factors and the above rule then show that
  $\delta^n(\{b\})=\emptyset$ for $n\gg 0.$
  This proves that all $\mathcal{F}(A, A')$ are semi-free dg
  $K$-modules.
\end{proof}

\begin{lemma}
  \label{l:tensor-with-K-h-flat-preserves-qequis}
  Let $G:\mathcal{R} \ra \mathcal{S}$ be a quasi-equivalence of dg
  $K$-categories and let $\mathcal{F}$ be a $K$-h-flat dg $K$-category
  (e.\,g.\ a cofibrant dg $K$-category,
  cf.\ Lemma~\ref{l:cofibrant-dg-K-cat-has-cofibrant-morphism-spaces}).
  Then $G \otimes_K \id_{\mathcal{F}}: \mathcal{R} \otimes_K
  \mathcal{F} \ra \mathcal{S} \otimes_K \mathcal{F}$ is a
  quasi-equivalence.
\end{lemma}

\begin{proof}
  This is obvious.
\end{proof}

\section{Smoothness of dg \texorpdfstring{$K$}{K}-categories}
\label{sec:smoothn-dg-K-categories}

We generalize results of \cite[Section~3]{lunts-categorical-resolution}.
Recall that a $K$-h-flat resolution in $\dgcat_K$ is a trivial fibration
$\mathcal{A}' \ra \mathcal{A}$ such that $\mathcal{A}'$ is $K$-h-flat,
and that a cofibrant resolution is a trivial fibration
$\tildew{\mathcal{A}} \ra \mathcal{A}$ such that
$\tildew{\mathcal{A}}$ is cofibrant. 
Any object of $\dgcat_K$ has a cofibrant resolution, and cofibrant
resolutions are $K$-h-flat by 
Lemma~\ref{l:cofibrant-dg-K-cat-has-cofibrant-morphism-spaces}.

\begin{lemma}
  \label{l:lift-cofib-reso-to-flat-reso}
  Let $f:\tildew{\mathcal{A}} \ra \mathcal{A}$ be a quasi-equivalence
  with cofibrant $\tildew{\mathcal{A}}$ 
  and $g:\mathcal{A}' \ra \mathcal{A}$ a trivial fibration. 
  Then there is a quasi-equivalence $h: \tildew{\mathcal{A}} \ra
  \mathcal{A}'$ such that $f=gh.$
  \begin{equation*}
    \xymatrix{
      &
      {\mathcal{A}'} \ar[d]^g \\
      {\tildew{\mathcal{A}}} \ar[r]^f \ar@{..>}[ru]^-{\exists h}&
      {\mathcal{A}}
    }
  \end{equation*}
  In particular, such a lift $h$ exists if $f$ is a
  cofibrant resolution and $g$ is a $K$-h-flat resolution.
\end{lemma}

\begin{proof}
  This follows directly from the definitions of a model category.
\end{proof}

If $N$ is 
a dg $\mathcal{A} \otimes_K \mathcal{B}^\opp$-module and $\mathcal{A}'
\ra \mathcal{A}$ and $\mathcal{B}' \ra \mathcal{B}$ are morphisms of
dg $K$-categories, we denote the restriction
of $N$ along
$\mathcal{A}' \otimes_K \mathcal{B}'^\opp \ra \mathcal{A} \otimes_K
\mathcal{B}^\opp$ by
$\leftidx{_{\mathcal{B}'}}{N}{_{\mathcal{A}'}}.$

\begin{lemma}
  \label{l:tfae-bimodule-smoothness}
  Let $\mathcal{A},$ $\mathcal{B}$ be dg $K$-categories, and let $N$
  be a dg $\mathcal{A} \otimes_K \mathcal{B}^\opp$-module. The
  following conditions are equivalent: 
  \begin{enumerate}[label=(Sw{\arabic*})] 
  \item 
    \label{enum:tfae-bimodule-smoothness-all-cofib-reso}
    $\leftidx{_{\tildew{\mathcal{B}}}}{N}{_{\tildew{\mathcal{A}}}} \in 
    \per(\tildew{\mathcal{A}} \otimes_K \tildew{\mathcal{B}}^\opp)$
    whenever
    $\tildew{\mathcal{A}} \ra \mathcal{A}$ and
    $\tildew{\mathcal{B}} \ra \mathcal{B}$ are cofibrant resolutions.
  \item 
    \label{enum:tfae-bimodule-smoothness-two-cofib-reso}
    $\leftidx{_{\tildew{\mathcal{B}}}}{N}{_{\tildew{\mathcal{A}}}} \in 
    \per(\tildew{\mathcal{A}} \otimes_K \tildew{\mathcal{B}}^\opp)$
    for a cofibrant resolution
    $\tildew{\mathcal{A}} \ra \mathcal{A}$ and a cofibrant resolution
    $\tildew{\mathcal{B}} \ra \mathcal{B}.$
  \item 
    \label{enum:tfae-bimodule-smoothness-right-cofib-reso}
    $\leftidx{_{{\mathcal{B}}}}{N}{_{\tildew{\mathcal{A}}}} \in 
    \per(\tildew{\mathcal{A}} \otimes_K {\mathcal{B}}^\opp)$
    for a cofibrant resolution
    $\tildew{\mathcal{A}} \ra \mathcal{A}.$
  \item 
    \label{enum:tfae-bimodule-smoothness-left-cofib-reso}
    $\leftidx{_{\tildew{\mathcal{B}}}}{N}{_{{\mathcal{A}}}} \in 
    \per({\mathcal{A}} \otimes_K \tildew{\mathcal{B}}^\opp)$
    for a cofibrant resolution
    $\tildew{\mathcal{B}} \ra \mathcal{B}.$
  \item 
    \label{enum:tfae-bimodule-smoothness-all-Khflat-reso}
    $\leftidx{_{\mathcal{B}'}}{N}{_{\mathcal{A}'}} \in 
    \per(\mathcal{A}' \otimes_K \mathcal{B}'^\opp)$
    whenever
    $\mathcal{A}' \ra \mathcal{A}$ and
    $\mathcal{B}' \ra \mathcal{B}$ are $K$-h-flat resolutions.
  \item 
    \label{enum:tfae-bimodule-smoothness-two-Khflat-reso}
    $\leftidx{_{\mathcal{B}'}}{N}{_{\mathcal{A}'}} \in 
    \per(\mathcal{A}' \otimes_K \mathcal{B}'^\opp)$
    for a $K$-h-flat resolution
    $\mathcal{A}' \ra \mathcal{A}$ and a $K$-h-flat resolution
    $\mathcal{B}' \ra \mathcal{B}.$
  \item 
    \label{enum:tfae-bimodule-smoothness-right-Khflat-reso}
    $\leftidx{_{{\mathcal{B}}}}{N}{_{\mathcal{A}'}} \in 
    \per(\mathcal{A}' \otimes_K {\mathcal{B}}^\opp)$
    for a $K$-h-flat resolution
    $\mathcal{A}' \ra \mathcal{A}.$
  \item 
    \label{enum:tfae-bimodule-smoothness-left-Khflat-reso}
    $\leftidx{_{\mathcal{B}'}}{N}{_{{\mathcal{A}}}} \in 
    \per({\mathcal{A}} \otimes_K \mathcal{B}'^\opp)$
    for a $K$-h-flat resolution
    $\mathcal{B}' \ra \mathcal{B}.$
  \end{enumerate}
\end{lemma}

\begin{proof}
  Obviously 
  \ref{enum:tfae-bimodule-smoothness-all-Khflat-reso} $\Rightarrow$ 
  \ref{enum:tfae-bimodule-smoothness-all-cofib-reso} $\Rightarrow$
  \ref{enum:tfae-bimodule-smoothness-two-cofib-reso} $\Rightarrow$
  \ref{enum:tfae-bimodule-smoothness-two-Khflat-reso}, and
  \ref{enum:tfae-bimodule-smoothness-right-cofib-reso} $\Rightarrow$
  \ref{enum:tfae-bimodule-smoothness-right-Khflat-reso}, and
  \ref{enum:tfae-bimodule-smoothness-left-cofib-reso} $\Rightarrow$
  \ref{enum:tfae-bimodule-smoothness-left-Khflat-reso}.

  Let $\tildew{a}: \tildew{\mathcal{A}} \ra \mathcal{A}$ be a cofibrant resolution
  and $a':\mathcal{A}' \ra \mathcal{A}$ a $K$-h-flat resolution.
  Then there is a quasi-equivalence $h: \tildew{\mathcal{A}} \ra \mathcal{A}'$
  such that $a' h=\tildew{a}$ (Lemma~\ref{l:lift-cofib-reso-to-flat-reso}).
  Let $\tildew{b}:\tildew{\mathcal{B}} \ra \mathcal{B}$ be a cofibrant
  resolution.
  In the commutative diagram
  \begin{equation*}
    \xymatrix{
      {\tildew{\mathcal{A}}\otimes_K \tildew{\mathcal{B}}^\opp} \ar[r]
      \ar[d] &
      {\mathcal{A}'\otimes_K \tildew{\mathcal{B}}^\opp} \ar[d] \\
      {\tildew{\mathcal{A}}\otimes_K {\mathcal{B}^\opp}} \ar[r] &
      {\mathcal{A}'\otimes_K {\mathcal{B}^\opp}} 
    }
  \end{equation*}
  all morphisms are quasi-equivalences, by Lemma
  \ref{l:tensor-with-K-h-flat-preserves-qequis} and the
  3-out-of-2-property.
  This shows
  (using equivalence \eqref{eq:restriction-equi-per-along-quequi})
  that 
  \ref{enum:tfae-bimodule-smoothness-all-cofib-reso}
  $\Leftrightarrow$
  \ref{enum:tfae-bimodule-smoothness-right-cofib-reso}
  $\Leftrightarrow$
  \ref{enum:tfae-bimodule-smoothness-right-Khflat-reso}.
  The proof of 
  \ref{enum:tfae-bimodule-smoothness-all-cofib-reso}
  $\Leftrightarrow$
  \ref{enum:tfae-bimodule-smoothness-left-cofib-reso}
  $\Leftrightarrow$
  \ref{enum:tfae-bimodule-smoothness-left-Khflat-reso} is similar.

  Let $b':\mathcal{B}' \ra \mathcal{B}$ be a $K$-h-flat
  resolution. There is
  a quasi-equivalence $l: \tildew{\mathcal{B}} \ra \mathcal{B}'$
  such that $b' l=\tildew{b}$ (Lemma~\ref{l:lift-cofib-reso-to-flat-reso}).
  In the commutative diagram
  \begin{equation*}
    \xymatrix{
      {\tildew{\mathcal{A}}\otimes_K \tildew{\mathcal{B}}^\opp} \ar[r]
      \ar[d] &
      {\mathcal{A}'\otimes_K \tildew{\mathcal{B}}^\opp} \ar[d] \\
      {\tildew{\mathcal{A}}\otimes_K {\mathcal{B}'^\opp}} \ar[r] &
      {\mathcal{A}'\otimes_K {\mathcal{B}'^\opp}} 
    }
  \end{equation*}
  all morphisms are quasi-equivalences, by Lemma
  \ref{l:tensor-with-K-h-flat-preserves-qequis}.
  This proves  
  \ref{enum:tfae-bimodule-smoothness-two-Khflat-reso}
  $\Leftrightarrow$
  \ref{enum:tfae-bimodule-smoothness-all-cofib-reso}
  as well as 
  \ref{enum:tfae-bimodule-smoothness-all-cofib-reso}
  $\Leftrightarrow$
  \ref{enum:tfae-bimodule-smoothness-all-Khflat-reso}.
\end{proof}

\begin{definition}
  \label{d:sweet}
  A dg $\mathcal{A}\otimes_K \mathcal{B}^\opp$-module $N$ is
  \define{\sweet{}} 
  (more precisely \define{$K$-\sweet{}})
  if it satisfies the equivalent conditions
  of Lemma~\ref{l:tfae-bimodule-smoothness}.
  (Note that one can rewrite all these conditions using 
  equality \eqref{eq:perA-equal-compactDA}.)
\end{definition}

\begin{lemma}
  \label{l:bimodule-smoothness-quasi-isomorphism-invariance}
  If $\mathcal{R} \ra \mathcal{A}$ and $\mathcal{S} \ra \mathcal{B}$
  are quasi-equivalences and 
  $N$ is a dg $\mathcal{A} \otimes_K \mathcal{B}^\opp$-module,
  then
  $N$ is $K$-\sweet{}
  if and only if 
  $\leftidx{_{\mathcal{S}}}{N}{_{{\mathcal{R}}}}$ is $K$-\sweet{}.
\end{lemma}

\begin{proof}
  Let $\mathcal{A}' \ra \mathcal{A}$ and
  $\mathcal{B}' \ra \mathcal{B}$ be $K$-h-flat resolutions, and let
  $\tildew{\mathcal{R}} \ra \mathcal{R}$  
  and $\tildew{\mathcal{S}} \ra \mathcal{S}$ be cofibrant resolutions.
  Lemma~\ref{l:lift-cofib-reso-to-flat-reso}
  shows that the quasi-equivalences 
  $\tildew{\mathcal{R}} \ra \mathcal{R} \ra \mathcal{A}$ 
  and
  $\tildew{\mathcal{S}} \ra \mathcal{S} \ra \mathcal{B}$ 
  lift to quasi-equivalences $\tildew{\mathcal{R}} \ra \mathcal{A}'$
  and $\tildew{\mathcal{S}} \ra \mathcal{B}'$ respectively. These lifts
  give rise to the commutative diagram 
  \begin{equation*}
    \xymatrix{
      {\tildew{\mathcal{R}}\otimes_K \tildew{\mathcal{S}}^\opp} \ar[r]
      \ar[d] &
      {\mathcal{A}'\otimes_K \tildew{\mathcal{S}}^\opp} \ar[d] \\
      {\tildew{\mathcal{R}}\otimes_K {\mathcal{B}'^\opp}} \ar[r] &
      {\mathcal{A}'\otimes_K {\mathcal{B}'^\opp}} 
    }
  \end{equation*}
  of quasi-equivalences
  (Lemma~\ref{l:tensor-with-K-h-flat-preserves-qequis}).
  Hence $\leftidx{_{\mathcal{B}'}}{N}{_{\mathcal{A}'}} \in 
  \per(\mathcal{A}' \otimes_K \mathcal{B}'^\opp)$ if and only if
  $\leftidx{_{\tildew{\mathcal{S}}}}{N}{_{\tildew{\mathcal{R}}}} \in 
  \per(\tildew{\mathcal{R}} \otimes_K
  \tildew{\mathcal{S}}^\opp).$ This proves the claim.
\end{proof}

By a dg $K$-algebra we mean a dg $K$-category with a unique object,
and we sometimes just refer to the endomorphism space of this object.

\begin{corollary}
  \label{c:bimodule-smoothness-quasi-isomorphism-invariance}
  Let $\mathcal{A}$ be a dg $K$-category and $B$ a dg $K$-algebra.
  Let $N$ be a dg $\mathcal{A} \otimes_K B^\opp$-module. Assume that
  the structure morphism $K \ra B$ is a quasi-isomorphism.
  Then $N$ is $K$-\sweet{} if and only if 
  $N_{\mathcal{A}}:=\res^{\mathcal{A} \otimes_K B^\opp}_{\mathcal{A}}(N) \in
  \per(\mathcal{A})$
  (restriction along $\mathcal{A} \sira \mathcal{A} \otimes_K K
  \ra \mathcal{A} \otimes_K B^\opp$).
\end{corollary}

\begin{proof}
  Apply the lemma to $\mathcal{R}=\mathcal{A}$ and $\mathcal{S}=K \ra
  \mathcal{B}=B.$
\end{proof}

Any ring $R$ can be viewed as an $R$-$R$-bimodule ("diagonal
bimodule"). Similarly, any dg $K$-category $\mathcal{A}$ gives rise to
the dg $\mathcal{A} \otimes_K \mathcal{A}^\opp$-module
($\mathcal{A}$-$\mathcal{A}$-bimodule) $\mathcal{A}$ whose action
morphisms 
\begin{align}
  \label{eq:diagonal-module}
  \mathcal{A}(A'',A''') \otimes_K  \mathcal{A}(A',A'') \otimes_K \mathcal{A}(A,A')
  & \ra \mathcal{A}(A,A'''),\\
  \notag
  f \otimes g \otimes h & \ra f.g.h:= f \comp g \comp h,
\end{align}
are just given by composition (where on the left we formally have to move
$\mathcal{A}(A'', A''')$ as $\mathcal{A}^\opp(A''', A'')$ to the
right).
We call this module the \define{diagonal bimodule}.

\begin{lemma}
  \label{l:tfae-smoothness}
  Let $\mathcal{A}$ be a dg $K$-category. The following conditions are equivalent:
  \begin{enumerate}[label=(Sm{\arabic*})]
  \item 
    \label{enum:tfae-smoothness-all-cofib-reso}
    $\tildew{\mathcal{A}} \in
    \per(\tildew{\mathcal{A}}\otimes_K \tildew{\mathcal{A}}^\opp)$ for every
    cofibrant resolution $\tildew{\mathcal{A}} \ra \mathcal{A}.$
  \item 
    \label{enum:tfae-smoothness-one-cofib-reso}
    $\tildew{\mathcal{A}} \in
    \per(\tildew{\mathcal{A}}\otimes_K \tildew{\mathcal{A}}^\opp)$ for a
    cofibrant resolution $\tildew{\mathcal{A}} \ra \mathcal{A}.$
  \item 
    \label{enum:tfae-smoothness-all-Khflat-reso}
    $\mathcal{A}' \in
    \per(\mathcal{A}'\otimes_K \mathcal{A}'^\opp)$ for every
    $K$-h-flat resolution $\mathcal{A}' \ra \mathcal{A}.$
  \item 
    \label{enum:tfae-smoothness-one-Khflat-reso}
    $\mathcal{A}' \in
    \per(\mathcal{A}'\otimes_K \mathcal{A}'^\opp)$ for a
    $K$-h-flat resolution $\mathcal{A}' \ra \mathcal{A}.$
  \item 
    \label{enum:tfae-smoothness-sweetness}
    The diagonal bimodule $\mathcal{A}$ is $K$-\sweet{}.
  \end{enumerate}
\end{lemma}

\begin{proof}
  This is a consequence of Lemma~\ref{l:tfae-bimodule-smoothness}
  applied to the diagonal bimodule $\mathcal{A},$ and the following observation:
  Any morphism of dg $K$-categories $F:\mathcal{R} \ra \mathcal{A}$
  gives rise to a morphism from the diagonal bimodule $\mathcal{R}$
  to the restriction 
  $\leftidx{_\mathcal{R}}{\mathcal{A}}{_\mathcal{R}}$ of the diagonal
  bimodule $\mathcal{A},$ which is a quasi-isomorphism if and only
  if $F$ induces quasi-isomorphisms on morphism spaces.
\end{proof}

\begin{definition}
  [{cf.\ e.\,g.\ \cite[Def.~2.3]{toen-finitude-homotopique-propre-lisse}}]
  \label{d:smooth}
  A (small) dg $K$-category $\mathcal{A}$ is \define{smooth} 
  (more precisely \define{$K$-smooth}) 
  if it satisfies the equivalent conditions
  of Lemma~\ref{l:tfae-smoothness}.
  (Note that one can rewrite these conditions using 
  equality \eqref{eq:perA-equal-compactDA}.)
\end{definition}

\begin{remark}
  In practice conditions \ref{enum:tfae-smoothness-all-Khflat-reso} and
  \ref{enum:tfae-smoothness-one-Khflat-reso} are useful for (dis)proving smoothness.
  For example a $K$-h-flat dg $K$-category $\mathcal{A}$
  is smooth if and only if
  $\mathcal{A} \in \per(\mathcal{A}\otimes_K \mathcal{A}^\opp).$
  
  More concretely, if a dg $K$-algebra $A$ is free (or semi-free or
  cofibrant or $K$-h-flat) when 
  considered as a dg $K$-module, then $A$ is smooth if and only if $A
  \in \per(A \otimes_K \mathcal{A}).$ 
\end{remark}

\begin{remark}
  \label{rem:smoothness-and-classical-smoothness}
  If $\groundring$ is a field considered as a dg algebra concentrated
  in degree zero, then 
  any dg $\groundring$-module is cofibrant (since it is isomorphic to
  a coproduct of shifts of $\groundring$ and shifts of
  $\Cone(\id_{\groundring})$)
  and hence (Lemma~\ref{l:cofibrant-dg-A-module-is-A-h-flat}))
  $\groundring$-h-flat. In particular any dg
  $k$-category is $\groundring$-h-flat.
  Hence a dg $\groundring$-category $\mathcal{A}$ is smooth if and
  only if
  $\mathcal{A} \in \per(\mathcal{A}\otimes_\groundring \mathcal{A}^\opp).$

  This shows that our definition of smoothness generalizes the usual notion
  of smoothness over a field 
  (see e.\,g.\ \cite[Def.~3.1]{lunts-categorical-resolution}). 
\end{remark}

\begin{examples}
  \label{ex:C-and-CX-and-smoothness}
  Consider $\DC[X]$ as a dg $\DC$-algebra with $X$ of positive degree
  and differential zero. Then we have:
  \begin{enumerate}
  \item 
    \label{enum:CX-C-smooth}
    $\DC[X]$ is $\DC$-smooth.
  \item 
    \label{enum:CXmodXn-not-C-smooth}
    $\DC[X]/(X^n)$ is not $\DC$-smooth for $n \geq 2,$ cf.\ 
    Proposition~\ref{p:homologically-positive-dg-algebra-plus-conditions-not-smooth} below.
  \end{enumerate}
  Assume now that $X$ has positive even degree, so that $\DC[X]$ is
  graded commutative.
  \begin{enumerate}[resume]
  \item 
    \label{enum:CXmodXn-not-CX-smooth}
    $\DC[X]/(X^n)$ is not $\DC[X]$-smooth, for $n \geq 1,$ cf.\ 
    Proposition~\ref{p:smoothness-over-local-graded-finite-homological-dim-algebras}  
    below. Note that $\DC \in \per(\DC \otimes_{\DC[X]} \DC)=\per(\DC),$
    so $\DC[X]$-smoothness of $\DC$ cannot be checked naively (without
    a suitable resolution). 
  \end{enumerate}
\end{examples}

\begin{remark}
  \label{rem:naive-smooth-and-opposite}
  We claim that the opposite of a smooth dg $K$-category is smooth.
  This follows from the following observations.
  If $\tildew{\mathcal{R}} \ra \mathcal{R}$ is a cofibrant resolution,
  then $\tildew{\mathcal{R}}^\opp \ra \mathcal{R}^\opp$ is a cofibrant
  resolution.
  If $\mathcal{A}$ and $\mathcal{B}$ are dg $K$-categories, there is
  an obvious isomorphism $\mathcal{A} \otimes_K \mathcal{B} \sira
  \mathcal{B} \otimes_K \mathcal{A}$ of dg $K$-categories.
  By restriction it induces an
  isomorphism $D(\mathcal{B} \otimes_K 
  \mathcal{A}) \sira D(\mathcal{A} \otimes_K \mathcal{B})$ which of
  course preserves compact objects. For
  $\mathcal{B}=\mathcal{A}^{\opp}$ it sends the
  diagonal bimodule $\mathcal{A}$ to the diagonal 
  bimodule $\mathcal{A}^{\opp}.$
  These statements prove the claim.
\end{remark}

Recall that two dg $K$-categories are quasi-equivalent if
they can be connected by a zig-zag of quasi-equivalences.

\begin{lemma}
  [Invariance of smoothness under quasi-equivalence]
  \label{l:smoothness-and-quasi-equis}
  If $\mathcal{A} \ra \mathcal{B}$ is a quasi-equivalence, then
  $\mathcal{A}$ is smooth if and only if $\mathcal{B}$ is smooth.

  In particular, if two dg $K$-categories are quasi-equivalent, then
  they are either both smooth or both not smooth. 
\end{lemma}

\begin{proof}
  This follows from
  Lemma~\ref{l:bimodule-smoothness-quasi-isomorphism-invariance}
  (and the observation in the proof of Lemma~\ref{l:tfae-smoothness}).
\end{proof}

\subsection{Invariance of smoothness under dg Morita equivalence}
\label{sec:invariance-smoothness-dg-morita-equivalence}

\begin{definition}
  [{cf.~\cite[3.8]{keller-on-dg-categories-ICM}}]
  \label{d:dg-morita-equivalent}
  Two dg $K$-categories $\mathcal{A},$ $\mathcal{B}$ are 
  \define{dg Morita equivalent} if $D(\mathcal{A})$ and
  $D(\mathcal{B})$ are connected by a zig-zag of tensor equivalences
  (as defined after \eqref{eq:def-derived-tensor-product}).
\end{definition}

The aim of this section is to prove 
Theorem~\ref{t:smoothness-preserved-by-dg-Morita-equivalence} below
which says that smoothness is invariant under dg Moria equivalence.

\begin{lemma}
  \label{l:lifting-tensor-functor-along-triv-fib}
  Let $\mathcal{A}$ be a dg $K$-category and 
  let $b: \tildew{\mathcal{B}} \ra \mathcal{B}$ be a trivial fibration
  in $\dgcat_K.$
  Let $X=\leftidx{_\mathcal{B}}{X}{_\mathcal{A}}$ be a dg
  $\mathcal{A} \otimes_K \mathcal{B}^\opp$-module, and let
  $X'$ be its restriction to a dg
  $\mathcal{A} \otimes_K \tildew{\mathcal{B}}^\opp$-module.
  Then the diagram
  \begin{equation*}
    \xymatrix
    {
      {D(\tildew{\mathcal{B}})} 
      \ar[rr]^-{? \otimes_{\tildew{\mathcal{B}}}^L X'} &&
      {D(\mathcal{A})}\\
      {D({\mathcal{B}})} 
      \ar[u]^{\res^{\mathcal{B}}_{\tildew{\mathcal{B}}}}
      \ar[rru]_-{? \otimes_\mathcal{B}^L X}
    }
  \end{equation*}
  commutes up to a natural isomorphism.
\end{lemma}

\begin{proof}
  \textbf{Step 1:}
  (In this step it is sufficient to assume that $b$ is epimorphic on
  objects and morphisms). 
  We claim that the obvious evaluation morphism
  \begin{equation}
    \label{eq:evaluation-trivial-fib}
    \mathcal{B} \otimes_{\tildew{\mathcal{B}}} X' \ra X
  \end{equation}
  is an isomorphism of dg $\mathcal{A} \otimes_K
  \mathcal{B}^\opp$-modules.
  (This generalizes $R/I \otimes_R M = M$ for $M$ an $R/I$-module.)

  The evaluation of 
  $\mathcal{B} \otimes_{\tildew{\mathcal{B}}} X'$ at $(A, B) \in \mathcal{A} \otimes_K
  \mathcal{B}^\opp$ is (by the definition of the tensor product) the
  cokernel of the obvious morphism 
  \begin{equation*}
    \beta: \bigoplus_{\tildew{B}',\tildew{B}'' \in \tildew{\mathcal{B}}}
    \mathcal{B}(b\tildew{B}'', B) 
    \otimes_K
    \tildew{\mathcal{B}}(\tildew{B}',\tildew{B}'') \otimes_K X(A,b\tildew{B}') \ra 
    \bigoplus_{\tildew{B} \in \tildew{\mathcal{B}}}
    \mathcal{B}(b\tildew{B}, B) \otimes_K X(A,b \tildew{B}).
  \end{equation*}
  The evaluation map from the object on the right to
  $X(A,B)$ factors through the cokernel to a morphism
  \begin{equation*}
    e: (\mathcal{B} \otimes_{\tildew{\mathcal{B}}} X') (A,B) \ra X(A,B).
  \end{equation*}
  We need to show that $e$ is an isomorphism.

  Since $b$ is surjective on objects
  \ref{enum:surj-epi-on-sets-of-objects}, there is an object $\ol{B}
  \in \tildew{B}$ such that $b\ol{B}=B.$ 
  Let $s$ be the composition
  \begin{equation*}
    X(A,B) \ra \mathcal{B}(b\ol{B},B) \otimes_K X(A, b\ol{B})
    \hra 
    \bigoplus_{\tildew{B} \in \tildew{\mathcal{B}}}
    \mathcal{B}(b\tildew{B}, B) \otimes_K X(A,b \tildew{B})
    \sra 
    (\mathcal{B} \otimes_{\tildew{\mathcal{B}}} X') (A,B)
  \end{equation*}
  where the first map is defined by $x \mapsto \id_{B} \otimes x,$ the
  second map is the canonical inclusion and the third map is the
  projection onto the cokernel.
  Then obviously $es=\id.$ Hence $e$ is surjective and it is enough to
  show that $s$ is surjective. Let $\tildew{B} \in
  \tildew{\mathcal{B}}$ and $f \otimes x \in 
  \mathcal{B}(b\tildew{B}, B) \otimes_K X(A,b \tildew{B})$ be a pure
  tensor.
  Since $b$ is surjective on morphism spaces there is an element
  $\ol{f} \in \tildew{\mathcal{B}}(\tildew{B}, \ol{B})$ such that
  $b(\ol{f})=f.$
  Then $\beta$ maps the element
  \begin{equation*}
    \id_B \otimes \ol{f} \otimes x \in 
    \mathcal{B}(b\ol{B}, B) 
    \otimes_K
    \tildew{\mathcal{B}}(\tildew{B},\ol{B}) \otimes_K X(A,b\tildew{B})
  \end{equation*}
  to $f \otimes x - \id_B \otimes fx.$ This implies that $s$ is
  surjective and proves our claim that
  \eqref{eq:evaluation-trivial-fib} is an isomorphism.

  \textbf{Step 2:}
  If $Y$ is a dg $\tildew{\mathcal{B}} \otimes_K
  \mathcal{B}^\opp$-module, there is an obvious natural transformation
  \begin{equation*}
    (? \otimes^L_{\tildew{\mathcal{B}}} X') 
    \comp  (? \otimes^L_{\mathcal{B}} Y) 
    \ra  (? \otimes^L_{\mathcal{B}} (Y
    \otimes_{\tildew{\mathcal{B}}} X'))
  \end{equation*}
  of functors $D(\mathcal{B}) \ra D(\mathcal{A}).$
  Putting
  $Y=\mathcal{B}=\leftidx{_{\mathcal{B}}}{\mathcal{B}}{_{\tildew{\mathcal{B}}}}$ 
  and using the isomorphism
  \eqref{eq:evaluation-trivial-fib}
  we obtain a natural transformation
  \begin{equation*}
    \tau: (? \otimes^L_{\tildew{\mathcal{B}}} X') 
    \comp  (? \otimes^L_{\mathcal{B}} \mathcal{B}) 
    \ra  (? \otimes^L_{\mathcal{B}} X)
  \end{equation*}
  Since obviously
  $(? \otimes^L_{\mathcal{B}} \mathcal{B}) \sira
  \res^{\mathcal{B}}_{\tildew{\mathcal{B}}}$ it is enough to show that
  $\tau$ is an isomorphism.

  Note that $\tau$ is a natural transformation of triangulated
  functors
  that commute with coproducts,
  and recall that $D(\mathcal{B})$ is the localizing subcategory of
  $D(\mathcal{B})$ generated by the objects $\Yoneda{B},$ for $B
  \in \mathcal{B},$ (cf.\ after \eqref{eq:perA-equal-compactDA}).
  Hence to show that $\tau$ is an isomorphism it is sufficient to show
  that $\tau_{\Yoneda{B}}$ is an isomorphism for all $B \in
  \mathcal{B}.$

  Let $B \in \mathcal{B}.$ Since $\Yoneda{B}$ is cofibrant 
  we have $\Yoneda{B} \otimes^L_\mathcal{B} X \cong \Yoneda{B}
  \otimes_{\mathcal{B}} X = X(?, B)$ in $D(\mathcal{A})$
  and
  $\Yoneda{B} \otimes^L_{\mathcal{B}} \mathcal{B} \cong
  \Yoneda{B} \otimes_{\mathcal{B}} \mathcal{B} = \mathcal{B}(b?, B)$
  in $D(\tildew{B}).$
  Since $b$ is surjective on objects
  \ref{enum:surj-epi-on-sets-of-objects}, there is an object $\tildew{B}$ 
  such that $b\tildew{B} = B.$ Then $\Yoneda{\tildew{B}} =
  \tildew{\mathcal{B}}(?, \tildew{B}) \xra{b} \mathcal{B}(b?, B)$ is a cofibrant resolution 
  by \ref{enum:surj-on-homs-surj-qiso} and 
  Theorem~\ref{t:CA-cofib-gen-model-cat}.
  Using this we have
  \begin{equation*}
    (\Yoneda{B} \otimes^L_{\mathcal{B}} \mathcal{B})
    \otimes_{\tildew{\mathcal{B}}}^L X' 
    \cong \Yoneda{\tildew{B}} \otimes_{\tildew{\mathcal{B}}} X' = X'(?,
    \tildew{B})=X(?, b \tildew{B}) = X(?, B) 
  \end{equation*}
  in $D(\mathcal{A}),$ and under this identifications
  $\tau_{\Yoneda{B}}$ is the identity of $X(?,B).$
\end{proof}

\begin{corollary}
  \label{c:lifting-bimodule-equivalences-to-cofibrants}
  Let $\mathcal{A}$ and $\mathcal{B}$ be dg $K$-categories. 
  Assume that $X=\leftidx{_\mathcal{B}}{X}{_\mathcal{A}}$ is a dg
  $\mathcal{A} \otimes_K \mathcal{B}^\opp$-module such that the
  functor 
  $LT_{X}:= (?
  \otimes_\mathcal{B}^L X):
  D(\mathcal{B}) \ra D(\mathcal{A})$ is an equivalence.

  Let $a: \tildew{\mathcal{A}} \ra \mathcal{A}$ and
  $b: \tildew{\mathcal{B}} \ra \mathcal{B}$ be cofibrant resolutions,
  and
  let $\tildew{X}$ be 
  the $\tildew{\mathcal{A}}\otimes_K \tildew{\mathcal{B}}^\opp$-module
  obtained by restriction from $X.$
  Then 
  $LT_{\tildew{X}}:= (?
  \otimes_{\tildew{\mathcal{B}}}^L \tildew{X}):
  D(\tildew{\mathcal{B}}) \ra D(\tildew{\mathcal{A}})$ is an
  equivalence.
\end{corollary}

\begin{proof}
  Let 
  $X_{\tildew{\mathcal{A}}}$ be $X$ viewed as an
  $\tildew{\mathcal{A}}\otimes_K \mathcal{B}^\opp$-module
  and consider the following diagram.
  \begin{equation*}
    \xymatrix
    {
      {D(\tildew{\mathcal{B}})} 
      \ar[rr]^-{? \otimes_{\tildew{\mathcal{B}}}^L \tildew{X}} &&
      {D(\tildew{\mathcal{A}})}\\
      {D({\mathcal{B}})} 
      \ar[rr]_-{? \otimes_\mathcal{B}^L {X}} 
      \ar[u]^{\res^{\mathcal{B}}_{\tildew{\mathcal{B}}}}
      \ar[rru]|-{? \otimes_\mathcal{B}^L X_{\tildew{\mathcal{A}}}}
      &&
      {D({\mathcal{A}})}
      \ar[u]_{\res^{\mathcal{A}}_{\tildew{\mathcal{A}}}}
    }
  \end{equation*}
  Its lower right triangle is obviously commutative. Its upper left
  triangle is commutative up to a natural isomorphism by
  Lemma~\ref{l:lifting-tensor-functor-along-triv-fib}.
  The assumptions imply
  (cf.\ \eqref{eq:restriction-derived})
  that both vertical functors and the lower horizontal functor are
  equivalences. Hence the remaining two arrows are equivalences.
\end{proof}

\begin{proposition}
  \label{p:naive-smoothness-preserved-by-dg-Morita-equivalence-reinspection}
  Let $\mathcal{A}$ and $\mathcal{B}$ be cofibrant dg $K$-categories. 
  Let $X'$ be a dg $\mathcal{A} \otimes_K
  \mathcal{B}^\opp$-module such that the functor $LT_{X'}:= (?
  \otimes_\mathcal{B}^L X'): D(\mathcal{B}) \ra D(\mathcal{A})$ is an
  equivalence.
  Then
  $\mathcal{A}$ is smooth if and only if $\mathcal{B}$ is smooth.
\end{proposition}

\begin{proof}
  The main argument of this proof is from
  \cite[Lemma~3.9]{lunts-categorical-resolution}. Some technical
  details are extracted from \cite[Section~6]{Keller-deriving-dg-cat}. 

  By Lemma~\ref{l:tfae-smoothness} we have to show that
  $\mathcal{A} \in \per(\mathcal{A}\otimes_K \mathcal{A}^\opp)$ 
  if and only if
  $\mathcal{B} \in \per(\mathcal{B}\otimes_K \mathcal{B}^\opp).$

  Let $\xi: X\ra X'$ be a cofibrant resolution in $C(\mathcal{A}
  \otimes_K \mathcal{B}^\opp).$
  Let 
  $N$ be a dg $\mathcal{B}$-module and consider the diagram
  \begin{equation}
    \label{eq:replace-by-naive-tensor-product}
    \xymatrix{
      {LT_{X'}(N)} \gar[r] \ar[rrd]_(0.3)\phi &
      {T_{X'}pN} \ar[d] &
      {T_X pN} \ar[l]_-{T_\xi}^\sim \ar[d]^\sim \gar[r] &
      {LT_X N} \\
      &
      {T_{X'}N} &
      {T_X N} \ar[l]^-{T_\xi} 
    }
  \end{equation}
  in $D(\mathcal{A})$ with obvious vertical and horizontal morphisms. The upper horizontal
  arrow is an isomorphism since the cofibrant dg $\mathcal{B}$-module $pN$ 
  is $\mathcal{B}$-h-flat 
  (Lemma~\ref{l:cofibrant-dg-A-module-is-A-h-flat}).
  The right vertical arrow is an isomorphism since $X$ is $\mathcal{B}^\opp$-h-flat
  (obvious variant of
  Prop.~\ref{p:cofibrant-bimodule-remains-cofibrant-if-insert-cofib-cat}, 
  part~\ref{enum:cofibrant-bimodule-remains-cofibrant-if-partially-evaluate}).
  We define $\phi=\phi_N$ to be the indicated composition of these
  isomorphisms.
  In fact this extends to a natural isomorphism
  $\phi: LT_{X'} \sira T_{X}$ of functors $D(\mathcal{B}) \ra
  D(\mathcal{A})$ (where $T_X$ is defined in the obvious way, using
  that $X$ is $\mathcal{B}^\opp$-h-flat).
  In particular $T_X$ is an equivalence.

  For any $B \in \mathcal{B},$ the dg $\mathcal{A}$-module
  $X^B=X(?,B)$ is cofibrant by
  Proposition~\ref{p:cofibrant-bimodule-remains-cofibrant-if-insert-cofib-cat}
  and in particular 
  h-projective (Lemma~\ref{l:cofobj-is-h-proj-dgA}).
  This implies that $H_X$ as defined in
  \eqref{eq:def-HX} maps acyclic dg $\mathcal{A}$-modules to acyclic
  dg $\mathcal{B}$-modules, and hence descends directly to a
  triangulated functor $H_X: D(\mathcal{A}) \ra D(\mathcal{B}).$
  The unit $\varepsilon$ and counit $\eta$ of the adjunction
  \eqref{eq:adj-tensor-hom-calModlevel} hence directly provide an
  adjunction $(T_X, H_X)$ between 
  $T_X: D(\mathcal{B}) \ra D(\mathcal{A})$ and $H_X: D(\mathcal{A})
  \ra D(\mathcal{B}).$ Since $T_X$ is an equivalence, $H_X$ is a
  quasi-inverse.
  Since $T_X$ preserves all coproducts, the same is true for $H_X.$
  
  Recall the definition of the dg $\mathcal{B} \otimes_K
  \mathcal{A}$-module $X^\perp$ from
  \eqref{eq:def-dual-bimodule}
  and that there is a canonical transformation
  $\tau: T_{X^\perp} \ra H_X.$ This morphism provides a natural
  transformation
  $\tildew{\tau}: LT_{X^\perp} \ra H_X$ of triangulated functors,
  defined on an object $M \in D(\mathcal{A})$ as the indicated
  composition in the following commutative diagram.
  \begin{equation}
    \label{eq:def-tilde-tau}
    \xymatrix{
      {LT_{X^\perp} M} \gar[r] \ar[rrd]_(.3){\tildew{\tau}} & 
      {T_{X^\perp}pM} \ar[r]^\tau
      \ar[d] &
      {H_X(pM)} \ar[d]^\sim \\
      & 
      {T_{X^\perp}M} \ar[r]^\tau &
      {H_X(M)}
    }
  \end{equation}
  It is clear that the vertical morphism on the right is an
  isomorphism, and it is easy to check that the upper horizontal arrow
  is an isomorphism if $M=\Yoneda{A},$ for all $A \in \mathcal{A}$;
  hence all $\tildew{\tau}_{\Yoneda{A}}$ are isomorphisms.
  Since both $LT_{X^\perp}$ and $H_X$
  preserve all coproducts, this implies already
  (by the same argument as in Step 2 of the proof of
  Lemma~\ref{l:lifting-tensor-functor-along-triv-fib})
  that $\tildew{\tau}$ is a natural isomorphism (and that the upper
  horizontal arrow in diagram \eqref{eq:def-tilde-tau} is an
  isomorphism for all $M$).
  
  Let $\upsilon :Y \ra X^\perp$ be a cofibrant resolution of $X^\perp$
  in $C(\mathcal{B} \otimes_K \mathcal{A}^\opp).$
  As above 
  (cf.\ \eqref{eq:replace-by-naive-tensor-product})
  we explicitly construct an isomorphism 
  $\psi: LT_{X^\perp} \sira T_Y.$
  
  Note that $\psi \comp \tildew{\tau}\inv: H_X \sira T_Y$ is an
  isomorphism which shows that $T_X$ has a quasi-inverse given by a
  tensor-functor. 
  
  For $N$ a dg $\mathcal{B}$-module consider the following commutative
  diagram in $D(\mathcal{B})$
  which is built from the adjunction morphism $\varepsilon_N,$ from
  \eqref{eq:def-tilde-tau} and the analog of 
  \eqref{eq:replace-by-naive-tensor-product}.
  \begin{equation*}
    \xymatrix{
      {N} \ar[r]^-{\varepsilon_N}_-\sim &
      {H_XT_X N} &
      {T_{X^\perp}T_X N} \ar[l]_-{\tau} &
      {T_{Y}T_X N} \ar[l]_-{T_\upsilon} \ar@{..>}@/_1cm/[ll]_-{\tau\comp T_\upsilon}^-\sim \\
      &
      {H_XpT_X N} \ar[u]^-\sim &
      {T_{X^\perp}pT_X N} \ar[l]_-{\tau}^-\sim \gar[d] \ar[u] &
      {T_{Y}pT_X N} \ar[l]_-{T_\upsilon}^-\sim \ar[u]_-\sim \\
      &&
      {LT_{X^\perp}T_X N}
      \ar[uul]^(.3){\tildew{\tau}}
      \ar[uur]_(.3){\psi}
    }
  \end{equation*}
  It implies that the dotted composition $\tau \comp T_\upsilon:T_YT_XN \ra
  H_XT_XN$ in the upper row is (as indicated) an isomorphism.
  This is important for the following reason: If $N$ has additionally
  a left dg $\mathcal{R}$-module structure (i.\,e.\ it is an
  dg $\mathcal{B}\otimes_K \mathcal{R}^\opp$-module) then all
  morphisms in the upper row are morphisms of dg $\mathcal{B} \otimes_K
  \mathcal{R}^\opp$-modules and in fact isomorphisms in
  $D(\mathcal{B} \otimes_K \mathcal{R}^\opp)$ since we can test this
  by plugging in $R \in \mathcal{R}.$ 
  (A priori the entries in the lower row have no left dg
  $\mathcal{R}$-module structure.)

  We apply this to $\mathcal{R}= \mathcal{B}$ and the diagonal
  bimodule $\mathcal{B}$ and obtain 
  isomorphisms in $D(\mathcal{B} \otimes_K \mathcal{B}^\opp)$ (or
  quasi-isomorphisms in $C(\mathcal{B} \otimes_K \mathcal{B}^\opp)$ 
  or $\mathcal{H}(\mathcal{B} \otimes_K \mathcal{B}^\opp)$)
  \begin{equation}
    \label{eq:B-per-X-otimes-Y}
    \xymatrix{
      {\mathcal{B}} \ar[r]^-{\varepsilon_\mathcal{B}}_-\sim 
      &      
      {H_XT_X \mathcal{B}} 
      &
      {T_{Y}T_X \mathcal{B}} \ar[l]_-{\tau\comp T_\upsilon}^-\sim
      \gar[r] &
      {\mathcal{B} \otimes_\mathcal{B} X \otimes_\mathcal{A} Y}
      \gar[r] &
      {X \otimes_\mathcal{A} Y}
    }    
  \end{equation}
  Similar we obtain for $M$ a dg $\mathcal{A}$-module the
  following diagram in $D(\mathcal{A}).$
  \begin{equation*}
    \xymatrix{
      {M} &
      {T_XH_X M} 
      \ar[l]_-{\eta_M}^-\sim &
      {T_X T_{X^\perp} M} \ar[l]_-{T_X\tau} &
      {T_XT_{Y} M} \ar[l]_-{T_XT_\upsilon}
      \ar@{..>}@/_1cm/[ll]_-{T_X\tau\comp T_XT_\upsilon}^-\sim \\ 
      &
      {T_XH_Xp M} \ar[u]^-\sim &
      {T_XT_{X^\perp}p M} \ar[l]_-{T_X\tau}^-\sim \gar[d] \ar[u] &
      {T_X T_{Y}p M} \ar[l]_-{T_XT_\upsilon}^-\sim \ar[u]_-\sim \\
      &&
      {T_X LT_{X^\perp} M}
      \ar[uul]^(.3){T_X\tildew{\tau}}
      \ar[uur]_(.3){T_X\psi}
    }
  \end{equation*}
  The upper row is again compatible with any available left dg module
  structure on $M.$ Applied to the diagonal bimodule
  $\mathcal{A}$ we obtain an isomorphism 
  \begin{equation}
    \label{eq:A-per-Y-otimes-X}
    \xymatrix{
      {\mathcal{A}} &
      {T_XH_X \mathcal{A}} 
      \ar[l]_-{\eta_\mathcal{A}}^-\sim &&
      {T_XT_{Y} \mathcal{A}} \ar[ll]_-{T_X\tau\comp
        T_XT_\upsilon}^-\sim \gar[r] &
      {\mathcal{A}\otimes_\mathcal{A} Y \otimes_\mathcal{B} X} \gar[r]
      &
      {Y \otimes_\mathcal{B} X}
    }
  \end{equation}
  in $D(\mathcal{A}\otimes_K \mathcal{A}^\opp).$

  The dg $K$-functor
  \begin{equation*}
    \leftidx{_{Y}}{\Delta}{_X}(?):= Y\otimes_\mathcal{B} ?
    \otimes_\mathcal{B} X: 
    \calMod(\mathcal{B}\otimes_K \mathcal{B}^\opp) 
    \ra
    \calMod(\mathcal{A}\otimes_K \mathcal{A}^\opp) 
  \end{equation*}
  is the composition of the two dg $K$-functors
  $(Y\otimes_\mathcal{B}?)$ and $(?\otimes_\mathcal{B} X)$;
  it preserves acyclic modules since $Y$ is $\mathcal{B}$-flat and $X$
  is $\mathcal{B}^\opp$-flat
  (Prop.~\ref{p:cofibrant-bimodule-remains-cofibrant-if-insert-cofib-cat}, 
  part~\ref{enum:cofibrant-bimodule-remains-cofibrant-if-partially-evaluate}).
  Hence it directly descends to a triangulated functor
  \begin{equation*}
    \leftidx{_Y}{\Delta}{_X} = Y\otimes_\mathcal{B} ?
    \otimes_\mathcal{B} X: 
    D(\mathcal{B}\otimes_K \mathcal{B}^\opp) 
    \ra
    D(\mathcal{A}\otimes_K \mathcal{A}^\opp).
  \end{equation*}

  Similarly, we define a functor
  \begin{equation*}
    \leftidx{_X}{\Delta}{_Y}=
    X\otimes_\mathcal{A} ?
    \otimes_\mathcal{A} Y: 
    D(\mathcal{A}\otimes_K \mathcal{A}^\opp)
    \ra
    D(\mathcal{B}\otimes_K \mathcal{B}^\opp). 
  \end{equation*}

  We claim that $\leftidx{_Y}{\Delta}{_X}$ and
  $\leftidx{_X}{\Delta}{_Y}$ are quasi-inverse to each other.
  This follows from \eqref{eq:B-per-X-otimes-Y} and
  \eqref{eq:A-per-Y-otimes-X} but let us include the details:
  The functor $\leftidx{_X}{\Delta}{_Y} \comp
  \leftidx{_Y}{\Delta}{_X}$ coincides with the obvious composition
  \begin{equation*}
    D(\mathcal{B}\otimes_K \mathcal{B}^\opp)
    \xra{X\otimes_{\mathcal{A}} Y \otimes_{\mathcal{B}} ?}
    D(\mathcal{B}\otimes_K \mathcal{B}^\opp)
    \xra{? \otimes_\mathcal{B} X \otimes_\mathcal{A} Y}
    D(\mathcal{B}\otimes_K \mathcal{B}^\opp).
  \end{equation*}
  The second functor is (canonically isomorphic to) the functor
  $LT_{X\otimes_\mathcal{A} Y}.$ The morphisms in 
  \eqref{eq:B-per-X-otimes-Y} are quasi-isomorphism when considered in
  $C(\mathcal{B}\otimes_K \mathcal{B}^\opp).$ Hence they induce an
  isomorphism between $LT_{X\otimes_\mathcal{A} Y}$ and
  $LT_{\mathcal{B}}=\id$
  (use Lemma~\ref{l:cofibrant-dg-A-module-is-A-h-flat}).
  A similar reasoning applies to the first functor, and hence
  $\leftidx{_X}{\Delta}{_Y} \comp \leftidx{_Y}{\Delta}{_X} \cong \id.$
  Similarly, 
  \eqref{eq:A-per-Y-otimes-X} implies 
  $\leftidx{_Y}{\Delta}{_X} \comp \leftidx{_X}{\Delta}{_Y} \cong \id.$

  The mutually quasi-inverse equivalences 
  $\leftidx{_Y}{\Delta}{_X}$ and $\leftidx{_X}{\Delta}{_Y}$ preserve
  compact objects.
  Using 
  \eqref{eq:A-per-Y-otimes-X} again we have
    $\leftidx{_Y}{\Delta}{_X} (\mathcal{B}) =
    Y \otimes_\mathcal{B} \mathcal{B} \otimes_\mathcal{B} X
    =
    Y \otimes_\mathcal{B} X
    \cong
    \mathcal{A},$
  and similarly 
  $\leftidx{_X}{\Delta}{_Y} (\mathcal{A}) \cong \mathcal{B}.$
  This implies that the diagonal bimodule $\mathcal{A}$ is compact if
  and only if the diagonal bimodule $\mathcal{B}$ is compact.
\end{proof}

\begin{theorem}
  \label{t:smoothness-preserved-by-dg-Morita-equivalence}
  Let $\mathcal{A}$ and $\mathcal{B}$ be dg $K$-categories.
  If $\mathcal{A}$ and $\mathcal{B}$ are dg Morita equivalent, then
  $\mathcal{A}$ is smooth if and only if $\mathcal{B}$ is smooth. 
\end{theorem}

\begin{proof}
  It is enough to show the claim if there is a tensor equivalence
  $D(\mathcal{B}) \sira D(\mathcal{A}).$
  Let $\tildew{\mathcal{A}} \ra \mathcal{A}$ and $\tildew{B} \ra
  \mathcal{B}$ be cofibrant resolutions. By 
  Corollary~\ref{c:lifting-bimodule-equivalences-to-cofibrants} we can
  lift our tensor equivalence to a tensor equivalence
  $D(\tildew{\mathcal{B}}) \sira D(\tildew{\mathcal{A}}).$
  Now the result follows from
  Proposition~\ref{p:naive-smoothness-preserved-by-dg-Morita-equivalence-reinspection}
  and
  Lemma~\ref{l:smoothness-and-quasi-equis}
\end{proof}

\begin{corollary}
  \label{c:smoothness-and-restriction}
  Let $\mathcal{B} \ra \mathcal{A}$ be a morphism in $\dgcat_K.$ If
  $\res^\mathcal{A}_\mathcal{B}: D(\mathcal{A}) \ra D(\mathcal{B})$ is
  an equivalence, then $\mathcal{A}$ is smooth if and only if
  $\mathcal{B}$ is smooth.
\end{corollary}

\begin{proof}
  Let $X$ be $\mathcal{A}$ viewed as a dg $\mathcal{B} \otimes_K
  \mathcal{A}^\opp$-module.
  Then $T_X=(? \otimes_\mathcal{A} X) \sira
  \res^\mathcal{A}_\mathcal{B},$ and hence $LT_X \sira
  \res^\mathcal{A}_\mathcal{B}$ as functors
  $D(\mathcal{A}) \ra D(\mathcal{B}).$
\end{proof}

\begin{corollary}
  \label{c:test-smoothness-on-classical-generator}
  Let $\mathcal{A}$ be a dg $K$-category. Assume that $\mathcal{A}$ is
  triangulated in the sense that it is
  pretriangulated and that its homotopy category $[\mathcal{A}]$ is
  Karoubian (= idempotent complete).
  Assume that there is an object $E \in \mathcal{A}$ such that $E$ is
  a classical generator of $[\mathcal{A}].$ 
  Let $\mathcal{A}(E):= \mathcal{A}(E,E)$ 
  be the dg $K$-algebra of endomorphisms of $E.$
  Then 
  \begin{equation*}
    \mathcal{A}(E,?): [\mathcal{A}] \sira \per(\mathcal{A}(E))
  \end{equation*}
  is a triangulated equivalence, and moreover
  $\mathcal{A}$ is $K$-smooth if and only if 
  $\mathcal{A}(E)$ is $K$-smooth.   
\end{corollary}

\begin{proof}
  Let $\mathcal{E} \subset \mathcal{A}$ be the full dg $K$-subcategory
  whose unique object is $E.$ 
  The upper horizontal functor in the commutative diagram
  \begin{equation*}
    \xymatrix{
      {[\mathcal{A}]} \ar[r] \ar[rd]_-{\mathcal{A}(E,?)}&
      {\per(\mathcal{A})} \ar[d]^-{\res^\mathcal{A}_\mathcal{E}}\\
      & {\per(\mathcal{E})}
    }
  \end{equation*}
  is induced by the Yoneda embedding. It is an equivalence since
  $\mathcal{A}$ is triangulated.
  Since $E$ is a classical generator of $[\mathcal{A}] \sira
  \per(\mathcal{A}),$ both non-horizontal functors in the above
  diagram are equivalences. 
  It follows (cf. explanations around \eqref{eq:perA-equal-compactDA},
  or \cite[Lemma~2.12]{lunts-categorical-resolution})
  that
  $\res^\mathcal{A}_\mathcal{E}: D(\mathcal{A}) \ra D(\mathcal{E})$ is
  an equivalence.
  Now use Corollary~\ref{c:smoothness-and-restriction}.
\end{proof}

\subsection{Directed dg \texorpdfstring{$K$}{K}-categories}
\label{sec:extensions-dg-K-categoreis}

Our aim is to prove
Theorem~\ref{t:diagonal-smooth-bimod-smooth-iff-extension-smooth}
below which generalizes and strengthens \cite[Prop.~3.11]{lunts-categorical-resolution}.

Let $\mathcal{A}$ and $\mathcal{B}$ be two small dg $K$-categories and let
$N=\leftidx{_\mathcal{A}}{N}{_\mathcal{B}}$ be a dg
$\mathcal{B}\otimes_K \mathcal{A}^\opp$-module. Let $\mathcal{E},$
symbolically denoted
\begin{equation*}
  \mathcal{E}=
  \begin{bmatrix}
    {\mathcal{B}} & {0}\\
    {N} & {\mathcal{A}}
  \end{bmatrix},
\end{equation*}
be the following dg $K$-category: 
Its objects are the disjoint union of the objects of $\mathcal{A}$ and
$\mathcal{B},$ and its morphisms are given by
\begin{align*}
  \mathcal{E}(A,A') & = \mathcal{A}(A,A'), &
  \mathcal{E}(B,A') & = N(B,A'),\\
  \mathcal{E}(A,B') & = 0, &
  \mathcal{E}(B,B') & = \mathcal{B}(B,B'),
\end{align*}
for objects $A, A' \in \Obj \mathcal{A} \subset \Obj \mathcal{E}$ and $B, B' \in
\Obj \mathcal{B} \subset \Obj \mathcal{E},$
and units and compositions are obvious 
(e.\,g.\ for $A \in \mathcal{A}$
and $B, B' \in \mathcal{B}$ composition is given by the action morphism
\begin{equation*}
  \mathcal{E}(B', A) \otimes_K \mathcal{E}(B,B')
  = N(B', A) \otimes_K \mathcal{B}(B,B') \ra N(B,A) = \mathcal{E}(B, A)
\end{equation*}
of the dg $\mathcal{B}$-module $N$).

\begin{remark}
  \label{rem:directed-dg-K-category}
  Conversely, if $\mathcal{E}$ is a small dg $K$-category such that we can
  split the set of objects into two disjoint subsets, giving rise to
  full subcategories $\mathcal{A}$ and $\mathcal{B},$ such that
  $\mathcal{E}(\mathcal{A}, \mathcal{B})=0,$ then 
  $\mathcal{E}= \tzmat{\mathcal{B}}{0}{N}{\mathcal{A}}$
  for $N:=\mathcal{E}|_{\mathcal{B}\otimes_K \mathcal{A}^\opp}$ the indicated
  restriction of the diagonal bimodule $\mathcal{E}.$
\end{remark}

\begin{remark}
  The quiver picture of $\mathcal{E}$ is $\mathcal{B}
  \xra{\leftidx{_\mathcal{A}}{N}{_\mathcal{B}}} \mathcal{A}.$
\end{remark}

If $S$ is a dg $\mathcal{E}$-module we can restrict it along the
obvious inclusions $\mathcal{A} \subset \mathcal{E}$ 
and $\mathcal{B} \subset \mathcal{E}$ and obtain a dg $\mathcal{A}$-module $S|_\mathcal{A}$ and a
dg $\mathcal{B}$-module $S|_\mathcal{B}.$ 
Furthermore the action morphisms
$S(A) \otimes_K \mathcal{E}(B,A) \ra S(B),$ for $A \in \mathcal{A}$
and $B \in \mathcal{B},$
induce a morphism
\begin{equation*}
  \phi_S: S|_\mathcal{A} \otimes_\mathcal{A} N \ra S|_{\mathcal{B}}
\end{equation*}
in $C(\mathcal{B}).$
In this manner we see that a dg $\mathcal{E}$-module $S$ is the same
as a triple
$(S_\mathcal{A}, S_\mathcal{B},
\phi:S_\mathcal{A} \otimes_\mathcal{A} N \ra S_\mathcal{B})$ where
$S_\mathcal{A}$ is a dg $\mathcal{A}$-module, $S_\mathcal{B}$ is a dg
$\mathcal{B}$-module and $\phi$ is a morphism in $C(\mathcal{B}).$ We
describe such a dg $\mathcal{E}$-module symbolically as
$S=\mathovalbox{
  \xymatrix{
    {S_\mathcal{B}} &
    {S_\mathcal{A}.} \ar[l]_-{\phi}
  }
}$

Morphisms of $f:S \ra S'$ of dg $\mathcal{E}$-modules are in this
description pairs $(f_\mathcal{B}, f_\mathcal{A})$ where
$f_\mathcal{B}: S_\mathcal{B} \ra S'_\mathcal{B}$ and
$f_\mathcal{A}: S_\mathcal{A} \ra S'_\mathcal{A}$ are morphisms
of dg $\mathcal{B}$- and $\mathcal{A}$-modules respectively such that
$f_{\mathcal{B}} \comp \phi_S = \phi_{S'} \comp (f_{\mathcal{A}}
\otimes_{\mathcal{A}} \id_N).$ We denote such a morphism
symbolically as $f=\mathovalbox{
  \begin{smallmatrix}
    f_\mathcal{B} &
    f_\mathcal{A} 
  \end{smallmatrix}
}.$

\begin{lemma}
  \label{l:restriction-to-target-of-N-preserves-compacts}
  Let $a:\mathcal{A} \subset \mathcal{E}$ be the obvious inclusion.
  The functor $a_*=\res^\mathcal{E}_\mathcal{A}: D(\mathcal{E}) \ra
  D(\mathcal{A})$ maps compact objects to compact objects.
\end{lemma}

\begin{proof}
  This functor has a right adjoint 
  $D(\mathcal{A}) \ra D(\mathcal{E}),$ defined by
  $U \mapsto 
    \mathovalbox{
      \xymatrix{
        {0} &
        {U} \ar[l]
      }
    },$
  which preserves all coproducts. This implies the statement.
\end{proof}

The inclusion $b:\mathcal{B} \subset \mathcal{E}$ defines the dg
$K$-functor
$\pro_\mathcal{B}^\mathcal{E}:= (? \otimes_\mathcal{B} \mathcal{E}): 
\calMod(\mathcal{B}) \ra \calMod(\mathcal{E}),$ given by
\begin{equation}
  \label{eq:production-BE}
  \pro_\mathcal{B}^\mathcal{E} (V)
  =
  \mathovalbox{
    \xymatrix{
      {V} &
      {0} \ar[l]
    }
  }.
\end{equation}
It preserves acyclics and descends to a triangulated functor
$b^*:=\pro_\mathcal{B}^\mathcal{E}: D(\mathcal{B}) \ra D(\mathcal{E}).$
This functor has the right adjoint functor
$b_*:=\res^\mathcal{E}_\mathcal{B}:D(\mathcal{E}) \ra D(\mathcal{B}),$
mapping $S$ as above to $S_\mathcal{B}.$

\begin{lemma}
  \label{l:prod-from-source-cpt-implies-cpt}
  Let $V \in D(\mathcal{B}).$ If 
  $b^*(V)=\pro_\mathcal{B}^\mathcal{E} (V)$ is compact in
  $D(\mathcal{E}),$ then $V$ is compact in $D(\mathcal{B}).$
\end{lemma}

\begin{proof}
  Obviously $b^*: D(\mathcal{B}) \ra D(\mathcal{E})$
  commutes with all coproducts. The unit $\id \ra b_* b^*$ of the
  adjunction is an isomorphism, so $b^*$ is fully faithful.
  This implies that statement.
\end{proof}

\subsection{Smoothness of directed dg \texorpdfstring{$K$}{K}-categories}
\label{sec:smoothn-direct-dg}

The following is the quiver picture of
$\mathcal{E}\otimes_K \mathcal{E}^\opp,$ where we additionally have
drawn the quiver of $\mathcal{E}$ on the horizontal axis and that of
$\mathcal{E}^\opp$ on the vertical axis 
(note that the $\mathcal{B} \otimes_K
\mathcal{A}$-module $N$ can be viewed as a $\mathcal{A}^\opp\otimes_K
\mathcal{B}$-module $N^\opp$):
\begin{equation}
  \label{eq:quiver-of-E-Eopp}
  \xymatrix@=0.4cm{
    {\mathcal{B}^\opp} &&
    {\mathcal{B} \otimes_K \mathcal{B}^\opp} 
    \ar[rr]^-{N \otimes_K \mathcal{B}^\opp} && 
    {\mathcal{A} \otimes_K \mathcal{B}^\opp} \\ \\
    {\mathcal{A}^\opp} \ar[uu]^-{N^\opp} &&
    {\mathcal{B} \otimes_K \mathcal{A}^\opp} 
    \ar[uu]^-{\mathcal{B} \otimes_K N^\opp}
    \ar[rruu]|-{N \otimes_K N^\opp}
    \ar[rr]_-{N \otimes_K \mathcal{A}^\opp} && 
    {\mathcal{A} \otimes_K \mathcal{A}^\opp} 
    \ar[uu]_-{\mathcal{A} \otimes_K N^\opp}
    \\
    &&
    {\mathcal{B}}
    \ar[rr]^{N} &&
    {\mathcal{A}}
  }
\end{equation}
We can describe 
dg $\mathcal{E}\otimes_K \mathcal{E}^\opp$-modules in a
similar way as explained above for dg $\mathcal{E}$-modules:
Let $M$ be such a module (we always view it
implicitly as a bimodule). Restriction along the four morphisms of dg
$K$-categories 
$\mathcal{B} \otimes_K \mathcal{B}^\opp \ra \mathcal{E} \otimes_K \mathcal{E}^\opp,$
$\mathcal{A} \otimes_K \mathcal{B}^\opp \ra \mathcal{E} \otimes_K
\mathcal{E}^\opp,$
$\mathcal{B} \otimes_K \mathcal{A}^\opp \ra \mathcal{E} \otimes_K \mathcal{E}^\opp,$
$\mathcal{A} \otimes_K \mathcal{A}^\opp \ra \mathcal{E} \otimes_K \mathcal{E}^\opp$
gives rise to the dg modules
\begin{equation*}
  \leftidx{_\mathcal{B}}{M}{_\mathcal{B}}:=
  \leftidx{_\mathcal{B}^\mathcal{E}}{\res}{_\mathcal{B}^\mathcal{E}}(M), \quad
  \leftidx{_\mathcal{B}}{M}{_\mathcal{A}}:=
  \leftidx{_\mathcal{B}^\mathcal{E}}{\res}{_\mathcal{A}^\mathcal{E}}(M), \quad
  \leftidx{_\mathcal{A}}{M}{_\mathcal{B}}:=
  \leftidx{_\mathcal{A}^\mathcal{E}}{\res}{_\mathcal{B}^\mathcal{E}}(M), \quad
  \leftidx{_\mathcal{A}}{M}{_\mathcal{A}}:=
  \leftidx{_\mathcal{A}^\mathcal{E}}{\res}{_\mathcal{A}^\mathcal{E}}(M).
\end{equation*}
Furthermore the action morphisms of $N$ from the right and from
the left give rise to closed degree zero morphisms
\begin{align*}
  \leftidx{_\mathcal{B}}{\theta}{_{\mathcal{A}\mathcal{B}}}:(\leftidx{_\mathcal{B}}{M}{_\mathcal{A}})
  \otimes_\mathcal{A} N & \ra \leftidx{_\mathcal{B}}{M}{_\mathcal{B}},&
  \leftidx{_\mathcal{A}}{\theta}{_{\mathcal{A}\mathcal{B}}}:(\leftidx{_\mathcal{A}}{M}{_\mathcal{A}})
  \otimes_\mathcal{A} N & \ra \leftidx{_\mathcal{A}}{M}{_\mathcal{B}},\\
  \leftidx{_{\mathcal{B}\mathcal{A}}}{\theta}{_{\mathcal{A}}}:N\otimes_\mathcal{B} (\leftidx{_\mathcal{B}}{M}{_\mathcal{A}})
  & \ra \leftidx{_\mathcal{A}}{M}{_\mathcal{A}},&
  \leftidx{_{\mathcal{B}\mathcal{A}}}{\theta}{_{\mathcal{B}}}:N\otimes_\mathcal{B} (\leftidx{_\mathcal{B}}{M}{_\mathcal{B}})
  & \ra \leftidx{_\mathcal{A}}{M}{_\mathcal{B}}
\end{align*}
of suitable dg modules, for example the last morphism is a morphism
in $C(\mathcal{B} \otimes_K \mathcal{A}^\opp).$
These morphisms fit in the commutative
diagram (since the $N$-left action and the $N$-right action commute)
\begin{equation}
  \label{eq:theta-compatibility-E-bimodule}
  \xymatrix{
    {N \otimes_\mathcal{B} (\leftidx{_\mathcal{B}}{M}{_\mathcal{B}})}
    \ar[d]^-{\leftidx{_{\mathcal{B}\mathcal{A}}}{\theta}{_{\mathcal{B}}}}
    & &
    {N \otimes_\mathcal{B} (\leftidx{_\mathcal{B}}{M}{_\mathcal{A}})
      \otimes_\mathcal{A} N}
    \ar[ll]_-{\id\otimes (\leftidx{_\mathcal{B}}{\theta}{_{\mathcal{A}\mathcal{B}}})}
    \ar[d]^-{(\leftidx{_{\mathcal{B}\mathcal{A}}}{\theta}{_{\mathcal{A}}})\otimes \id}
    \\
    {\leftidx{_\mathcal{A}}{M}{_\mathcal{B}}}
    & &
    {(\leftidx{_\mathcal{A}}{M}{_\mathcal{A}}) \otimes_\mathcal{A} N.}
    \ar[ll]_-{\leftidx{_\mathcal{A}}{\theta}{_{\mathcal{A}\mathcal{B}}}}
  }
\end{equation}
We conclude that a dg $\mathcal{E}\otimes_\mathcal{E}^\opp$-module $M$
is (equivalent to) the datum 
\begin{equation*}
  \leftidx{_\mathcal{B}}{M}{_\mathcal{A}},
  \leftidx{_\mathcal{A}}{M}{_\mathcal{A}},
  \leftidx{_\mathcal{B}}{M}{_\mathcal{B}},
  \leftidx{_\mathcal{A}}{M}{_\mathcal{B}},
  \leftidx{_\mathcal{B}}{\theta}{_{\mathcal{A}\mathcal{B}}},
  \leftidx{_\mathcal{A}}{\theta}{_{\mathcal{A}\mathcal{B}}},
  \leftidx{_{\mathcal{B}\mathcal{A}}}{\theta}{_{\mathcal{A}}},
  \leftidx{_{\mathcal{B}\mathcal{A}}}{\theta}{_{\mathcal{B}}},
\end{equation*}
of dg bimodules and closed degree zero morphisms as above such that
\eqref{eq:theta-compatibility-E-bimodule} commutes.
It is convenient to describe such a bimodule $M$ symbolically by the diagram
\begin{equation*}
  \mathovalbox{
    \xymatrix{
      {\leftidx{_\mathcal{B}}{M}{_\mathcal{B}}}
      \ar[d]^-{\leftidx{_{\mathcal{B}\mathcal{A}}}{\theta}{_{\mathcal{B}}}}
      &
      {\leftidx{_\mathcal{B}}{M}{_\mathcal{A}}}
      \ar[l]_-{\leftidx{_\mathcal{B}}{\theta}{_{\mathcal{A}\mathcal{B}}}}
      \ar[d]^-{\leftidx{_{\mathcal{B}\mathcal{A}}}{\theta}{_{\mathcal{A}}}}
      \\
      {\leftidx{_\mathcal{A}}{M}{_\mathcal{B}}}
      &
      {\leftidx{_\mathcal{A}}{M}{_\mathcal{A}}.}
      \ar[l]_-{\leftidx{_\mathcal{A}}{\theta}{_{\mathcal{A}\mathcal{B}}}}
    }
  }
\end{equation*}

The diagonal bimodule $\mathcal{E}$ (cf.\ \eqref{eq:diagonal-module})
is given by the diagram
\begin{equation}
  \label{eq:E-diagonal}
  \mathcal{E} = \mathovalbox{
    \xymatrix@=\myxysize{
      {\mathcal{B}} \ar[d]^\id
      &
      {0} \ar[l] \ar[d] 
      \\
      {N}
      &
      {\mathcal{A}.} \ar[l]_\id 
    }
  }
\end{equation}
where we identify $\mathcal{A} \otimes_{\mathcal{A}} N=N$ and
$N \otimes_{\mathcal{B}} \mathcal{B}=N.$

We described in \eqref{eq:production-BE}
an extension of scalars functor. Similarly, we have
extension of scalars functors for bimodules, e.\,g.\ the morphism
$\mathcal{A}\otimes_K \mathcal{A}^\opp \ra \mathcal{E} \otimes_K
\mathcal{E}^\opp$ induces the extensions of scalars functor
\begin{equation*}
  \leftidx{_\mathcal{A}^{\mathcal{E}}}{\pro}{_\mathcal{A}^\mathcal{E}}:
  \calMod(\mathcal{A}\otimes_K \mathcal{A}^\opp) 
  \ra
  \calMod(\mathcal{E}\otimes_K \mathcal{E}^\opp).
\end{equation*}
A computation shows that it maps a dg $\mathcal{A} \otimes_K
\mathcal{A}^\opp$-module $X$
to
\begin{equation}
  \label{eq:prod-AA-EE}
  (\leftidx{_\mathcal{A}^{\mathcal{E}}}{\pro}{_\mathcal{A}^\mathcal{E}})(X) =
  \mathovalbox{
    \xymatrix@=\myxysize{
      {0} \ar[d] 
      &
      {0} \ar[l] \ar[d]
      \\
      {X\otimes_\mathcal{A} N}
      &
      {X.} \ar[l]_-{\id}
    }
  }
\end{equation}
In particular we can apply this to the diagonal
$\mathcal{A}\otimes_K \mathcal{A}^\opp$-module $\mathcal{A}$ and obtain
\begin{equation}
  \label{eq:prod-AA-EE-A}
  (\leftidx{_\mathcal{A}^{\mathcal{E}}}{\pro}{_\mathcal{A}^\mathcal{E}})(\mathcal{A}) =
  \mathovalbox{
    \xymatrix@=\myxysize{
      {0} \ar[d] &
      {0} \ar[l] \ar[d]\\
      {N}
      &
      {\mathcal{A}.} \ar[l]_-\id 
    }
  }
\end{equation}

Similarly, we can induce along $\mathcal{B} \otimes_K
\mathcal{B}^\opp \ra \mathcal{E} \otimes_K \mathcal{E}^\opp$ and
obtain for the diagonal bimodule $\mathcal{B}$ that
\begin{equation}
  \label{eq:prod-BB-EE-B}
  (\leftidx{_\mathcal{B}^{\mathcal{E}}}{\pro}{_\mathcal{B}^\mathcal{E}})(\mathcal{B}) =
  \mathovalbox{
    \xymatrix@=\myxysize{
      {\mathcal{B}} \ar[d]^-\id 
      &
      {0} \ar[l] \ar[d] 
      \\
      {N}
      &
      {0.} \ar[l] 
    }
  }
\end{equation}

Similarly, extension of scalars along
$\mathcal{B} \otimes_K \mathcal{A}^\opp \ra \mathcal{E} \otimes_K
\mathcal{E}^\opp$ applied to $N$ gives
\begin{equation}
  \label{eq:prod-BA-EE-N}
  (\leftidx{_\mathcal{A}^{\mathcal{E}}}{\pro}{_\mathcal{B}^\mathcal{E}})(N) =
  \mathovalbox{
    \xymatrix@=\myxysize{
      {0} \ar[d] &
      {0} \ar[l] \ar[d] \\
      {N}
      &
      {0.} \ar[l] 
    }
  }
\end{equation}

Now we can prove the following generalization and strengthening of
\cite[Prop.~3.11]{lunts-categorical-resolution}.

\begin{theorem}
  \label{t:diagonal-smooth-bimod-smooth-iff-extension-smooth}
  Let $\mathcal{A}$ and $\mathcal{B}$ be dg $K$-categories, let $N$
  be a dg $\mathcal{B}\otimes_K \mathcal{A}^\opp$-module, and let
  \begin{equation*}
    \mathcal{E}=
    \begin{bmatrix}
      {\mathcal{B}} & {0}\\
      {N} & {\mathcal{A}.}
    \end{bmatrix}
  \end{equation*}
  The following conditions are equivalent:
  \begin{enumerate}[label=(E{\arabic*})]
  \item 
    \label{enum:diagonal-smooth-bimod-cpt}
    $\mathcal{A}$ and $\mathcal{B}$ are $K$-smooth and $N$ is
    $K$-\sweet{} (see Definition~\ref{d:sweet}).
  \item 
    \label{enum:extension-smooth}
    $\mathcal{E}$ is $K$-smooth.
  \end{enumerate}
\end{theorem}

\begin{remark}
  \label{rem:dg-algebras-and-bimodule-versus-general-setting}
  If $A$ and $B$ are dg $K$-algebras and $N$ is a dg $B \otimes_K
  A^\opp$-module, then we can view $A$ and $B$ as dg $K$-categories
  and form the dg $K$-category $\mathcal{E}= 
  \Big[\begin{smallmatrix}
    B & 0\\
    N & A
  \end{smallmatrix}\Big]
  $
  with two objects as above.
  On the other hand we can consider the obvious dg $K$-algebra $E:=
  \Big[\begin{smallmatrix}
    B & 0\\
    N & A
  \end{smallmatrix}\Big]
  .$
  Then the obvious dg $K$-equivalence $\calMod(E) \sira
  \calMod(\mathcal{E})$ is given by an $\mathcal{E} \otimes_K
  E^\opp$-bimodule, and hence $E$ and $\mathcal{E}$ are dg Morita
  equivalent.
  In particular 
  $E$ is $K$-smooth if and only if $\mathcal{E}$ is $K$-smooth
  (Theorem~\ref{t:smoothness-preserved-by-dg-Morita-equivalence}).

  This remark obviously generalizes to dg K-categories/algebras of the
  form
  $
  \Big[\begin{smallmatrix}
    B & M\\
    N & A
  \end{smallmatrix}\Big]
  ,$ 
  where $M$ is a dg $A \otimes_K B^\opp$-module, and also to bigger
  matrix categories/algebras.
\end{remark}

\begin{example}
  \label{ex:kVk-smooth-iff-V-findim}
  Assume that $\groundring$ is a field,
  and let $V$ be a (dg) $\groundring$-module. Then 
  the
  dg $\groundring$-algebra
  $
  \Big[\begin{smallmatrix}
    \groundring & 0\\
    V & \groundring
  \end{smallmatrix}\Big]
  $
  is $\groundring$-smooth if and only if $V$ is finite dimensional.
\end{example}

\begin{example}
  \label{ex:C-on-diagonal-not-smooth}
  Consider $\DC[X]$ as a dg $\DC$-algebra with $X$ of positive
  even degree.
  Then the dg $\DC[X]$-algebra
  $
  \Big[\begin{smallmatrix}
    \DC & 0\\
    \DC & \DC[X]
  \end{smallmatrix}\Big]
  $
  is not $\DC[X]$-smooth since $\DC$ is not $\DC[X]$-smooth, see
  Example~\ref{ex:C-and-CX-and-smoothness} \ref{enum:CXmodXn-not-CX-smooth}.
\end{example}

\begin{proof}
  \textbf{Step 1:}
  Reduction to the case that $\mathcal{A}$ and $\mathcal{B}$ are
  cofibrant dg $K$-categories and that
  $N$ is a cofibrant dg $\mathcal{A} \otimes_K
  \mathcal{B}^\opp$-module.
  
  Let ${\tildew{\mathcal{A}}} \ra \mathcal{A}$ and ${\tildew{\mathcal{B}}}
  \ra \mathcal{B}$ be cofibrant resolutions 
  and let
  ${\tildew{N}} \ra 
  \leftidx{_{{\tildew{\mathcal{A}}}}}{N}{_{{\tildew{\mathcal{B}}}}}$ 
  be a cofibrant resolution (in $C({\tildew{\mathcal{B}}}\otimes_K
  {\tildew{\mathcal{A}}}^\opp)$) of the restriction 
  $\leftidx{_{{\tildew{\mathcal{A}}}}}{N}{_{{\tildew{\mathcal{B}}}}}:=
  (\leftidx{^\mathcal{A}_{{\tildew{\mathcal{A}}}}}{\res}{^\mathcal{B}_{{\tildew{\mathcal{B}}}}})(N)$ 
  of $N.$
  Note that ${\tildew{N}} \sira
  \leftidx{_{{\tildew{\mathcal{A}}}}}{N}{_{{\tildew{\mathcal{B}}}}}$
  in
  $D({\tildew{\mathcal{B}}} \otimes_K
    {\tildew{\mathcal{A}}}^\opp).$
  Then \ref{enum:diagonal-smooth-bimod-cpt} is by definition
  equivalent to 
  \begin{equation*}
    \text{${\tildew{\mathcal{A}}} \in \per({\tildew{\mathcal{A}}}
      \otimes_K {\tildew{\mathcal{A}}}^\opp),$ 
    ${\tildew{\mathcal{B}}} \in \per({\tildew{\mathcal{B}}} \otimes_K
    {\tildew{\mathcal{B}}}^\opp),$ 
    and 
    ${\tildew{N}} \in \per({\tildew{\mathcal{B}}} \otimes_K
    {\tildew{\mathcal{A}}}^\opp),$}
  \end{equation*}
  which is equivalent to $\tildew{\mathcal{A}},$ $\tildew{\mathcal{B}}$
  being $K$-smooth and $\tildew{N}$ being $K$-\sweet{}.

  Let 
  \begin{equation*}
    \mathcal{E}'=
    \begin{bmatrix}
      {{\tildew{\mathcal{B}}}} & 0\\
      {{\tildew{N}}} & {{\tildew{\mathcal{A}}}}
    \end{bmatrix}.
  \end{equation*}
  The obvious morphism $\mathcal{E}' \ra \mathcal{E}$ of dg
  $K$-categories is a 
  trivial fibration 
  (check the conditions \ref{enum:surj-on-homs-surj-qiso} and
  \ref{enum:surj-epi-on-sets-of-objects} using 
  the description of the trivial fibrations in
  Theorem~\ref{t:CA-cofib-gen-model-cat}), and even a
  $K$-h-flat resolution since $\mathcal{E}'$ is $K$-h-flat 
  (use Lemma~\ref{l:cofibrant-dg-K-cat-has-cofibrant-morphism-spaces} and
  Proposition~\ref{p:cofibrant-bimodule-remains-cofibrant-if-insert-cofib-cat},
  part~
  \ref{enum:cofibrant-bimodule-remains-cofibrant-if-evaluate}).
  Hence 
  \ref{enum:extension-smooth}
  is by definition equivalent to
  $\mathcal{E}' \in \per(\mathcal{E}' \otimes_K \mathcal{E}'^\opp),$
  i.\,e.\ to $K$-smoothness of $\mathcal{E}'.$

  \textbf{Step 2:}
  Assume $\mathcal{A}$ and $\mathcal{B}$ are cofibrant dg $K$-categories and that
  $N$ is a cofibrant dg $\mathcal{A} \otimes_K
  \mathcal{B}^\opp$-module. Let
  $\mathcal{E} = \tzmat{\mathcal{B}}{0}{N}{\mathcal{A}}.$
  By Step 1 it is enough to prove that \ref{enum:diagonal-smooth-bimod-cpt}
  and \ref{enum:extension-smooth} are equivalent under these additional
  assumptions.

  Since $N$ is a cofibrant dg $\mathcal{B} \otimes_K \mathcal{A}^\opp$-module,
  the functor
  $\leftidx{_\mathcal{A}^{\mathcal{E}}}{\pro}{_\mathcal{A}^\mathcal{E}}$
  preserves acyclics 
  (use \eqref{eq:prod-AA-EE} and the fact that $N$ is
  $\mathcal{A}^\opp$-h-flat
  by
  Prop.~\ref{p:cofibrant-bimodule-remains-cofibrant-if-insert-cofib-cat},
  part~\ref{enum:cofibrant-bimodule-remains-cofibrant-if-partially-evaluate}
  (and Lemma~\ref{l:cofibrant-dg-K-cat-has-cofibrant-morphism-spaces}))
  and hence trivially descends (without taking cofibrant/choosing h-projective resolutions) 
  to a functor between the derived categories denoted by the same symbol.
  The analog statement holds for
  $\leftidx{_\mathcal{B}^{\mathcal{E}}}{\pro}{_\mathcal{B}^\mathcal{E}}$ (similar proof) and for
  $\leftidx{_\mathcal{A}^{\mathcal{E}}}{\pro}{_\mathcal{B}^\mathcal{E}}$ (obvious from
  the fact that it maps $Z$ to $\mathovalbox{\tzquadmat 00Z0},$ as in
  \eqref{eq:prod-BA-EE-N} for $Z=N$).
  These extension of scalars functors preserve compact
  objects since restriction, their right adjoints, obviously commute
  with all coproducts.

  The short exact sequence
  \begin{equation*}
    0 \ra
    \mathovalbox{
      \xymatrix@=\myxysize{
        {0} \ar[d] &
        {0} \ar[l] \ar[d]\\
        {N}
        &
        {0.} \ar[l]
      }
    }
    \xra{
      \begin{bmatrix}
        \mathovalbox{
          \tzquadmat 0010
        } \\
        \mathovalbox{\tzquadmat 00{-1}0} 
      \end{bmatrix}
    }
    \mathovalbox{
      \xymatrix@=\myxysize{
        {0} \ar[d] &
        {0} \ar[l] \ar[d]\\
        {N}
        &
        {\mathcal{A}.} \ar[l]_-\id 
      }
    }
    \oplus
    \mathovalbox{
      \xymatrix@=\myxysize{
        {\mathcal{B}} \ar[d]^-\id 
        &
        {0} \ar[l] \ar[d] 
        \\
        {N}
        &
        {0.} \ar[l] 
      }
    }
    \xra{
      \begin{bmatrix}
        \mathovalbox{\tzquadmat 0011}, &
        \mathovalbox{\tzquadmat 1010} 
      \end{bmatrix}
    }
    \mathovalbox{
      \xymatrix@=\myxysize{
        {\mathcal{B}} \ar[d]^-\id 
        &
        {0} \ar[l] \ar[d] 
        \\
        {N}
        &
        {\mathcal{A}.} \ar[l]_-{\id} 
      }
    }
    \ra
    0
  \end{equation*}
  in $C(\mathcal{E}\otimes_K \mathcal{E}^\opp)$
  gives rise 
  (cf.\ \eqref{eq:prod-AA-EE-A},
  \eqref{eq:prod-BB-EE-B},
  \eqref{eq:prod-BA-EE-N}, and
  \eqref{eq:E-diagonal}) to the triangle
  \begin{equation}
    \label{eq:prod-triangle-diagonal-E}
    (\leftidx{_\mathcal{A}^{\mathcal{E}}}{\pro}{_\mathcal{B}^\mathcal{E}})(N) 
    \ra
    (\leftidx{_\mathcal{A}^{\mathcal{E}}}{\pro}{_\mathcal{A}^\mathcal{E}})(\mathcal{A})\oplus
    (\leftidx{_\mathcal{B}^{\mathcal{E}}}{\pro}{_\mathcal{B}^\mathcal{E}})(\mathcal{B})
    \ra
    \mathcal{E} \ra 
    [1](\leftidx{_\mathcal{A}^{\mathcal{E}}}{\pro}{_\mathcal{B}^\mathcal{E}})_(N)
  \end{equation}
  in $D(\mathcal{E}\otimes_K \mathcal{E}^\opp).$

  If
  \ref{enum:diagonal-smooth-bimod-cpt}
  holds, the first two corners of the triangle
  \eqref{eq:prod-triangle-diagonal-E} are compact,
  hence the third
  corner $\mathcal{E}$ is compact. This proves
  \ref{enum:extension-smooth}.

  Conversely assume that
  \ref{enum:extension-smooth}
  holds, i.\,e.\ the diagonal bimodule $\mathcal{E}$ is compact in 
  $D(\mathcal{E}\otimes_K \mathcal{E}^\opp).$

  We first prove that $\mathcal{B}
  \in \per(\mathcal{B}\otimes_K \mathcal{B}^\opp).$
  Let $\mathcal{F}:= \mathcal{A} \otimes_K \mathcal{B}^\opp \subset
  \mathcal{E}\otimes_K \mathcal{E}^\opp$ be the full dg
  $K$-subcategory (the upper right corner in the quiver picture
  \eqref{eq:quiver-of-E-Eopp}), and let 
  $\mathcal{U}$ be its complement. Then there are no non-zero morphisms in
  $\mathcal{E}\otimes_K \mathcal{E}^\opp$ from $\mathcal{F}$ to
  $\mathcal{U},$ hence 
  we are in the situation of 
  Section~\ref{sec:extensions-dg-K-categoreis}, cf.\
  Remark~\ref{rem:directed-dg-K-category}.
  In particular we have the functor (cf.\ \eqref{eq:production-BE})
  \begin{equation*}
    u^*:= \pro_\mathcal{U}^{\mathcal{E} \otimes_K \mathcal{E}^\opp}:
    D(\mathcal{U}) \ra D(\mathcal{E}\otimes_K \mathcal{E}^\opp) 
  \end{equation*}
  and its right adjoint functor 
  $u_*=\res_\mathcal{U}^{\mathcal{E} \otimes_K \mathcal{E}^\opp}.$ 
  Since the diagonal bimodule $\mathcal{E}$ 
  (cf.\ \eqref{eq:E-diagonal}) "has support in $\mathcal{U}$" it
  satisfies
  $\mathcal{E}=u^*u_*\mathcal{E},$ and
  Lemma~\ref{l:prod-from-source-cpt-implies-cpt}
  shows that
  $u_*\mathcal{E} \in D(\mathcal{U})$ is compact.

  Now consider the full dg $K$-subcategory $i: \mathcal{B} \otimes_K
  \mathcal{B}^\opp \subset \mathcal{U}$ (the upper left corner in the quiver picture
  \eqref{eq:quiver-of-E-Eopp}).
  There are no non-zero morphisms from 
  $\mathcal{B} \otimes_K \mathcal{B}^\opp$ to
  its complement 
  in $\mathcal{U},$ hence we are again in the
  situation of
  Section~\ref{sec:extensions-dg-K-categoreis}, and have the functor
  $i_*:=\res^\mathcal{U}_{\mathcal{B} \otimes_K \mathcal{B}^\opp}:
  D(\mathcal{U}) \ra D(\mathcal{B}\otimes_K \mathcal{B}^\opp).$
  Since 
  $i_*u_*\mathcal{E} \in D(\mathcal{B}\otimes_K \mathcal{B}^\opp)$ is
  the diagonal bimodule $\mathcal{B}$ and compact by
  Lemma~\ref{l:restriction-to-target-of-N-preserves-compacts}
  we obtain $\mathcal{B} \in \per(\mathcal{B}\otimes_K \mathcal{B}^\opp).$

  Similarly, the inclusion $\mathcal{A} \otimes_K \mathcal{A}^\opp \hra
  \mathcal{U}$ yields 
  $\mathcal{A} \in \per(\mathcal{A}\otimes_K \mathcal{A}^\opp).$
  Then in the triangle \eqref{eq:prod-triangle-diagonal-E} the objects
  $(\leftidx{_\mathcal{A}^{\mathcal{E}}}{\pro}{_\mathcal{A}^\mathcal{E}})(\mathcal{A})
  \oplus
  (\leftidx{_\mathcal{B}^{\mathcal{E}}}{\pro}{_\mathcal{B}^\mathcal{E}})(\mathcal{B})$
  and $\mathcal{E}$ are compact, hence
  $(\leftidx{_\mathcal{A}^{\mathcal{E}}}{\pro}{_\mathcal{B}^\mathcal{E}})(N)$ is compact.
  If we denote the inclusion $\mathcal{B} \otimes_K \mathcal{A}^\opp \hra
  \mathcal{E}\otimes_K \mathcal{E}^\opp$ by $v$ we have
  $v^*=\leftidx{_\mathcal{A}^{\mathcal{E}}}{\pro}{_\mathcal{B}^\mathcal{E}}.$
  Hence compactness of $v^*(N)$ implies compactness of $N$ in
  $D(\mathcal{B}\otimes_K \mathcal{A}^\opp)$ by
  Lemma~\ref{l:prod-from-source-cpt-implies-cpt}, i.\,e.\
  $N \in \per(\mathcal{B}\otimes_K \mathcal{A}^\opp).$
  This proves \ref{enum:diagonal-smooth-bimod-cpt}.
\end{proof}

\subsection{Quillen adjunction for dg categories}
\label{sec:quillen-adjunction-dg-categories}

Let $R \ra S$ be a morphism of graded commutative dg algebras.
As in Section~\ref{sec:model-struct-categ} we can consider the
categories $\dgcat_R$ and $\dgcat_S$ of small dg $R$- and dg
$S$-categories.
There are obvious extension and restriction of scalars functors
\begin{equation*}
  \xymatrix{
    {\dgcat_R}
    \ar@/^/[rr]^{(? \otimes_R S)}
    &&
    {\dgcat_S}
    \ar@/^/[ll]^{\res_R^S}
  }
\end{equation*}
and we have an obvious adjunction $((? \otimes_R S), \res_R^S, \phi).$
We sometimes write
$\mathcal{A}_S:= \mathcal{A} \otimes_R S,$ if $\mathcal{A}$ is a dg
$R$-category.

\begin{proposition}
  \label{p:base-change-dg-cats-quillen-adj-equi}
  The adjunction $((? \otimes_R S), \res_R^S, \phi)$ is a Quillen adjunction
  (where $\dgcat_R$ and $\dgcat_S$ are equipped with the model
  structure of Theorem~\ref{t:dgcatK-cofib-gen-model-cat}).

  It is a Quillen equivalence if and only if $R \ra S$ is a
  quasi-isomorphism.
\end{proposition}

\begin{proof}
  From the proof of Theorem~\ref{t:dgcatK-cofib-gen-model-cat} one obtains
  explicitly a set $I_R$ (resp.\ $I_S$) of generating cofibrations and
  a set $J_R$ (resp.\ $J_S$) of generating trivial cofibrations for
  the model structure on $\dgcat_R$ (resp.\ $\dgcat_S$).
  Then obviously the image of $I_R$ (resp.\ $J_R$) under
  $(? \otimes_R S)$ is precisely $I_S$ (resp.\ $J_S$),
  and
  the first statement follows from
  \cite[Lemma~2.1.20]{hovey-model-categories}.

  To prove the second statement, assume that $((? \otimes_R S), \res_R^S,
  \phi)$ is a Quillen equivalence.  
  Note that $R \in \dgcat_R$ is (semi-free and) cofibrant and that any
  object of $\dgcat_S,$ for example $S,$ is fibrant. 
  Since $\id: R \otimes_R S \ra S$ is a quasi-equivalence, the
  corresponding morphism $\phi(\id): R \ra \res_R^S(S)$ is a
  quasi-equivalence. This just means that $R \ra S$ is a
  quasi-isomorphism.

  Conversely, assume that $R \ra S$ is a quasi-isomorphism.
  Since any cofibrant dg $R$-category is $R$-h-flat 
  (Lemma~\ref{l:cofibrant-dg-K-cat-has-cofibrant-morphism-spaces})
  and any dg
  $S$-category is fibrant, it is enough to show the following claim:
  Let $\mathcal{A}$ be an $R$-h-flat dg $R$-category and $\mathcal{B}$
  a dg $S$-category. Then a morphism $f: \mathcal{A}\otimes_R S \ra
  \mathcal{B}$ is a quasi-equivalence if and only if $\phi(f):
  \mathcal{A} \ra \res^S_R(\mathcal{B})$ is a quasi-equivalence.

  Given $f: \mathcal{A}\otimes_R S \ra \mathcal{B},$ the morphism
  $\phi(f)$ is the obvious composition 
  \begin{equation}
    \label{eq:phi-of-f}
    \mathcal{A} 
    \ra
    \res^S_R (\mathcal{A} \otimes_R S) \xra{\res^S_R(f)} \res^S_R(\mathcal{B}).
  \end{equation}
  The first morphism of this composition 
  can be viewed as the morphism
  $\mathcal{A} \otimes_R R \ra
  \mathcal{A} \otimes_R S$
  obtained from $R \ra S$; it satisfies 
  \ref{enum:quasi-equi-H0-essentially-epi} for trivial reasons and
  \ref{enum:quasi-equi-qiso-homs} since $\mathcal{A}$ is $R$-h-flat.
  The second morphism is a quasi-equivalence if and only if $f$ is a
  quasi-equivalence. Hence the 2-out-of-3-property proves our claim.
\end{proof}

\subsection{Smoothness and base change}
\label{sec:smoothn-base-change}

We continue the discussion in Section
\ref{sec:quillen-adjunction-dg-categories} and keep the assumptions there.

Let $Q:\dgcat_R \ra \dgcat_R$ be a fixed cofibrant replacement functor.
If $\mathcal{A}$ is a dg $R$-category, then $Q(\mathcal{A})$ is
cofibrant and we have a trivial fibration $Q(\mathcal{A}) \ra
\mathcal{A}$ (= a cofibrant resolution) which is natural in $\mathcal{A}.$

The map $\mathcal{A} \mapsto Q(\mathcal{A})_S$ defines a
functor $\dgcat_R \ra \dgcat_S.$
It maps weak equivalences to weak equivalences 
and hence induces the following functor between homotopy categories,
\begin{align*}
  (? \otimes_R^L S): \HoMC(\dgcat_R) & \ra \HoMC(\dgcat_S),\\
  \mathcal{A} & \mapsto 
  \mathcal{A} \otimes_R^L S = Q(\mathcal{A})_S = Q(\mathcal{A})\otimes_R S.
\end{align*}

\begin{remark}
  \label{rem:how-to-compute-base-change}
  We explain how $R$-h-flatness may help when computing $\mathcal{A} \otimes_R^L S.$
  Let $\mathcal{A}' \ra \mathcal{A}$ be an $R$-h-flat resolution, for
  example a cofibrant resolution
  (Lemma~\ref{l:cofibrant-dg-K-cat-has-cofibrant-morphism-spaces}).
  Then
  $Q(\mathcal{A}) \ra \mathcal{A}$ factors as a quasi-equivalence
  $Q(\mathcal{A}) \ra \mathcal{A}'$ followed by $\mathcal{A}' \ra \mathcal{A}$
  (Lemma~\ref{l:lift-cofib-reso-to-flat-reso}). 
  Consider the commutative diagram
  \begin{equation*}
    \xymatrix{
      {\mathcal{A}} \gar[d] 
      &
      {\mathcal{A}'} \ar[r] \ar[l] &
      {\mathcal{A}'_S} \ar@{}[r]|-{=} & 
      {\mathcal{A}' \otimes_R S} \\
      {\mathcal{A}} &
      {Q(\mathcal{A})} \ar[r] \ar[l] \ar[u] &
      {Q(\mathcal{A})_S} \ar@{}[r]|-{=} \ar[u] & 
      {Q(\mathcal{A}) \otimes_R S} \\
      {R} \ar[u] \ar@{}[r]|-{=} &
      {R} \ar[u] \ar[r] &
      {S} \ar[u]
    }
  \end{equation*}
  whose lower three vertical arrows are merely symbolical.
  Note that $Q(\mathcal{A})_S \ra \mathcal{A}'_S$ is a
  quasi-equivalence
  (use
  Lemma~\ref{l:qisos-between-Khflats-remain-qisos}). Hence
  $R$-h-flat-resolutions are enough for 
  computing the base change $Q(\mathcal{A})_S$ up to
  quasi-equivalence. 
  Similarly, if $S$
  is $R$-h-flat, then we can replace $\mathcal{A}'$ by $\mathcal{A}$ in
  the above diagram (and $\mathcal{A}' \ra \mathcal{A}$ by
  $\id_\mathcal{A}$) and obtain a quasi-equivalence 
  $Q(\mathcal{A})_S \ra \mathcal{A}_S$
  (Lemma~\ref{l:tensor-with-K-h-flat-preserves-qequis}). 
  These observations can be used for testing $S$-smoothness of
  $Q(\mathcal{A})_S$ (Lemma~\ref{l:smoothness-and-quasi-equis}).
\end{remark}

\begin{theorem}
  \label{t:smoothness-and-base-change}
  Let $R \ra S$ be a morphism of graded commutative dg
  algebras and let $\mathcal{A}$ be a (small) 
  dg $R$-category.
  \begin{enumerate}[label=(BC{\arabic*})]
  \item 
    (Smoothness and base change)
    \label{enum:smoothness-and-base-change-preserved}
    If $\mathcal{A}$ is
    $R$-smooth, then $Q(\mathcal{A})_S$ is $S$-smooth.
  \end{enumerate}
  Now assume that $R \ra S$ is a quasi-isomorphism.
  \begin{enumerate}[resume,label=(BC{\arabic*})]
  \item 
    \label{enum:smoothness-and-qiso-base-change-prod}
    $\mathcal{A}$ is $R$-smooth if and only if $Q(\mathcal{A})_S$
    is $S$-smooth. 
  \item 
    \label{enum:smoothness-and-qiso-base-change-res}
    If $\mathcal{B}$ is a (small) 
    dg $S$-category, then
    $\mathcal{B}$ is $S$-smooth if and only if 
    $\res^S_R(\mathcal{B})$ 
    is $R$-smooth. 
  \end{enumerate}
\end{theorem}

\begin{proof}
  We need some preparations. We abbreviate $\res:= \res^S_R.$

  \textbf{Step 1:}
    Let $\mathcal{T}$ be a dg $R$-category. 
    The adjunction morphism
    \begin{equation}
      \label{eq:adjunction-T-to-res-TS}
      \mathcal{T} \ra \res (\mathcal{T}_S)
    \end{equation}
    is a morphism in $\dgcat_R$ and gives rise to the functor
    \begin{equation}
      \label{eq:smoothness-and-base-change-CT-CTS}
      \pro_\mathcal{T}^{\mathcal{T}_S}:=\pro_\mathcal{T}^{\res
        \mathcal{T}_S}: C(\mathcal{T}) \ra  
      C(\res(\mathcal{T}_S)) = C(\mathcal{T}_S)
    \end{equation}
    where the equality is the canonical identification explained in
    Remark~\ref{rem:dependence-on-K}.
    Explicitly, a dg $\mathcal{T}$-module $M$ is mapped to the dg
    $\mathcal{T}_S$-module $M_S:= \pro_{\mathcal{T}}^{\mathcal{T}_S}(M)$
    which is given by 
    \begin{equation*}
      M_S(T) = M(T) \otimes_R S
    \end{equation*}
    at $T \in \mathcal{T}_S$ and has the obvious action morphisms.
    On the level of derived categories we obtain the functor
    \begin{equation}
      \label{eq:smoothness-and-base-change-DT-DTS}
      L\pro_\mathcal{T}^{\mathcal{T}_S}: D(\mathcal{T}) \ra 
      D(\res \mathcal{T}_S) = D(\mathcal{T}_S),
    \end{equation}
    mapping $M$ to $p(M)_S,$
    which preserves compact objects
    (as explained in Section~\ref{sec:restr-prod}).

    \textbf{Step 2:}
    Let $\mathcal{A},$ $\mathcal{B}$ be dg $R$-categories.
    Then 
    \eqref{eq:smoothness-and-base-change-CT-CTS}
    applied to $\mathcal{T}=\mathcal{A} \otimes_R \mathcal{B}^\opp$ yields a
    functor
    \begin{equation*}
      \pro_{\mathcal{A}\otimes_R
        \mathcal{B}^\opp}^{\mathcal{A}_S\otimes_S
        (\mathcal{B}_S)^\opp}: C(\mathcal{A}\otimes_R \mathcal{B}^\opp) \ra 
      C(\mathcal{A}_S \otimes_S (\mathcal{B}_S)^\opp).
    \end{equation*}
    Explicitly, let $X$ be a dg $\mathcal{A} \otimes_R
    \mathcal{B}^\opp$ module. Then 
    $X_S$
    is
    given on $(A,B) \in
    \mathcal{A}_S \otimes_S (\mathcal{B}_S)^\opp$ by
    \begin{equation*}
      X_S(A,B) = X(A,B) \otimes_R S
    \end{equation*}
    with obvious action morphisms.
    In particular, for $\mathcal{A}=\mathcal{B}$ this shows that the
    diagonal bimodule $\mathcal{A}$ is mapped to the diagonal
    bimodule $\mathcal{A}_S.$
    We need a similar statement on the level of derived categories.

    \textbf{Step 3:}
    Let $\mathcal{A},$ $\mathcal{B}$ and $X$ be as above, but assume
    in addition that $\mathcal{A}$ and $\mathcal{B}$ have cofibrant
    morphism spaces and
    that
    $X(A,B)$ is $R$-h-flat for all $A \in \mathcal{A},$ $B \in
    \mathcal{B}^\opp.$
    Let $\gamma: p(X) \ra X$ be a cofibrant resolution
    of $X$ in
    $C(\mathcal{A}\otimes_R \mathcal{B}^\opp).$ 
    We claim that
    $p(X)_S \ra X_S$ is a quasi-isomorphism.     
    By 
    Proposition~\ref{p:cofibrant-bimodule-remains-cofibrant-if-insert-cofib-cat}
    all $p(X)(A,B)$ are cofibrant dg $R$-modules and in particular
    $R$-h-flat.
    Hence, by Lemma~\ref{l:qisos-between-Khflats-remain-qisos},
    \begin{equation*}
      (p(X)(A,B))\otimes_R S \xra{\gamma(A,B) \otimes_R \id_S}
      X(A,B) \otimes_R S
    \end{equation*}
    is a quasi-isomorphism for all $A \in \mathcal{A},$ $B \in
    \mathcal{B},$ proving our claim.
    Hence
    \begin{equation*}
      L\pro_{\mathcal{A}\otimes_R
        \mathcal{B}^\opp}^{\mathcal{A}_S\otimes_S
        (\mathcal{B}_S)^\opp}(X) = p(X)_S \sira X_S
    \end{equation*}
    in $D(\mathcal{A}_S \otimes_S (\mathcal{B}_S)^\opp).$
    
\textbf{Step 4:}
    Assume that $\tildew{\mathcal{A}}$ is a cofibrant dg $R$-category.
    Then $\tildew{\mathcal{A}}$ has cofibrant and $K$-h-flat morphism spaces
    (Lemma~\ref{l:cofibrant-dg-K-cat-has-cofibrant-morphism-spaces}),
    so we can apply
    Step 3
    to the diagonal bimodule $X:= \tildew{\mathcal{A}}.$ 
    This implies that
    \begin{equation}
      \label{eq:smoothness-and-base-change-DAAopp-DASAoppS}
      L\pro_{\tildew{\mathcal{A}}\otimes_R
        \tildew{\mathcal{A}}^\opp}^{\tildew{\mathcal{A}}_S\otimes_S
        (\tildew{\mathcal{A}}_S)^\opp}: 
      D(\tildew{\mathcal{A}}\otimes_R \tildew{\mathcal{A}}^\opp) \ra 
      D(\tildew{\mathcal{A}}_S \otimes_S (\tildew{\mathcal{A}}_S)^\opp)
    \end{equation}
    maps the diagonal bimodule $\tildew{\mathcal{A}}$ to an object isomorphic
    to the diagonal bimodule $\tildew{\mathcal{A}}_S.$

  Now we can prove our claims.

  \ref{enum:smoothness-and-base-change-preserved}:
  Assume that $\mathcal{A}$ is $R$-smooth. Then
  by definition the diagonal bimodule $Q(\mathcal{A})$ is compact.
  Hence from 
  Step 4
  with
  $\tildew{\mathcal{A}} = Q(\mathcal{A})$ (and the
  fact that the functor
  in \eqref{eq:smoothness-and-base-change-DAAopp-DASAoppS} preserves
  compact objects) we see that the diagonal bimodule
  $Q(\mathcal{A})_S$ is compact. 
  This is equivalent to $Q(\mathcal{A})_S$ being
  $S$-smooth (note that $Q(\mathcal{A})_S$ is a cofibrant dg
  $S$-category, as follows from
  Proposition~\ref{p:base-change-dg-cats-quillen-adj-equi} and the fact that
  any left Quillen functor preserves cofibrant objects).

  Assume now that $R \ra S$ is a quasi-isomorphism. 

  \ref{enum:smoothness-and-qiso-base-change-prod}:
  If $\mathcal{T}$ is an $R$-h-flat dg $R$-category,
  \eqref{eq:adjunction-T-to-res-TS} is a quasi-equivalence 
  (cf.\ the explanation below \eqref{eq:phi-of-f})
  and 
  \eqref{eq:smoothness-and-base-change-DT-DTS} is an equivalence
  (cf.\ before \eqref{eq:restriction-equi-per-along-quequi}).
  The cofibrant dg $R$-category $Q(\mathcal{A})$ is $R$-h-flat
  (Lemma~\ref{l:cofibrant-dg-K-cat-has-cofibrant-morphism-spaces}).
  But then also $Q(\mathcal{A}) \otimes_R Q(\mathcal{A})$ is
  $R$-h-flat, and hence
  \eqref{eq:smoothness-and-base-change-DAAopp-DASAoppS} for
  $\tildew{\mathcal{A}}=Q(\mathcal{A})$ is an equivalence mapping the
  diagonal bimodule $Q(\mathcal{A})$ to (an object isomorphic to)
  the diagonal bimodule $Q(\mathcal{A})_S.$ 
  This means that $\mathcal{A}$ is smooth if and only if
  $Q(\mathcal{A})_S$ is smooth.

  \ref{enum:smoothness-and-qiso-base-change-res}:
  Let $\mathcal{B}$ be a dg $S$ category and consider the cofibrant resolution 
  $Q(\res(\mathcal{B})) \ra \res(\mathcal{B}).$
  Then $Q(\res(\mathcal{B}))_S \ra \mathcal{B}$ is a 
  quasi-equivalence, since 
  $((? \otimes_R S), \res_R^S, \phi)$ is a Quillen equivalence 
  (Prop.~\ref{p:base-change-dg-cats-quillen-adj-equi}).
  (It is even a cofibrant resolution; use
  \ref{enum:surj-on-homs-surj-qiso} and \ref{enum:surj-epi-on-sets-of-objects}.)
  Now use
  Lemma~\ref{l:smoothness-and-quasi-equis}
  and \ref{enum:smoothness-and-qiso-base-change-prod}.
\end{proof}

\subsection{Locally perfect categories and smoothness}
\label{sec:locally-perf-categories-and-smoothness}

We extend some presumably well-known results
(cf.\ e.\,g.\ part of the proof of \cite[Lemma~4.1]{keller-CY}, or
\cite[Prop.~3.4]{shklyarov-serre-duality-cpt-smooth-arXiv})
to the dg $K$-setting.
This section may be skipped: only
Corollary~\ref{c:smooth-over-field-and-perfect} 
is used later on in one of the two proofs of
Proposition~\ref{p:homologically-positive-dg-algebra-plus-conditions-not-smooth}.

\begin{definition}
  [{cf.\ \cite[Def.~2.4]{toen-vaquie-moduli-objects-dg-cats}}]
  \label{d:locally-perfect}
  A dg $K$-category $\mathcal{A}$ is 
  \textbf{locally $K$-perfect} (or \textbf{locally $K$-proper}) if
  $\mathcal{A}(A,A')$ is a
  compact dg $K$-module (i.\,e.\ in $\per(K)$) for
  all $A, A' \in \mathcal{A}.$
\end{definition}

It is easy to show that local $K$-perfectness is invariant under
quasi-equivalences.

\begin{definition}
  \label{def:K-proper}
  Let $\mathcal{A}$ be dg $K$-category.  
  A dg $\mathcal{A}$-module $M$ is called
  \define{locally $K$-perfect} (or \define{locally $K$-proper}) if $M(A)$ is
  compact when considered as an
  object of $D(K),$ for all $A \in \mathcal{A}.$
  The full subcategory of $D(\mathcal{A})$ consisting of
  locally $K$-perfect dg $\mathcal{A}$-modules is denoted by
  $D_{lp}(\mathcal{A}).$ 
\end{definition}

Clearly, local $K$-perfectness of dg $\mathcal{A}$-modules is
invariant under isomorphisms in $D(\mathcal{A}).$

\begin{lemma}
  \label{l:K-proper-objects-and-perfect-category-versus-properness}
  Let $\mathcal{A}$ be a dg $K$-category. 
  Then $\mathcal{A}$ is locally $K$-perfect
  if and only if $\per(\mathcal{A}) \subset D_{lp}(\mathcal{A}).$
\end{lemma}

\begin{proof}
  In general, $D_{lp}(\mathcal{A})$ is a strict full triangulated
  subcategory of $D(\mathcal{A})$ and closed under summands.
  Hence it contains $\per(\mathcal{A})$ if and only if it contains all
  $\Yoneda{A},$ for $A \in \mathcal{A}$;
  this condition just means that $\mathcal{A}$ is locally $K$-perfect.
\end{proof}

\begin{lemma}
  \label{l:restriction-along-quequi-and-lp-modules}
  If $F: \mathcal{B} \ra \mathcal{A}$ is a
  quasi-equivalence of dg $K$-categories, 
  then restriction along $F$ induces an equivalence
  $\res^\mathcal{A}_\mathcal{B}:D_{lp}(\mathcal{A}) \ra
  D_{lp}(\mathcal{B}).$
\end{lemma}

\begin{proof}
  We know from Section~\ref{sec:restr-prod} that 
  $\res^\mathcal{A}_\mathcal{B}:D(\mathcal{A}) \sira
  D(\mathcal{B})$ is an equivalence.
  If $M$ is a locally $K$-perfect dg $\mathcal{A}$-module, then
  obviously $\res^\mathcal{A}_\mathcal{B}(M)$ is locally $K$-perfect.
  Since local $K$-perfectness is invariant under isomorphisms in
  $D(\mathcal{B})$ it is enough to show that the converse is also true.
  
  Let $A \in \mathcal{A}.$ Then there is an object $B \in \mathcal{B}$
  such that $A$ and $F(B)$ are isomorphic in $[\mathcal{A}].$
  Application of $[M]: [\mathcal{A}^\opp] \ra
  \mathcal{H}(\mathcal{A})$
  shows that $M(A)$ and $M(F(B))$ are isomorphic in
  $\mathcal{H}(\mathcal{A})$ and a fortiori in $D(\mathcal{A}).$ 
  Hence if $\res^\mathcal{A}_\mathcal{B}(M)$ is locally $K$-perfect,
  then $M(F(B))$ and hence $M(A)$ are compact in $D(K).$ This implies
  that $M$ is locally $K$-perfect.
\end{proof}

\begin{proposition}
  \label{p:K-proper-objects-and-perfect-category-for-smooth-and-flat}
  Let $\mathcal{A}$ be a dg $K$-category. If $\mathcal{A}$ is
  $K$-smooth, then 
  $D_{lp}(\mathcal{A}) \subset \per(\mathcal{A}).$
\end{proposition}

\begin{proof}
  By 
  Section~\ref{sec:restr-prod} 
  and 
  Lemma~\ref{l:restriction-along-quequi-and-lp-modules} 
  we know that
  restriction 
  along a $K$-h-flat
  resolution $\tildew{\mathcal{A}} \ra 
  \mathcal{A}$ induces an equivalence
  $\res_{\tildew{\mathcal{A}}}^{\mathcal{A}}:
  D(\mathcal{A}) \ra D(\tildew{\mathcal{A}})$
  identifying $\per(\mathcal{A})$
  with $\per(\tildew{\mathcal{A}})$ 
  and $D_{lp}(\mathcal{A})$ with $D_{lp}(\tildew{\mathcal{A}}).$
  Hence it is enough to prove the claim under the additional
  assumption that $\mathcal{A}$ is $K$-h-flat.
  
  We start with some preparations.
  If $A$ is an object of a dg $K$-category $\mathcal{A},$ there is a
  unique dg $K$-functor $K \ra \mathcal{A}$ (where $K$ is viewed as a dg $K$-algebra)
  mapping the unique object of $K$ to $A.$ We
  denote this functor by $K=K_A \ra 
  \mathcal{A}$ to indicate its dependence on $A.$

  Let $\mathcal{A},$ $\mathcal{B}$ be dg $K$-categories and let $M$
  be a dg $\mathcal{A}$-module.
  Consider the dg $K$-functor
  \begin{align*}
    F_M:\calMod(\mathcal{B}\otimes_K \mathcal{A}^\opp) & \ra
    \calMod(\mathcal{B}),\\
    X & \mapsto M \otimes_\mathcal{A} X,
  \end{align*}
  and let $LF_M:D(\mathcal{B}\otimes_K \mathcal{A}^\opp) \ra
  D(\mathcal{B}),$ $X \mapsto M \otimes_\mathcal{A} p(X),$ be its left
  derived functor.

  If $X$ is a cofibrant dg $\mathcal{B} \otimes_K
  \mathcal{A}^\opp$-module, then $p(X) \ra X$ 
  is a quasi-isomorphism between cofibrant objects 
  in $C(\mathcal{B} \otimes_K \mathcal{A}^\opp)$ 
  and an isomorphism in $\mathcal{H}(\mathcal{B} \otimes_K
  \mathcal{A}^\opp),$ hence 
  $LF_M(X) \sira F_M(X)$ in $D(\mathcal{B}).$

  Let $B \in \mathcal{B},$ $A \in \mathcal{A}^\opp.$ 
  It is easy to see that
  \begin{equation*}
    F_M(\Yoneda{(B,A)}) = (M(A))\otimes_K\Yoneda{B}=
    (\res^\mathcal{A}_{K_A}(M)) \otimes_K \Yoneda{B}
  \end{equation*}
  in $C(\mathcal{B}).$ Since $\Yoneda{(B,A)}$ is cofibrant we have
  \begin{equation}
    \label{eq:LFM-resKYonedaB}
    LF_M(\Yoneda{(B,A)}) \sira
    (\res^\mathcal{A}_{K_A}(M)) \otimes_K \Yoneda{B}
  \end{equation}
  in $D(\mathcal{B}).$ 
  
  Let $\pro_{K_B}^{\mathcal{B}}: \calMod(K) \ra
  \calMod(\mathcal{B})$ be the extension of scalars functor along $K=K_B
  \ra \mathcal{B}.$ Obviously $\pro_{K_B}^\mathcal{B}(N)=N\otimes_K
  \Yoneda{B}.$
  Assume now that $\mathcal{B}$ is $K$-h-flat. Then 
  $\pro_{K_B}^\mathcal{B}$ preserves quasi-isomorphism. 
  Hence $p(\res^\mathcal{A}_{K_A}(M)) \ra \res^\mathcal{A}_{K_A}(M)$
  yields an isomorphism
  \begin{equation}
    \label{eq:LproKB-resKYonedaB}
    L\pro_{K_B}^{\mathcal{B}} (\res^\mathcal{A}_{K_A}(M)) \sira
    (\res^\mathcal{A}_{K_A}(M)) \otimes_K \Yoneda{B}    
  \end{equation}
  in $D(\mathcal{B}).$

  Assume in addition that $M \in D_{lp}(\mathcal{A}).$ Since
  $L\pro_{K_B}^\mathcal{B}$ preserves compact objects,
  \eqref{eq:LFM-resKYonedaB} and \eqref{eq:LproKB-resKYonedaB} show
  that
  $LF_M(\Yoneda{(B,A)}) \in \per(\mathcal{B}).$

  This implies that $LF_M$ induces a functor
  \begin{equation*}
    LF_M:\per(\mathcal{B}\otimes_K \mathcal{A}^\opp) \ra
    \per(\mathcal{B})
  \end{equation*}
  (under the assumptions that $\mathcal{B}$ is $K$-h-flat and $M$ is
  locally $K$-perfect). 

  Now assume that $\mathcal{A}$ is $K$-h-flat and $K$-smooth.
  Then $\mathcal{A} \in \per(\mathcal{A}\otimes_K \mathcal{A}^\opp)$
  by Lemma~\ref{l:tfae-smoothness}. 
  The above arguments applied to $\mathcal{B}=\mathcal{A}$ show:
  If $M$ is a locally $K$-perfect dg $\mathcal{A}$-module, then
  \begin{equation*}
    LF_M(\mathcal{A}) \in \per(\mathcal{A}).
  \end{equation*}

  We now prove $D_{lp}(\mathcal{A}) \subset
  \per(\mathcal{A}).$
  Let $N \in D_{lp}(\mathcal{A}).$ 
  Then $p(N)$ is locally $K$-perfect and we obtain
  \begin{equation*}
    \per(\mathcal{A}) \ni LF_{p(N)}(\mathcal{A}) 
    = p(N) \otimes_\mathcal{A} p(\mathcal{A})
  \end{equation*}
  Since $p(\mathcal{A}) \ra \mathcal{A}$ is a quasi-isomorphism
  and $p(N)$ 
  is $\mathcal{A}$-h-flat
  (Lemma~\ref{l:cofibrant-dg-A-module-is-A-h-flat})
  we obtain
  \begin{equation*}
    p(N) \otimes_\mathcal{A} p(\mathcal{A})
    \sira
    p(N) \otimes_\mathcal{A} \mathcal{A} = p(N)
    \sira N
  \end{equation*}
  in $D(\mathcal{A}).$
  This shows $N \in \per(\mathcal{A}).$
\end{proof}

\begin{corollary}
  \label{c:K-proper-objects-and-perfect-category-for-smooth-and-proper-and-flat}
  Let $\mathcal{A}$ be a dg $K$-category that is 
  $K$-smooth and locally $K$-perfect. Then
  \begin{equation*}
    D_{lp}(\mathcal{A}) = \per(\mathcal{A}).
  \end{equation*}
\end{corollary}

\begin{proof}
  Follows from
  Proposition~\ref{p:K-proper-objects-and-perfect-category-for-smooth-and-flat}
  and
  Lemma~\ref{l:K-proper-objects-and-perfect-category-versus-properness}.
\end{proof}

\begin{corollary}
  [{cf.~\cite[Prop.~3.4]{shklyarov-serre-duality-cpt-smooth-arXiv}}]
  \label{c:smooth-over-field-and-perfect}
  Let $A$ be a dg $\groundring$-algebra over a field $\groundring,$ and let $M$ be a dg $A$-module.
  \begin{enumerate}
  \item 
    \label{enum:smooth-fin-dim-then-perfect}
    If $A$ is $\groundring$-smooth, then $\dim_\groundring H(M)<\infty$ implies $M \in
    \per(A).$
  \item 
    If $A$ is $\groundring$-smooth and $H(A)$ is finite dimensional, then
    $\dim_\groundring H(M) < \infty$ if and only if $M \in \per(A).$
  \end{enumerate}
\end{corollary}

\begin{proof}
  Obviously, $M \in D_{lp}(A)$ if and only if $M \in \per(\groundring)$ if and
  only if $H(M)$ is finite dimensional.
  Similarly, $A$ is locally perfect over $\groundring$ if and only if $H(A)$ is
  finite dimensional.
  Now use
  Proposition~\ref{p:K-proper-objects-and-perfect-category-for-smooth-and-flat}
  and
  Corollary~\ref{c:K-proper-objects-and-perfect-category-for-smooth-and-proper-and-flat}.
\end{proof}

\subsection{Smoothness criteria}
\label{sec:smoothness-criteria}

In this section $\groundring$ will be a field (viewed as a dg ring
concentrated in degree zero).
Our aim is to prove the two smoothness criteria provided by
Propositions~\ref{p:homologically-positive-dg-algebra-plus-conditions-not-smooth}
and
\ref{p:smoothness-over-local-graded-finite-homological-dim-algebras} below.
The latter proposition will be essential for
Section~\ref{sec:smoothn-equiv-derived-cats}.

We need some preparations for the proof of the first criterion.

\begin{lemma}
  [{\cite[Lemma~9.5]{ELO-defo-theory-I}}]
  \label{l:cohomology-positive-qiso-to-positive}
  Let $\groundring$ be a field and $A$ a dg $\groundring$-algebra. Assume
  that
  $H^0(A)=\groundring$ and $H^i(A)=0$ for all $i<0.$ 
  Then there is a dg $\groundring$-subalgebra $U$ of $A$ such that
  $U^0=\groundring,$ $U^i=0$ for all $i<0,$ and the inclusion $U
  \hra A$ is a quasi-isomorphism.
\end{lemma}

\begin{proof}
  Let $C \subset Z^1(A)$ be a linear subspace
  such that the restriction of $Z^1(A) \ra H^1(A)$ to $C$ is an isomorphism,
  and let $B \subset A^1$ be a linear subspace such that $d:A^1 \ra
  A^2$ induces an isomorphism $B \sira d(A^1).$ 
  Define $U^i:=0$ for $i<0,$ $U^0:=\groundring,$ $U^1:= C\oplus B,$ $U^i:= A^i$
  for $i>1,$ and take $U:=\bigoplus U^i.$
\end{proof}

If $N=\bigoplus_{i \in \DZ} N^i$ is a graded abelian group which is
bounded 
and nonzero, we define
its amplitude
by
\begin{equation*}
  \ampl(N):= 
  \max\{i \in \DZ \mid N^i \not=0\}
  -\min\{i \in \DZ \mid N^i\not=0\}.
\end{equation*}

\begin{lemma}
  \label{l:cohomology-range-of-perfects-over-pos-dg-algs}
  Let $\groundring$ be a field
  and let $A$ be a dg $\groundring$-algebra such that $H^i(A)=0$ for all $i<0$ and
  $H^0(A)=\groundring.$
  Assume that $H(A)$ is bounded above and let $m \in \DN$ be maximal
  such that $H^m(A)\not=0$; assume that $m>0$ (i.\,e. $\groundring\not=H(A)$).
  Let $M \in \per(A)$ with $M\not\cong 0.$ Then 
  $H(M)$ is bounded and nonzero and
  \begin{equation*}
    \ampl(H(M)) \geq m.
  \end{equation*}
  Moreover, $\dim_\groundring H^{\op{top}}(M) \geq \dim_\groundring H^m(A),$ 
  where $H^{\op{top}}(M)$ is the highest non-vanishing cohomology of $M.$
  In particular, $M$ has at least two nonvanishing cohomology
  groups.
\end{lemma}

\begin{proof}
  Let $U \subset A$ be as in
  Lemma~\ref{l:cohomology-positive-qiso-to-positive}.
  Then $\res_U^A$ induces an equivalence $\per(A) \sira \per(U)$
  preserving cohomology. Hence by replacing $A$ by $U$ we can assume
  in addition that $A^i=0$ 
  for all $i<0$ and that $A^0=\groundring.$ 

  Let $M \in \per(A),$ $M \not\cong 0.$
  Since $A$ is positively graded with $A^0=\groundring\subset A$ as a dg
  subalgebra, we can assume that $M$ has a
  finite filtration $0=F_0 \subset F_1 \subset \dots \subset F_n=M$ in
  $C(A)$ such that $F_i/F_{i-1} \cong [l_i] A$ for suitable $l_1 \geq
  l_2 \geq \dots \geq l_n$
  (this follows from \cite[Thm.~1]{OSdiss-perfect-dg-acs};
  see Rem.~9 there for a picture). 

  Obviously $M^i=0$ for $i < -l_1,$ and $1 \in A^0=([l_1]A)^{-l_1}
  \cong {F_1}^{-l_1} \subset M^{-l_1}$ is a cocycle in $M$ that defines
  an nonzero element of $H^{-l_1}(M).$ 

  We prove by induction on the length $n \geq 1$ of the filtration
  that there is a surjection
  $H^{-l_n+m}(M) \ra H^m(A) \not= 0$ and 
  that $H^{-l_n+m+i}(M)$ vanishes for $i>0.$
  For $n=1$ this is obviously true.
  Let $n \geq 2$ and assume that the claim is true for $n-1.$
  Consider the short exact sequence
  \begin{equation*}
    0 \ra F_{n-1} \hra M \sra [l_n]A \ra 0.
  \end{equation*}
  For $i>0$ we have obviously $H^{-l_n+m+i}([l_n]A)=H^{m+i}(A)=0$ and
  by induction 
  $H^{-l_n+m+i}(F_{n-1})=
  H^{-l_{n-1}+m+(i+l_{n-1}-l_n)}(F_{n-1})=0$ since $l_{n-1}-l_n\geq 0.$
  The long exact cohomology sequence then proves that
  $H^{-l_n+m+i}(M)=0$ for $i>0$ and that 
  \begin{equation*}
    H^{-l_n+m}(M) \ra H^{-l_n+m}([l_n]A)= H^m(A)\not= 0
  \end{equation*}
  is surjective.

  We have proved that the lowest (resp.~highest) cohomology of $M$
  lives in degree $-l_1$ (resp.~$-l_n+m$). Hence
  $\ampl(H(M))= -l_n+m+l_1 \geq m.$
\end{proof}

\begin{proposition}
  \label{p:homologically-positive-dg-algebra-plus-conditions-not-smooth}
  Let $\groundring$ be a field and let $A$ be a dg $\groundring$-algebra such that $H^i(A)=0$
  for all $i<0,$ 
  $H^0(A)=\groundring,$ and $H(A)$ is bounded above.
  Then $A$ is $\groundring$-smooth if and only if $\groundring=H(A).$
\end{proposition}

\begin{proof}
  Recall that smoothness is invariant under quasi-isomorphisms of dg
  $\groundring$-algebras (Lemma~\ref{l:smoothness-and-quasi-equis}).
  If $\groundring=H(A)$ then $\groundring \ra A$ is a quasi-isomorphism and hence $A$ is
  $\groundring$-smooth.

  Assume that $\groundring\not= H(A).$ We give two proofs showing that $A$ is
  not $\groundring$-smooth.

  First proof:
  Lemma
  \ref{l:cohomology-positive-qiso-to-positive} shows that we can
  assume in addition that $A^i=0$ for all $i<0$ and that $A^0=\groundring.$
  Then it is obvious that $A$ has a one dimensional
  "augmentation module" $\groundring.$ 
  If $A$ is $\groundring$-smooth, 
  part~\ref{enum:smooth-fin-dim-then-perfect} of
  Corollary~\ref{c:smooth-over-field-and-perfect} implies that $\groundring
  \in \per(A).$ This contradicts
  Lemma~\ref{l:cohomology-range-of-perfects-over-pos-dg-algs}.

  Second (easier) proof:
  Since we work over a field, smoothness of $A$ is equivalent to $A$
  being in $\per(A\otimes_\groundring A^\opp).$
  Let $r \in \DZ$ be maximal such that $H^r(A)\not=0.$ By assumption
  $0 < r < \infty.$
  Since $H(A\otimes_\groundring A^\opp) \cong H(A)
  \otimes_\groundring H(A^\opp)$ (at least as graded $\groundring$-modules), 
  $A\otimes_\groundring A^\opp$ satisfies the assumptions of
  Lemma~\ref{l:cohomology-range-of-perfects-over-pos-dg-algs}, and
  $m=2r$ is the degree of the highest nonzero cohomology group of
  $A\otimes_\groundring A^\opp.$ 
  The assumption $A \in \per(A\otimes_\groundring A^\opp)$ 
  yields the contradiction $r=\ampl(H(A)) \geq 2r.$ Hence $A$ is not
  $\groundring$-smooth.
\end{proof}

Before we can give the proof of the second criterion, 
Proposition~\ref{p:smoothness-over-local-graded-finite-homological-dim-algebras},
we explain some preparatory results and introduce some notation.
Let 
\begin{equation*}
  M= (\ldots \ra M_i \xra{f_i} M_{i+1} \ra \dots)
\end{equation*}
be a complex in $C(\groundring),$ i.\,e.\ a complex of complexes in
$\groundring$-vector spaces. We associate with $M$ a double complex 
$\Dbl(M)$ 
(of $\groundring$-vector spaces)
defined as follows: 
\begin{align*}
  \Dbl(M)^{ij} & := M_i^j,\\
  d': \Dbl(M)^{ij} & \ra \Dbl(M)^{i+1,j}, & m & \mapsto f_i(m),\\
  d'': \Dbl(M)^{ij} & \ra \Dbl(M)^{i,j+1}, & m & \mapsto (-1)^{i}d_{M_i}(m).
\end{align*}
Let $(N, d', d'')$ be a double complex. 
For $l \in \DZ$ let $F^lN$ be the subcomplex such that $(F^lN)^{ij}=0$
if $i < l$ and $(F^lN)^{ij}=N^{ij}$ if $i \geq l.$
The $F^lN$ define a decreasing filtration $\dots \supset F^lN \supset F^{l+1}N
\supset \dots$ on $N.$
We define the total complex 
$\Tot(N)$ associated to $N,$ a complex of vector spaces, by
\begin{align*}
  \Tot(N)^{n} & := \bigoplus_{i+j=n} N^{ij},\\
  d: \Tot(N)^{n} & \ra \Tot(N)^{n+1},\\
  N^{ij} \ni m& \mapsto d'(m)+d''(m) \in N^{i+1,j}\oplus N^{i,j+1}
  \subset \Tot(N)^{n+1}.
\end{align*}
It has an induced filtration.
We define $\tot:= \Tot \comp \Dbl.$

Let $K$ be a dg $\groundring$-algebra. 
Given a complex $M$ in $C(K),$ i.\,e.\ a complex of dg $K$-modules,
then each "column"
$\Dbl(M)^{i*}$ of $\Dbl(M)$ is obviously a graded $K$-module 
(and
differential and $K$-module structure are related by $d''(mk)=d''(m) k +
(-1)^{i+j} m d_K(k)$ for $m \in M^j_i$ and $k \in K$) 
and $d':\Dbl(M)^{i*}\ra \Dbl(M)^{i+1,*}$ is
$K$-linear. It follows that $\tot(M)$ becomes a dg $K$-module which is
equipped with a decreasing filtration by dg $K$-submodules $F^l\tot(M).$
Moreover it is clear that $\Dbl$ and $\Tot$ are functorial.

\begin{lemma}
  \label{l:minimal-graded-projective-resolution-and-cofibrant-resolution}
  Let $\groundring$ be a field and $K$ a graded $\groundring$-algebra
  such that $K^i=0$ for $i<0,$ $K^0=\groundring.$
  Let $M$ be a graded (right) $K$-module which is bounded below,
  i.\,e.\ there is $m \in \DZ$ such that $M^i=0$ for all $i<m.$
  Then $M$
  has a "minimal" graded free resolution, i.\,e.\ there is a complex
  \begin{equation*}
    P= (\ldots \ra P_i \xra{p_i} P_{i+1} \ra \ldots \ra P_{-1} \xra{p_{-1}}
    P_0 \ra 0 \ra \ldots)
  \end{equation*}
  of graded free $K$-modules together with a quasi-isomorphism 
  $P \ra M,$ given by $p_0: P_0 \ra M$ such that
  \begin{enumerate}
  \item 
    \label{enum:mgp-and-cof-resolution-cofibres}
    the obvious morphism $\tot(P) \ra M$ (induced by $p_0$) is a
    cofibrant resolution of $M$
    in $C(K)$ 
    and is bounded below by $m,$ i.\,e.\ $\tot(P)^i=0$ for $i<m$
    (here we view $M=(M,d_M=0)$ and all $(P_i, d_{P_i}=0)$ as dg $K$-modules, where
    $K=(K,d_K=0)$);
  \item 
    \label{enum:mgp-and-cof-resolution-tensor-augm-vanishes}
    the differential in
    the complex $P\otimes_K \groundring$ in $C(\groundring)$ vanishes
    (and in particular the differential in $\tot(P)\otimes_K
    \groundring=\tot(P\otimes_K \groundring)$ vanishes); 
  \item 
    \label{enum:mgp-and-cof-resolution-noetherian}
    if $K$ is a (right) Noetherian ring and $M$ is a finitely
    generated $K$-module, then
    all $P_i$ are finitely generated $K$-modules;
  \item 
    \label{enum:mgp-and-cof-resolution-proj-dim}
    if $M$ has finite projective dimension $s$ as a $K$-module,
    then $P_i=0$ for all $i < -s.$
  \end{enumerate}
  In particular, if $K$ is Noetherian and $M$ is finitely generated
  and of finite projective dimension, then
  $\dim_\groundring H(M \otimes^L_{K} \groundring) < \infty,$ where 
  $M \otimes^L_{K} \groundring \in D(\groundring)$ is obtained by extension of scalars along
  the obvious (augmentation) morphism $K \ra \groundring$ of dg $\groundring$-algebras.
\end{lemma}

\begin{proof}
  We sometimes write $\ol{N}=N \otimes_K \groundring$ if $N$ is a graded
  $K$-module, and similarly for morphisms. 

  Since $P_0:=\ol{M}\otimes_{\groundring} K$ is a graded projective $K$-module, the
  obvious morphism $P_0 \ra \ol{M}$ factors through
  $M \ra \ol{M}$ to a morphism $p_0:P_0 \ra M.$
  Note that $\ol{p_0}:\ol{P_0} \ra \ol{M}$ is an 
  isomorphism; this implies in particular that $\ol{\Kokern p_0}=0,$
  hence $\Kokern p_0=0$ (since $\Kokern p_0$ is bounded below), and
  $p_0$ is surjective. Apply this method now 
  to the kernel of $p_0.$ By 
  induction we obtain a graded projective resolution
  \begin{equation*}
    \ldots \ra P_i \xra{p_i} P_{i+1} \ra \ldots \ra P_{-1} \xra{p_{-1}} P_0
    \xra{p_0} M \ra 0
  \end{equation*}
  of $M,$ such that all $\ol{p_i}$ for $i \leq -1$ vanish.
  This shows
  \ref{enum:mgp-and-cof-resolution-tensor-augm-vanishes}.
  Since the "vertical" differential $d''$ in $\Tot(P)$ vanishes it is
  obvious that $\tot(P) \ra M$ is a surjective quasi-isomorphism.
  There is an obvious filtration on $\tot(P)$ showing that $\tot(P)$
  is semi-free as a dg $K$-module and and hence cofibrant. 
  By construction $P_0$ is generated in degrees $\geq m$ and 
  $p_0: P_0 \ra M$ is an isomorphism in degrees $\leq
  m$; hence its kernel is generated in degrees $\geq m+1.$ So
  $P_{-1}$ is generated in degrees $\geq m+1$ and $p_{-1}$
  induces an isomorphism onto the kernel of $p_0$ in degrees $\leq
  m+1.$ By induction $P_i$ is generated in degrees $\geq m-i.$
  This implies that $\tot(P)$ is generated in degrees $\geq m.$
  These arguments show
  \ref{enum:mgp-and-cof-resolution-cofibres}.
  Claim \ref{enum:mgp-and-cof-resolution-noetherian} is obvious.

  Assume that $M$ has projective dimension $s.$
  Then $p_{-s}$ induces an
  isomorphism from $P_{-s}$ onto the kernel of $p_{-s+1}$
  (see 
  \cite[X.\S8.7, Prop.~8 (and
  Cor.~2)]{bourbaki-algebre-chap-10-algebre-homologique}). 
  Hence 
  $P_i=0$ for all $i<-s,$ proving
  \ref{enum:mgp-and-cof-resolution-proj-dim}.

  Assume that $K$ is Noetherian and $M$ is finitely generated and of finite projective
  dimension.
  Then $M \otimes_K^L \groundring$ is isomorphic 
  to the dg $\groundring$-module $\tot(P) \otimes_K \groundring=\tot(P \otimes_K \groundring)$
  with vanishing differential and finite (total) dimension.
\end{proof}

\begin{lemma}
  \label{l:acyclic-iff-derived-reduction-acyclic}
  Let $\groundring$ be a field and $K$ a graded $\groundring$-algebra
  such that $K^i=0$ for $i<0,$ $K^0=\groundring.$ We view $K$ as a dg
  ($\groundring$-)algebra with differential $d_K=0.$

  Let $M$ be a dg $K$-module that (or whose cohomology) is
  concentrated in degrees $\geq m$ for some $m \in \DZ.$
  Then 
  $H(M \otimes_K^L \groundring)$ is concentrated in degrees $\geq m,$ and
  $M \otimes_K^L \groundring$ is acyclic if and only if $M$ is acyclic.
\end{lemma}

\begin{proof}
  If $M$ is acyclic it is clear that $M \otimes_K^L \groundring$ is acyclic.

  If the cohomology of $M$ is concentrated in degrees $\geq m,$
  replace $M$ by the dg $K$-submodule defined as follows
  (cf.\ Lemma~\ref{l:cohomology-positive-qiso-to-positive}):
  It is zero in all degrees $<m,$ coincides with $M$ in all degrees
  $>m,$ and in degree $m$ it is the direct sum of a subspace of $Z(M)^m$ that goes 
  isomorphically onto $H^m(M)$ and a subspace of $M^m$ that goes
  isomorphically onto $B^{m+1}(M).$

  Hence we can assume without loss of generality that $M$ is
  concentrated in degrees $\geq m.$ Since $d_K=0$ the morphism $d_M:M
  \ra [1]M$ is ($K$-linear and hence) a morphism in $C(K).$ The short
  exact sequence $Z(M) \hra M \xsra{d_M} [1]B(M)$ in $C(K)$ yields a
  triangle in $D(K)$ and by rotation a triangle
  \begin{equation*}
    B(M) \ra Z(M) \ra M \xra{d_M} [1]B(M)
  \end{equation*}
  for some morphism $B(M) \ra Z(M)$ in $D(K).$
  Since $B(M)$ and $Z(M)$ have vanishing differential and are
  concentrated in degrees $\geq m+1$ and $\geq m$ respectively,
  Lemma~\ref{l:minimal-graded-projective-resolution-and-cofibrant-resolution} 
  yields cofibrant resolutions $P \ra B(M)$ and $Q \ra Z(M)$ in
  $C(K)$ such that
  $P$ and $Q$ are graded free as $K$-modules, concentrated in
  degrees
  $\geq m+1$ and $\geq m$ respectively, and such that the differential
  of $P \otimes_K \groundring$ and of $Q \otimes_K \groundring$ vanishes.
  Replacing the first two terms of the above triangle by these
  cofibrant resolutions yields a triangle
  \begin{equation*}
    P \ra Q \ra M \ra [1]P. 
  \end{equation*}
  Since $P$ is h-projective (even cofibrant) we can assume that the
  morphism $P \ra Q$ is represented by a morphism $e: P \ra Q$ in
  $C(K).$ Note that $\Cone(e)$ is h-projective and concentrated in
  degrees $\geq m.$ Hence there is a quasi-isomorphism $\Cone(e) \ra
  M.$
  This implies that $\Cone(e\otimes_K \id_\groundring)=\Cone(e)\otimes_K \groundring \cong
  M \otimes_K^L \groundring$ has cohomology concentrated in degrees $\geq m.$

  If $M \otimes_K^L \groundring$ is acyclic, then 
  $e \otimes_K \id_\groundring: P\otimes_K \groundring \ra Q \otimes_K
  \groundring$ is a
  quasi-isomorphism, and an isomorphism since the differentials of 
  $P\otimes_K \groundring$ and $Q \otimes_K \groundring$ vanish. A morphism $f$ of bounded
  below graded free $K$-modules is an isomorphism if and only if
  $f\otimes_K \id_\groundring$ is an isomorphism. This implies that $e:P \ra
  Q$ is an isomorphism, and hence $\Cone(e) \cong M$ is acyclic.
\end{proof}

\begin{proposition}
  \label{p:smoothness-over-local-graded-finite-homological-dim-algebras}
  Let $\groundring$ be a field and $K$ a graded
  commutative graded $\groundring$-algebra
  such that $K^i=0$ for $i<0,$ $K^0=\groundring.$ We view $K$ as a (graded
  commutative) dg $\groundring$-algebra with differential $d_K=0.$
  Assume that $K$ is a (right) Noetherian ring of finite global
  dimension.
  Let $A$ be a dg $K$-algebra such that $H(A)$ is a finitely generated
  $H(K)$-module
  (of course $H(K)=K$) 
  satisfying $H(A)^i=0$ for $i<0$
  and $H^0(A)=\groundring.$ 
  Furthermore we assume that the structure morphism $K \ra A$ induces
  a monomorphism $K^1=H^1(K) \hra H^1(A)$ on the first cohomology
  groups (this is the case for example if $K^1$ vanishes).
  
  Then $A$ is $K$-smooth if and only if (the structure morphism) $K
  \ra A$ is a quasi-isomorphism.
\end{proposition}

\begin{proof}
  If $K \ra A$ is a quasi-isomorphism then $A$ is obviously $K$-smooth
  (Lemma~\ref{l:smoothness-and-quasi-equis}).
  
  Consider $K \ra A$ as a morphism of dg $K$-modules, and fit it into
  a triangle
  \begin{equation*}
    K \ra A \ra Q \ra [1]K
  \end{equation*}
  in $D(K).$ Since $K \ra A$ induces isomorphisms on
  cohomology in all degrees $\leq 0$ and a monomorphism in degree one,
  $H(Q)$ is concentrated in degrees $\geq 1.$
  Extension of scalars along $K \sra K/K^{> 0}=\groundring$
  yields the triangle
  \begin{equation*}
    \groundring \ra A\otimes_K^L \groundring  \ra Q\otimes_K^L \groundring  \ra [1]\groundring.
  \end{equation*}
  Lemma~\ref{l:acyclic-iff-derived-reduction-acyclic} and the long
  exact cohomology sequence show that $H^i(A\otimes_K^L \groundring)$ vanishes
  for $i < 0$ and that $\groundring \sira H^0(A \otimes_K^L \groundring)$ 
  canonically; moreover
  $\groundring \sira H(A \otimes^L_K \groundring)$
  if and only if $K \ra A$ is a quasi-isomorphism.

  The dg $K$-version of
  \cite[Lemma~13.5]{drinfeld-dg-quotients}
  and 
  Lemma~\ref{l:cofibrant-dg-K-categories}
  yield a cofibrant resolution $\tildew{A} \ra A$ where $\tildew{A}$
  is a (semi-free and) cofibrant dg $K$-algebra $\tildew{A}.$
  Let $B:=\tildew{A} \otimes_K \groundring.$

  Since $\tildew{A} \ra A$ can be viewed as a cofibrant resolution in
  $C(K)$ by Lemma
  \ref{l:cofibrant-dg-K-cat-has-cofibrant-morphism-spaces}, we
  obtain $B \cong A\otimes^L_K \groundring$ in $D(\groundring).$ 
  From the above we obtain $H(B)^i=0$ for $i <0$ and $H^0(B)=\groundring$;
  moreover $\groundring=H(B)$ if and only if $K \ra A$ is a quasi-isomorphism.
  
  \textbf{Claim:} $H(B)$ is bounded above (even finite dimensional as
  a $\groundring$-vector space).

  Assuming this claim, we proceed as follows.
  Let $A$ be $K$-smooth. Then $Q(A)\otimes_K \groundring$ is
  $\groundring$-smooth by Theorem~\ref{t:smoothness-and-base-change}, 
  part \ref{enum:smoothness-and-base-change-preserved},
  and the same is true for $B$ by 
  Remark~\ref{rem:how-to-compute-base-change}.
  Now Proposition~\ref{p:homologically-positive-dg-algebra-plus-conditions-not-smooth} 
  shows that $\groundring=H(B).$ Hence $K \ra A$ is a quasi-isomorphism.

  \textbf{Proof of the claim:}
  We prove that $H(A\otimes^L_K \groundring)$
  has finite $\groundring$-dimension. 
  Let $P$ be a "minimal" graded free resolution of the graded left
  $K$-module $M=\groundring$ as provided by
  Lemma~\ref{l:minimal-graded-projective-resolution-and-cofibrant-resolution}
  (note that $K$ is graded commutative).
  By assumption on $K$ we know that $P_{i}=0$ for $i \ll 0$ and that
  all $P_i$ are finitely generated $K$-modules.
  Then $A\otimes_K^L \groundring \cong A \otimes_K \tot(P) = \tot(A\otimes_K P)$
  in $D(\groundring).$ Recall the decreasing filtration $F^l\tot(A\otimes_K P)$
  of the dg $\groundring$-module $\tot(A\otimes_K P).$ It is finite in our case.
  It gives rise to a spectral sequence $\{E_r^{ij}\}$ converging to
  $E_\infty^{ij}=\gr^i(H^{i+j}(\tot(A\otimes_K P))).$ Hence it is
  sufficient to prove that some page of the spectral sequence is
  finite dimensional.

  Note that $P_i$ is a graded free $K$-module, hence it is isomorphic
  to a direct sum of (finitely 
  many) shifts of $K$
  as a dg $K$-module; hence $A\otimes_K P_i$ is isomorphic to 
  a direct sum of shifts of $A$ (as a dg $A$-module), and
  the $E_1$-page of our spectral sequence is given by
  \begin{align*}
    E_1^{ij} 
    & = H^{i+j}(\gr^i \tot(A\otimes_K P)) = H^{i+j}([-i]A
    \otimes_K P_i) \\
    & = H^j(A\otimes_K P_i)= H^j(A) \otimes_K P_i,
  \end{align*}
  with differential $d_1:E_1^{ij} \ra E_1^{i+1,j}$ equal to 
  \begin{equation*}
    \id_{H^j(A)}\otimes p_i: 
    H^j(A) \otimes_K P_i \ra 
    H^j(A) \otimes_K P_{i+1}. 
  \end{equation*}

  Let $H(A) \otimes_K P$ be the complex in $C(H(A))$ obtained from the
  complex $P$ in $C(K)$ by extension of scalars along $K=H(K) \ra H(A).$
  The vertical differential  $d''$ of its double complex
  $\Dbl(H(A) \otimes_K P)$ vanishes since $d_{H(A)}=0$ and $d_{P_i}=0.$
  If we forget it we have $E_1 = \Dbl(H(A) \otimes_K P).$
  The vanishing of $d''$ then implies that
  $E_2=H(E_1)$ has the same dimension as 
  $H(\tot(H(A)\otimes_K P)).$ Note that
  $\tot(H(A)\otimes_K P)= H(A) \otimes_K \tot(P) \cong H(A)
  \otimes_K^L \groundring.$ 
  Lemma~\ref{l:minimal-graded-projective-resolution-and-cofibrant-resolution}, 
  now applied to $M=H(A)$ (which is finitely generated and of finite
  projective dimension as a $K$-module), shows that
  $\dim_\groundring H(H(A) \otimes^L_{K} \groundring) < \infty.$
  Hence $E_2$ is finite dimensional.
\end{proof}

\section{Smoothness of equivariant derived categories}
\label{sec:smoothn-equiv-derived-cats}

\subsection{Sheaves of dg modules over sheaves of dg algebras}
\label{sec:sheaves-dg-modules-over-dg-algebras}

Let $X$ be a topological space. 
We work with sheaves of modules over a fixed field $\groundring$ 
(up to and including Section~\ref{sec:inj-model-structure-on-sheaf-CA}
it could be a commutative ring; later on it will be $\DR$) on $X.$
Let $\mathcal{A}=\mathcal{A}_X$ be a sheaf of dg ($\groundring$-)algebras on $X.$
We denote by $\calMod(\mathcal{A})$ the
following dg ($\groundring$-)category: Objects are dg (right)
$\mathcal{A}$-modules (= sheaves of dg modules over the sheaf
$\mathcal{A}$ of dg algebras), morphisms are $\mathcal{A}$-linear (and not
necessarily compatible with the differentials), a morphism has
degree $n$ if it raises the degree by $n,$ and the differentials on
morphism spaces are defined in the usual way.

We define the abelian category $C(\mathcal{A}),$ the (triangulated)
homotopy category 
$\mathcal{H}(\mathcal{A}):=[\calMod(\mathcal{A})]$
and the (triangulated) derived
category $D(\mathcal{A})$ of dg $\mathcal{A}$-modules in the obvious
way. Quasi-isomorphisms in $C(\mathcal{A})$ or
$\mathcal{H}(\mathcal{A})$ are defined in the obvious way.

For example, the constant sheaf $\ul{\groundring}=\ul{\groundring}_X$
on $X$ with stalk $\groundring$ is a sheaf of dg algebras
and $D(\ul{\groundring}_X)$ is the usual (unbounded) derived category
of sheaves of $\groundring$-vector spaces on $X.$ We denote this
category sometimes by $D(X).$ 

\subsubsection{Injective model structure on \texorpdfstring{$C(\mathcal{A})$}{C(A)}}
\label{sec:inj-model-structure-on-sheaf-CA}

The following result is presumably known but we could not find a good
reference (cf.\ \cite{beke-sheafifiable-homotopy-model-categories} 
and \cite{hovey-model-cat-str-chain-cmplxs-sheaves-arxiv}).

\begin{proposition}
  \label{p:model-structure-sheaf-C-A}
  There is a cofibrantly generated model structure on
  $C(\mathcal{A})$ such that the weak equivalences are the
  quasi-isomorphisms and the cofibrations are the monomorphisms.
  
  We call it the \define{injective} model structure on $C(\mathcal{A}).$
\end{proposition}

\begin{proof}
  Our proof essentially coincides with the proof of
  \cite[Prop.~3.13]{beke-sheafifiable-homotopy-model-categories}
  and is based on the result
  \cite[Thm.~1.7]{beke-sheafifiable-homotopy-model-categories}
  by J.~Smith. We
  therefore need to check the assumptions and conditions c0-c3 there,
  for the 
  category $C(\mathcal{A}),$ $\mathcal{W}$ the class of
  quasi-isomorphisms in $C(\mathcal{A}),$ and $I$ a suitable set of
  morphisms to be found. 
  
  If $j: U \hra X$ is the inclusion of an open subset, define
  $S_{n,U}:=[n] j_!j^*\mathcal{A}$ and $D_{n,U}:=\Cone(\id_{S_{n,U}}).$
  Then, for $M \in C(\mathcal{A}),$ we have canonical isomorphisms
  \begin{align}
    \label{eq:SnU-and-cycles}
    (C(\mathcal{A}))(S_{n,U}, M) & \sira Z^{-n}(M(U)),\\
    \label{eq:DnU-and-sections}
    (C(\mathcal{A}))(D_{n,U}, M) & \sira M^{-n-1}(U).
  \end{align} 
  The second equality implies that $C(\mathcal{A})$ is a Grothendieck
  category (cf.\ the proof of \cite[Thm.~18.1.6]{KS-cat-sh}), and in
  particular locally presentable
  (\cite[Prop.~3.10]{beke-sheafifiable-homotopy-model-categories}).
  In the same way $C(H(\mathcal{A}))$ is locally presentable.

  Let $H: C(\mathcal{A}) \ra C(H(\mathcal{A}))$ be the cohomology
  functor. It follows as in the proof
  of \cite[Prop.~3.13]{beke-sheafifiable-homotopy-model-categories},
  using
  \cite[Prop.~1.15 and
  1.18]{beke-sheafifiable-homotopy-model-categories},
  that c3 is satisfied.

  Let $\Mono$ be the class of monomorphisms in $C(\mathcal{A}).$
  Then \cite[Prop.~1.12]{beke-sheafifiable-homotopy-model-categories}
  provides a set $I \subset \Mono$ such that $\Mono = I\hy\cof.$
  It is clear that $\Mono \cap \mathcal{W}$ is closed under pushouts
  and transfinite compositions (cf.\ proof of
  \cite[Cor.~1.7]{hovey-model-cat-str-chain-cmplxs-sheaves-arxiv}). This shows c2.

  Condition c0 is obvious, so we are left to show c1, i.\,e.\
  $I\hy\inj \subset \mathcal{W}.$
  Since $I\hy\inj=(I\hy\cof)\hy\inj=\Mono\hy\inj$ we need to show
  that $\Mono\hy\inj \subset \mathcal{W}.$
  Let $(\phi: M \ra N) \in \Mono\hy\inj.$ 
  Let $x \in X.$ 
  We need to show that $\phi_x: M_x \ra N_x$ is a quasi-isomorphism.
  Any element of $Z^{-n}(N_x)$ comes from some $f \in
  Z^{-n}(N(U)),$ for some open subset $U \subset X$ containing $x.$ 
  This element corresponds (use \eqref{eq:SnU-and-cycles}) to the
  lower horizontal morphism in the commutative diagram
  \begin{equation*}
    \xymatrix{
      {0} \ar[r] \ar[d] &
      {M} \ar[d]^-\phi \\
      {S_{n,U}} \ar[r]^-f \ar@{..>}[ur]|h&
      {N}
    }
  \end{equation*}
  which admits the indicated lift $h.$ This shows that $\phi_x$ is
  surjective on cocycles.
  
  Now assume that an element of $Z^{-n}(M_x)$ becomes a coboundary in
  $N_x.$ Then there are an open subset $U \subset X$ containing $x$ and
  elements $f \in Z^{-n}(M(U))$ and $g \in N^{-n-1}(U)$ such that $f_x$ is the
  given cocycle and $\phi(f)= d_N(g).$ If we interpret $g$ and $f$ as
  morphisms using
  \eqref{eq:SnU-and-cycles} and
  \eqref{eq:DnU-and-sections} we obtain a commutative diagram
  \begin{equation*}
    \xymatrix{
      {S_{n,U}} \ar[r]^-f \ar[d] &
      {M} \ar[d]^-\phi \\
      {D_{n,U}} \ar[r]^-g \ar@{..>}[ur]|h &
      {N}
    }
  \end{equation*}
  whose left vertical morphism is the obvious monomorphisms to the
  cone; then there is a lift $h$ as indicated showing that the
  cocycle $f$
  is already a coboundary.

  Now we can apply
  \cite[Thm.~1.7]{beke-sheafifiable-homotopy-model-categories}.
\end{proof}

A dg $\mathcal{A}$-module $I$ is called h-injective
if all morphisms $N \ra I$ in $C(\mathcal{A})$ with acyclic $N$
are homotopic to zero, 
i.\,e.\ $(\mathcal{H}(\mathcal{A}))(N, I)=0.$
The arguments dual to those used in the proof
of Lemma~\ref{l:cofobj-is-h-proj-dgA}
show that 
fibrant objects in $C(\mathcal{A})$ are h-injective.

Denote by $\mathcal{H}(\mathcal{A})_{\hinj}$ 
(resp.~$\calMod(\mathcal{A})_{\hinj}$)
the full subcategory of
$\mathcal{H}(\mathcal{A})$ 
(resp.~$\calMod(\mathcal{A})$)
consisting of h-injective dg $\mathcal{A}$-modules.
Standard arguments show that
the canonical  functor
\begin{equation}
  \label{eq:homotopy-cat-h-injectives-isom-derived-cat-dg-sheaves-version}
  \mathcal{H}(\mathcal{A})_{\hinj} \sira D(\mathcal{A})
\end{equation}
is a triangulated equivalence. 
Since $\mathcal{H}(\mathcal{A})_{\hinj}
:=[\calMod(\mathcal{A})_{\hinj}]$
this means that $\calMod(\mathcal{A})_{\hinj}$ is a dg enhancement of $D(\mathcal{A}).$

In Section~\ref{sec:proj-model-structure-dg-A-modules} a (P) resolution
of an object $X$ was defined to be a trivial fibration $C \ra X$ with
$C$ having property (P).
All objects in the (projective) model categories considered there were fibrant.
Now in the (injective) model category $C(\mathcal{A}),$ all objects
are cofibrant, so it is convenient to extend this definition and to say that a
(P) resolution of an object $X$ is a trivial cofibration $X \ra F$ with
$F$ having property (P). Even if not mentioned we hope that it
is always clear from the context which object is resolved.

We fix for any any dg $\mathcal{A}$-module $N$ a fibrant (and hence
h-injective) resolution, i.\,e.\ 
a monomorphic quasi-isomorphism $N \hra \iota(N)$ in $C(\mathcal{A})$ with $\iota(N)$
fibrant (and hence h-injective). Then $N \mapsto \iota(N)$ extends to a functor
\begin{equation}
  \label{eq:homotopy-cat-h-injectives-isom-derived-cat-dg-sheaves-version-quasi-inverse}
  \iota:D(\mathcal{A}) \ra \mathcal{H}(\mathcal{A})_{\hinj}
\end{equation}
which is
quasi-inverse to 
\eqref{eq:homotopy-cat-h-injectives-isom-derived-cat-dg-sheaves-version}.
We will use $\iota$ for (right-)deriving certain functors.

\subsubsection{Extension and restriction }
\label{sec:production-and-restriction}

The structure morphism $\ul{\groundring} \ra \mathcal{A}$ gives rise to dg
functors $\res:=\res^\mathcal{A}_{\ul{\groundring}}$ and $\pro:=\pro^\mathcal{A}_{\ul{\groundring}},$
\begin{equation}
  \label{eq:prod-res-adjunction-sheaves-dg-algs}
  \xymatrix{
    {\calMod(\ul{\groundring})} \ar@/^0.3pc/[r]^{\pro} &
    \ar@/^0.3pc/[l]^{\res} 
    {\calMod(\mathcal{A}),} 
  }
\end{equation}
and there is an adjunction $(\pro, \res)$ given by the obvious isomorphisms.

The following assumption on the structure morphism will be satisfied
in our main applications.
\begin{equation}
  \label{eq:k-to-A-qiso}
  \tag{Str-Qiso}
  \text{\bf{The morphism $\ul{\groundring} \ra \mathcal{A}$ is a quasi-isomorphism.}}
\end{equation}

\begin{lemma}
  \label{l:qiso-sheaves-dg-algebras-equiv-derived}
  Assume that \eqref{eq:k-to-A-qiso} is satisfied. Then the adjunction
  \eqref{eq:prod-res-adjunction-sheaves-dg-algs} induces quasi-inverse
  equivalences
  \begin{equation*}
    \label{eq:prod-res-adjunction-sheaves-dg-algs-derived-equiv-if-qiso}
    \xymatrix{
      {D(\ul{\groundring})} \ar@/^0.3pc/[r]^{\pro} \ar@{}[r]|\sim &
      \ar@/^0.3pc/[l]^{\res} 
      {D(\mathcal{A})} 
    }
  \end{equation*}
  of triangulated categories.
\end{lemma}

\begin{proof}
  Let $M \in \calMod(\ul{\groundring}).$ Since we work over a field
  the dg functor $(M\otimes_{\ul{\groundring}}?)$
  preserves (acyclics and) quasi-isomorphisms (test on the stalks).
  If we apply this functor to the quasi-isomorphism
  $\ul{\groundring} \ra \mathcal{A}$ we see that the
  adjunction morphism 
  $\varepsilon_M: M \ra \res(M \otimes_{\ul{\groundring}} \mathcal{A})$
  is a quasi-isomorphism. 
  
  Similarly, for $N \in \calMod(\mathcal{A}),$ the adjunction morphism
  $\delta_N: (\res N)\otimes_{\ul{\groundring}} \mathcal{A} \ra N$ is a
  quasi-isomorphism:
  This is the case if and only if $\res(\delta_N)$ is a
  quasi-isomorphism; consider
  \begin{equation*}
    \res N \xra{\varepsilon_{\res N}} \res((\res N) \otimes_{\ul{\groundring}}
    \mathcal{A}) \xra{\res(\delta_N)} \res N;
  \end{equation*}
  the composition is $\id_{\res N},$ and the first morphism is a
  quasi-isomorphism as observed above; hence $\res(\delta_N)$ is a
  quasi-isomorphism.

  Note that $\pro$ (we work over a field) and
  $\res$ both preserve acyclics. Hence these two functors descend to
  triangulated functors between
  $D(\ul{\groundring})$ and $D(\mathcal{A})$ which are adjoint and
  quasi-inverse to each other by what we observed above.
\end{proof}

\begin{remark}
  \label{rem:model-structure-dg-sheaves-quillen-adjunction}
  Assume that \eqref{eq:k-to-A-qiso} is satisfied. 
  One may ask whether $(\pro, \res)$ defines a Quillen equivalence
  between $C(\ul{\groundring})$ and $C(\mathcal{A})$
  if we equip each of these categories with the
  injective model structure from
  Proposition~\ref{p:model-structure-sheaf-C-A}.
  One can use 
  \cite[Thm.~11.1.13]{fresse-modules-over-operads}
  to transfer the injective model structure on $C(\ul{\groundring})$
  to a model structure on $\mathcal{A}$ such that we obtain a Quillen
  equivalence. For this model structure on $C(\mathcal{A}),$ the weak
  equivalences are the quasi-isomorphisms, and the cofibrations are
  contained in the monomorphisms. We did not check whether we have
  equality there.
\end{remark}

\subsubsection{Standard functors for a decomposition into an open and a closed subspace}
\label{sec:stand-funct-open-and-closed-subspace}

We continue the above discussion by providing some preparations for
the proof of
Theorem~\ref{t:X-G-smooth-iff-all-orbits-G-smooth} below.

Let $i:F \hra X$ be a closed embedding and $j:U \hra X$ the
complementary open embedding. 
Let $\mathcal{A}_F:=i^*\mathcal{A}$ and
$\mathcal{A}_U:=j^*\mathcal{A}.$
Since the obvious dg functors $i^*,$ $i_{\as}:=i_*=i_!$ and $j_!,$ $j^{\as}:=j^!=j^*$ between
$\calMod(\mathcal{A}_F),$
$\calMod(\mathcal{A}_X),$ and
$\calMod(\mathcal{A}_U)$
preserve acyclics they induce the horizontal functors in the
following diagram:
\begin{equation*}
  \xymatrix@=45pt{
    {D(\ul{\groundring}_F)} \ar[r]|{i_{\as}} &
    \ar@/_1pc/[l]_{i^*} 
    {D(\ul{\groundring}_X)} 
    \ar[r]|{j^{\as}} 
    & 
    \ar@/_1pc/[l]_{j_!} 
    {D(\ul{\groundring}_U)}
    \\
    {D(\mathcal{A}_F)} 
    \ar[u]^-{\res^{\mathcal{A}_F}_{\ul{\groundring}_F}}
    \ar[r]|{i_{\as}} &
    \ar@/_1pc/[l]_{i^*} 
    {D(\mathcal{A}_X)} 
    \ar[u]^-{\res^{\mathcal{A}_X}_{\ul{\groundring}_X}}
    \ar[r]|{j^{\as}} 
    & 
    \ar@/_1pc/[l]_{j_!} 
    {D(\mathcal{A}_U)}
    \ar[u]^-{\res^{\mathcal{A}_U}_{\ul{\groundring}_U}}
  }
\end{equation*}
The vertical functors are the obvious restriction functors (which
preserve acyclics and hence descend trivially to the derived
categories) along the respective structure morphisms as described in
Section~\ref{sec:production-and-restriction}.
In this diagram the four obvious squares commute, and we have
adjunctions $(i^*, i_{\as})$ and $(j_!, j^{\as}).$

Since $i:F \ra X$ is a closed embedding we have the dg functor
$i^!:\calMod(\mathcal{A}_X) \ra \calMod(\mathcal{A}_F).$ It preserves
h-injectives since its left adjoint $i_\as$ preserves acyclics.
We define $Ri^!$ to be the composition
\begin{equation*}
  Ri^!: D(\mathcal{A}_X) 
  \xra{\iota} \mathcal{H}(\mathcal{A})_{\hinj} 
  \xra{i^!} \mathcal{H}(\mathcal{A}_F)_{\hinj} 
  \ra D(\mathcal{A}_F).
\end{equation*}
Then we have an adjunction $(i_{\as}, Ri^!).$ Similarly, we define
$Rj_*:D(\mathcal{A}_U) \ra D(\mathcal{A}_X)$
and obtain an adjunction $(j^{\as}, Rj_*),$ and we can do the same with
$\ul{\groundring}$ instead of $\mathcal{A}.$

Now we assume that \eqref{eq:k-to-A-qiso} is satisfied. 
Then also $\ul{\groundring}_F \ra \mathcal{A}_F$ and $\ul{\groundring}_U \ra
\mathcal{A}_U$ are quasi-isomorphisms.
We obtain the diagram
\begin{equation}
  \label{eq:six-functors-dg-modules-over-dg-sheaf-of-algs}
  \xymatrix@=45pt{
    {D(\ul{\groundring}_F)} \ar[r]|{i_{\as}} &
    \ar@/_1pc/[l]_{i^*} 
    \ar@/^1pc/[l]^{Ri^!} 
    {D(\ul{\groundring}_X)} 
    \ar[r]|{j^{\as}} 
    & 
    \ar@/_1pc/[l]_{j_!} 
    \ar@/^1pc/[l]^{Rj_*} 
    {D(\ul{\groundring}_U)}
    \\
    {D(\mathcal{A}_F)} 
    \ar[u]^-{\res^{\mathcal{A}_F}_{\ul{\groundring}_F}}_-{\sim}
    \ar[r]|{i_{\as}} &
    \ar@/_1pc/[l]_{i^*} 
    \ar@/^1pc/[l]^{Ri^!} 
    {D(\mathcal{A}_X)} 
    \ar[u]^-{\res^{\mathcal{A}_X}_{\ul{\groundring}_X}}_-{\sim}
    \ar[r]|{j^{\as}} 
    & 
    \ar@/_1pc/[l]_{j_!} 
    \ar@/^1pc/[l]^{Rj_*} 
    {D(\mathcal{A}_U).}
    \ar[u]^-{\res^{\mathcal{A}_U}_{\ul{\groundring}_U}}_-{\sim}
  }
\end{equation}
The vertical arrows are equivalences 
(Lemma~\ref{l:qiso-sheaves-dg-algebras-equiv-derived}), and the square
with horizontal 
sides $Ri^!$ (resp.\ $Rj_*$) commutes (up to natural isomorphisms)
since $\res$ preserves h-injectives (its left adjoint $\pro$ preserves acyclics).

\begin{lemma}
  \label{l:suitable-h-injective-lifts}
  Assume that \eqref{eq:k-to-A-qiso} is satisfied. 
  \begin{enumerate}
  \item 
    \label{enum:resolve-on-closed-by-res-fibrant}
    Let $\mathcal{F} \in C(\ul{\groundring}_F)$ be an object. Then there is a
    fibrant object
    $\mathcal{F}' \in C(\mathcal{A}_F)$ 
    and a monomorphic quasi-isomorphism $\mathcal{F} \ra
    \res(\mathcal{F}')$ in $C(\ul{\groundring}_F).$ 
    This morphism remains a monomorphic quasi-isomorphism under $i_{\as}.$
    Moreover, $i_{\as} \mathcal{F}'$ is h-injective (and fibrant) and
    $\res(i_{\as}\mathcal{F}') \cong 
    i_{\as} \res (\mathcal{F}')$ in $C(\ul{\groundring}_F).$
  \item 
    \label{enum:resolve-extension-by-zero-from-open-by-res-fibrant}
    Let $\mathcal{U} \in C(\ul{\groundring}_U)$ be given. Then there is
    $\mathcal{V} \in C(\mathcal{A}_X)$ fibrant
    together with a monomorphic quasi-isomorphism $j_! (\mathcal{U}) \hra
    \res(\mathcal{V})$ and 
    a fibrant resolution 
    $\mathcal{U} \otimes_{\ul{\groundring}_U} \mathcal{A}_U 
    \hra j^{\as}(\mathcal{V})$ in $C(\mathcal{A}_U)$
    Moreover, $\mathcal{U} \sira j^{\as} j_! \mathcal{U} \hra j^{\as}
    \res(\mathcal{V}) \cong \res j^{\as}(\mathcal{V})$ is a monomorphic
    quasi-isomorphism in $C(\ul{\groundring}_U).$
  \end{enumerate}
\end{lemma}

\begin{proof}
  \ref{enum:resolve-on-closed-by-res-fibrant}:
  The object $\pro(\mathcal{F})=\mathcal{F} \otimes_{\ul{\groundring}_F}
  \mathcal{A}_F$ has a fibrant resolution 
  $\mathcal{F}\otimes_{\ul{\groundring}_F} \mathcal{A}_F \ra
  \mathcal{F}'.$ Apply $\res$ (which preserves monomorphism and
  quasi-isomorphisms) and use the monomorphic quasi-isomorphism
  $\varepsilon_\mathcal{F}: \mathcal{F} \ra \res
  (\mathcal{F} \otimes_{\ul{\groundring}_F} \mathcal{A}_F)$ (test on
  stalks: on the stalk at $x \in X$ this morphism is given by applying
  the exact functor
  $(\mathcal{F}_x \otimes_{\ul{\groundring}_{F,x}} ?)$ 
  to $\groundring=\ul{\groundring}_{F,x} \ra 
  \mathcal{A}_x$; the latter morphism is a quasi-isomorphisms and
  injective since $\groundring=Z^0(\groundring)=H^0(\groundring)$).

  Obviously, $i_{\as}$ preserves monomorphisms and quasi-isomorphisms.
  Note that $i_{\as}$ preserves h-injectives, since its left adjoint $i^*$ preserves
  acyclics (since $i^*$ preserves monomorphic
  quasi-isomorphisms (= trivial cofibrations), $i_{\as}$ preserves fibrations). It is clear
  that $\res$ and $i_{\as}$ commute.

  \ref{enum:resolve-extension-by-zero-from-open-by-res-fibrant}
  The first statement is proved as above, using a fibrant resolution
  $j_!(\mathcal{U}) \otimes_{\ul{\groundring}_X}
  \mathcal{A}_X \hra \mathcal{V}.$
  The left adjoint $j_!$ 
  of $j^{\as}$ preserves trivial
  cofibrations, hence $j^{\as}$ preserves fibrations. In particular
  $j^{\as}(\mathcal{V})$ is fibrant. If we apply $j^{\as}$ to the cofibrant resolution
  $j_!(\mathcal{U}) \otimes_{\ul{\groundring}_X}
  \mathcal{A}_X \hra \mathcal{V}$ we obtain a monomorphic
  quasi-isomorphism
  $\mathcal{U} \otimes_{\ul{\groundring}_U}
  \mathcal{A}_U \hra j^{\as}(\mathcal{V}).$
  The last statement is clear since again $j^{\as}$ preserves trivial
  cofibrations (= monomorphic quasi-isomorphisms).
\end{proof}

\subsection{Refined enhancements of equivariant derived categories}
\label{sec:enhancem-equiv-deriv-cats}

Let $G$ be a connected complex affine algebraic group
and $X$ a complex $G$-variety.
We usually equip $G$ and $X$ with the classical topology.
In the following, our field $\groundring$ will be $\DR.$

Let $p:EG \ra BG=EG/G$ be a universal $G$-principal fiber bundle such 
that $BG$ is an $\infty$-dimensional manifold in the sense of
\cite[12.2]{BL} or a smooth (paracompact) manifold of finite
dimension, and such that $EG$ is $\infty$-acyclic. Such a bundle exists
by \cite[12.4.2.a]{BL}, and we can and will assume that $EG$ is
open-locally pre-$\infty$-acyclic (as defined in
the Appendix \ref{sec:derived-fibered-base}) and that $BG$ is locally
contractible. There is a  
sheaf $\Omega_{BG}$ of graded commutative dg ($\DR$-)algebras (the de Rham
sheaf) on $BG$ such that the obvious morphism $\ul{\DR}_{BG} \ra
\Omega_{BG}$ is a quasi-isomorphism and each $\Omega_{BG}^i$ is soft
and acyclic (\cite[12.2.3]{BL}).

Let $X_G:= EG \times_G X$ be the quotient of $EG \times X$ by the
diagonal $G$-action. 
Since $EG \times X \ra X$ is an $\infty$-acyclic resolution,
the bounded constructible $G$-equivariant derived category
$D^b_{G,c}(X)$ can be viewed as a full 
subcategory of $D(X_G)$ \cite[2.9.9]{BL}: it consists of those
objects whose pullback to $EG \times X$ is isomorphic to the pullback of an
object of $D^b_c(X).$ 
Let 
\begin{equation*}
  c: X_G \ra BG=\point_G  
\end{equation*}
be the obvious morphism. 
Since the obvious morphism $\ul{\DR}_{X_G} \ra c^*(\Omega_{BG})$
is a quasi-isomorphism,
Lemma~\ref{l:qiso-sheaves-dg-algebras-equiv-derived} shows that
$D(X_G)$ and $D(c^*(\Omega_{BG}))$ are equivalent as triangulated
categories. Hence we can view $D^b_{G,c}(X)$ as a full subcategory of
$D(c^*(\Omega_{BG})).$
Recall that $\calMod(c^*(\Omega_{BG}))_{\hinj}$ is a dg enhancement of
$D(c^*(\Omega_{BG})).$
Let $\calMod_G(c^*(\Omega_{BG}))_{\hinj}$ be the full subcategory of 
$\calMod(c^*(\Omega_{BG}))_{\hinj}$ consisting of objects that are
isomorphic (in $D(c^*(\Omega_{BG}))$) to an object of $D^b_{G,c}(X).$
Then 
$\calMod_G(c^*(\Omega_{BG}))_{\hinj}$ is a dg enhancement of
$D^b_{G,c}(X).$
Note that $\calMod(c^*(\Omega_{BG}))$ is a dg
$\Gamma(\Omega_{BG})$-category in the obvious way
(and that $\Gamma(\Omega_{BG})$ is graded commutative), and therefore
$\calMod_G(c^*(\Omega_{BG}))_{\hinj}$ is a dg
$\Gamma(\Omega_{BG})$-category.

\begin{definition}
  \label{d:smoothness-DbGc}
  We say that $X$
  is \define{(homologically) $G$-smooth} (or more precisely
  \define{$D^b_{G,c}$-smooth}),
  if 
  $\calMod_G(c^*(\Omega_{BG}))_{\hinj}$ is a smooth\footnote{
    Here we implicitly replace
    $\calMod_G(c^*(\Omega_{BG}))_{\hinj}$ by a 
    dg $\Gamma(\Omega_{BG})$-equivalent
    dg $\Gamma(\Omega_{BG})$-subcategory which is small.
  }
  dg $\Gamma(\Omega_{BG})$-category.
\end{definition}


The following lemma explains that $G$-smoothness can be tested on the
dg endomorphisms of a classical generator.

\begin{lemma}
  \label{l:test-smoothness-on-generator}
  Assume that there is an object $E \in
  \calMod_G(c^*(\Omega_{BG}))_{\hinj}$ such that $E$ is a classical
  generator of $[\calMod_G(c^*(\Omega_{BG}))_{\hinj}] \sira
  D^b_{G,c}(X).$ 
  Then $X$ is $G$-smooth if and only if 
  $(\calMod(c^*(\Omega_{BG})))(E)$
  is
  $\Gamma(\Omega_{BG})$-smooth.
\end{lemma}

\begin{proof}
  Note that $D^b_{G,c}(X)$ is Karoubian since it has a bounded
  t-structure
  \cite{le-chen-karoubi-trcat-bdd-t-str}. Then the result follows from
  Corollary~\ref{c:test-smoothness-on-classical-generator}.
\end{proof}

\begin{remark}
  \label{rem:independent-of-choice}
  We show that Definition~\ref{d:smoothness-DbGc} does not depend
  on the choice of $p: EG \ra BG.$
  Note that in the case of interest
  to us
  (where $G$ has
  only finitely many orbits in $X$ and all stabilizers are connected),
  independence of $p$ will also be a consequence of 
  Theorems~\ref{t:X-G-smooth-iff-all-orbits-G-smooth}
  and \ref{t:one-orbit-case-smoothness-and-quisos}
  (and
  Propositions~\ref{p:smoothness-one-orbit-case}
  and \ref{p:smoothness-one-orbit-case-algebraic-groups}). In
  general one may argue as follows.

  Let $p'': E''G \ra B''G$ be another choice. Consider $E'G:= EG
  \times E''G$ with the diagonal $G$-action and let $p':E'G \ra
  B'G$ be the quotient map. Then $B'G$ is an
  ($\infty$-dimensional) manifold in a natural way, and $p'$ is
  another possible choice.
  It is enough to show that $X$ is $G$-smooth with respect to $p$
  if and only if it is $G$-smooth with respect to $p'.$

  Consider the commutative diagram
  \begin{equation*}
    \xymatrix{
      {X'_G:=E'G \times_G X} \ar[r]^-{c'} \ar[d]^-{\pi} & 
      {B'G} \ar[d]^-{\gamma} \\
      {X_G:=EG \times_G X} \ar[r]^-{c} & 
      {BG.}
    }
  \end{equation*}
  Pull-back of differential forms 
  \cite[12.2.6]{BL}
  defines a morphism
  $\gamma^*(\Omega_{BG}) \ra \Omega_{B'G}$ and a
  quasi-isomorphism $\Gamma(\Omega_{BG}) \ra
  \Gamma(\Omega_{B'G}).$
  We also obtain 
  quasi-isomorphisms
  $\ul{\DR}_{X'_G} \ra \pi^*c^*(\Omega_{BG}) =
  c'^*\gamma^*(\Omega_{BG}) \ra c'^*(\Omega_{B'G}).$
  Lemma~\ref{l:qiso-sheaves-dg-algebras-equiv-derived}
  implies that
  $\res:=\res^{c'^*(\Omega_{B'G})}_{c'^*\gamma^*(\Omega_{BG})}:
  D(c'^*(\Omega_{B'G})) \ra D(c'^*\gamma^*(\Omega_{BG}))$ is an
  equivalence.
  
  Let $D^b_{G,c}(X;c^*(\Omega_{BG}))$ (resp.\
  $D^b_{G,c}(X;c'^*(\Omega_{B'G}))$) denote
  $D^b_{G,c}(X)$ viewed as a full subcategory of
  $D(c^*(\Omega_{BG}))$ (resp.\ $D(c'^*(\Omega_{B'G}))$).
  Let $D^b_{G,c}(X;c'^*\gamma^*(\Omega_{BG}))$ be the essential
  image of $D^b_{G,c}(X;c'^*(\Omega_{B'G}))$ under the functor
  $\res.$
  We obtain equivalences 
  \begin{equation*}
    D^b_{G,c}(X;c'^*(\Omega_{B'G}))
    \xra{\res}
    D^b_{G,c}(X;c'^*\gamma^*(\Omega_{BG}))
    \xla{\pi^*} 
    D^b_{G,c}(X;c^*(\Omega_{BG}))
  \end{equation*}
  where $\pi^*$ is induced from
  $\pi^*: D(c^*(\Omega_{BG})) \ra
  D(c'^*\gamma^*(\Omega_{BG}))$ and is an equivalence by 
  Lemma~\ref{l:qiso-sheaves-dg-algebras-equiv-derived} and the
  fact that $E''G$ is $\infty$-acyclic (use \cite[Prop.~61]{OSdiss-equi-mathz}).

  Let $\calMod_G(c'^*\gamma^*(\Omega_{BG}))_{\hinj}$ be the full
  subcategory of  
  $\calMod(c'^*\gamma^*(\Omega_{BG}))_{\hinj}$ consisting of objects that are
  isomorphic (in $D(c'^*\gamma^*(\Omega_{BG}))$) to an object of
  $D^b_{G,c}(X;c'^*\gamma^*(\Omega_{BG})).$
  Then 
  $\calMod_G(c'^*\gamma^*(\Omega_{BG}))_{\hinj}$ is a dg enhancement of
  $D^b_{G,c}(X;c'^*\gamma^*(\Omega_{BG}))$
  and a dg
  $\Gamma(\Omega_{BG})$-category.

  Using the quasi-isomorphism 
  $\Gamma(\Omega_{BG}) \ra
  \Gamma(\Omega_{B'G})$
  and
  Theorem~\ref{t:smoothness-and-base-change}.\ref{enum:smoothness-and-qiso-base-change-res}
  we see that
  $X$ is $G$-smooth with respect to $p'$ if and only if
  $\calMod_G(c'^*(\Omega_{B'G}))_{\hinj}$ is a smooth dg
  $\Gamma(\Omega_{BG})$-category.
  By Lemma~\ref{l:smoothness-and-quasi-equis}
  it is hence enough to show 
  that the dg $\Gamma(\Omega_{BG})$-categories 
  $\calMod_G(c'^*(\Omega_{B'G}))_{\hinj},$ 
  $\calMod_G(c'^*\gamma^*(\Omega_{BG}))_{\hinj}$
  and
  $\calMod_G(c^*(\Omega_{BG}))_{\hinj}$
  are connected by a zig-zag of quasi-equivalences.
 
  Consider the auxiliary dg $\Gamma(\Omega_{BG})$-category
  $\mathcal{C}$ whose objects 
  are triples $(L,J, \delta)$ where $L \in
  \calMod_G(c'^*(\Omega_{B'G}))_{\hinj},$ $J \in
  \calMod_G(c'^*\gamma^*(\Omega_{BG}))_{\hinj},$ and $\delta:
  \res(L) \ra J$ is a closed degree zero morphism in
  $\calMod_G(c'^*\gamma^*(\Omega_{BG}))$ that becomes an
  isomorphism in $D(c'^*\gamma^*(\Omega_{BG})).$ Morphisms
  between such triples are defined as in the proof of the
  proposition in section "Morphism
  oriented \v{C}ech enhancement" in
  \cite{valery-olaf-matfak-semi-orth-decomp}. Similarly,
  we define a dg $\Gamma(\Omega_{BG})$-category $\mathcal{D}$
  whose objects 
  are triples $(I,J, \epsilon)$ where $I \in
  \calMod_G(c^*(\Omega_{BG}))_{\hinj},$ $J \in
  \calMod_G(c'^*\gamma^*(\Omega_{BG}))_{\hinj},$ and $\epsilon:
  \pi^*(I) \ra J$ is a closed degree zero morphism in
  $\calMod_G(c'^*\gamma^*(\Omega_{BG}))$ that becomes an
  isomorphism in $D(c'^*\gamma^*(\Omega_{BG})).$
  We leave it to the reader to check that the obvious projection
  functors
  \begin{equation*}
    \calMod_G(c'^*(\Omega_{B'G}))_{\hinj}
    \la
    \mathcal{C} 
    \ra
    \calMod_G(c'^*\gamma^*(\Omega_{BG}))_{\hinj}
    \la
    \mathcal{D}
    \ra
    \calMod_G(c^*(\Omega_{BG}))_{\hinj}
  \end{equation*}
  are quasi-equivalences of $\Gamma(\Omega_{BG})$-categories. 
\end{remark}

\subsection{Smoothness of homogeneous spaces}
\label{sec:smoothness-homogeneous-spaces}

Let $G$ be a connected complex affine algebraic group and let $H
\subset G$ be a closed subgroup. We discuss the smoothness of the
$G$-variety $X=G/H.$ 
Let $p:EG \ra BG$ be as above.
Then $EG \ra EG/H$ is a universal $H$-principal fiber
bundle, and we define $BH:= EG/H.$ 
Then 
\begin{equation*}
  X_G = EG \times_G G/H = (EG \times_G G)/H =EG/H=BH
\end{equation*}
canonically, so $c: X_G =BH \ra BG.$
Since $c$ is a locally trivial bundle with fiber $G/H$ we can equip 
$BH$ with the structure of a (possibly $\infty$-dimensional) manifold.
In particular this defines the de Rham sheaf $\Omega_{BH}$ on $BH.$
Pullback of differential forms \cite[12.2.6]{BL}
yields a morphism
\begin{equation*}
  \Gamma(\Omega_{BG}) \ra \Gamma(\Omega_{BH})
\end{equation*}
of (graded commutative) dg $\DR$-algebras. Taking cohomology we get
the usual morphism 
\begin{equation*}
  H_G(\point) \ra H_G(G/H)=H_H(\point)
\end{equation*}
on equivariant cohomology (since $BG$ and $BH$ are locally contractible and
paracompact, sheaf cohomology coincides with singular cohomology).

\begin{theorem}
  \label{t:one-orbit-case-smoothness-and-quisos}
  Keep the above assumptions and assume in addition that
  $H$ is connected. The following conditions are equivalent:
  \begin{enumerate}
  \item 
    \label{enum:GmodH-G-smooth}
    $X=G/H$ is $G$-smooth.
  \item 
    \label{enum:smooth-via-pullback-forms}
    $\Gamma(\Omega_{X_G})=\Gamma(\Omega_{BH})$ is
    $\Gamma(\Omega_{BG})$-smooth.
  \item 
    \label{enum:pullback-forms-qiso}
    $\Gamma(\Omega_{BG}) \ra \Gamma(\Omega_{BH})$ is a quasi-isomorphism.
  \item 
    \label{enum:iso-on-equiv-coho}
    $H_G(\point) \ra H_G(G/H)=H_H(\point)$ is an isomorphism.
  \item 
    \label{enum:structure-morph-in-endos-erzeuger-qiso}
    If $E$ is an h-injective dg $c^*(\Omega_{BG})$-module that is
    isomorphic to $\Omega_{X_G}$ in $D(c^*(\Omega_{BG})),$
    the structure morphism $\Gamma(\Omega_{BG}) \ra
    (\calMod(c^*(\Omega_{BG})))(E)$
    is a quasi-isomorphism.
  \end{enumerate}
  (More equivalent conditions can be found in
  Propositions~\ref{p:smoothness-one-orbit-case}
  and \ref{p:smoothness-one-orbit-case-algebraic-groups}.)
\end{theorem}

\begin{proof}
  The equivalence of the two conditions 
  \ref{enum:pullback-forms-qiso} and
  \ref{enum:iso-on-equiv-coho}
  is obvious.

  Define the dg $\Gamma(\Omega_{BG})$ category
  $\mathcal{M}:=\calMod(c^*(\Omega_{BG})).$ 
  Since $H$ is connected the constant sheaf $\ul{\DR}_{X_G}$ on $X_G$
  is a classical 
  generator of $D^b_{G,c}(X)$ considered as a subcategory of $D(X_G)$ 
  (use \cite[Induction equivalence~2.6.3 and Prop.~2.7.2]{BL}).
  Hence $c^*(\Omega_{BG})$) is a classical
  generator of $D^b_{G,c}(X)$ considered as a subcategory of 
  $D(c^*(\Omega_{BG})).$ 
  Let $E$ be an h-injective dg $c^*(\Omega_{BG})$-module that is
  isomorphic to $\Omega_{X_G}$ in $D(c^*(\Omega_{BG})).$ Then there
  is a quasi-isomorphism $\varepsilon: \Omega_{X_G} \ra E$ in
  $C(c^*(\Omega_{BG})).$
  Then $X=G/H$ is $G$-smooth if and only if
  $\mathcal{M}(E)$ is $\Gamma(\Omega_{BG})$-smooth,
  by Lemma~\ref{l:test-smoothness-on-generator}.

  Pullback of differential forms 
  along $c: X_G \ra {BG}$
  defines a 
  (monomorphic) quasi-isomorphism
  of sheaves of dg algebras
  \begin{equation*}
    c^\sharp:c^*(\Omega_{BG}) \ra \Omega_{X_G}
  \end{equation*}
  since both sheaves are resolutions of $\ul{\DR}_{X_G}.$
  
  Consider the dg $\mathcal{M}(\Omega_{X_G})
  \otimes_{\Gamma(\Omega_{BG})} \mathcal{M}(E)^\opp$-module
  $\mathcal{M}(\Omega_{X_G},E),$ 
  \begin{align*}
    \mathcal{M}(E) \curvearrowright 
    \mathcal{M}(\Omega_{X_G},E) \curvearrowleft 
    \mathcal{M}(\Omega_{X_G}).
  \end{align*}
  Restriction along the morphism of dg
  $\Gamma(\Omega_{BG})$-algebras
  \begin{equation*}
    \lambda: \Gamma(\Omega_{X_G}) \ra \mathcal{M}(\Omega_{X_G}),
    \quad \omega \mapsto (\omega \wedge ?),
  \end{equation*}
  turns
  $\mathcal{M}(\Omega_{X_G},E)$ into
  an dg $\Gamma(\Omega_{X_G}) \otimes_{\Gamma(\Omega_{BG})}
  \mathcal{M}(E)^\opp$-module, 
  \begin{align*}
    \mathcal{M}(E) \curvearrowright 
    \mathcal{M}(\Omega_{X_G},E) \curvearrowleft 
    \Gamma(\Omega_{X_G}).
  \end{align*}
  Note that $\varepsilon \in Z^0(\mathcal{M}(\Omega_{X_G},E)).$
  We claim that the action maps
  \begin{align}
    \label{eq:left-action}
    (? \comp \varepsilon): \mathcal{M}(E) 
    & \ra \mathcal{M}(\Omega_{X_G},E) \quad \text{and}\\ 
    \label{eq:right-action}
    (\varepsilon \comp ?) \comp \lambda: \Gamma(\Omega_{X_G}) 
    & \ra \mathcal{M}(\Omega_{X_G},E)
  \end{align}
  are quasi-isomorphisms of
  dg $\Gamma(\Omega_{BG})$-modules. 
  (The idea of the proof of this claim is taken from
  \cite{soergel-private-note-2010}.)
  This is obvious for
  \eqref{eq:left-action} since
  $E$ is h-injective and
  $\varepsilon$ is an isomorphism in the derived category. 
  The action map 
  \eqref{eq:right-action}
  appears as the upper horizontal composition in the
  following diagram in $C(\Gamma(\Omega_{BG}))$:
  \begin{equation*}
    \xymatrix{
      {\Gamma(\Omega_{X_G})}
      \ar[r]^-{\lambda} 
      \gar[d] &
      {\mathcal{M}(\Omega_{X_G},\Omega_{X_G})}
      \ar[r]^-{\varepsilon \comp ?} &
      {\mathcal{M}(\Omega_{X_G},E)}
      \ar[dd]^-{? \comp c^\sharp} \\
      {\Gamma(\Omega_{X_G})}
      \ar[r]^-{\Gamma(\varepsilon)} 
      \ar[d]^-{\can}_-{\sim} &
      {\Gamma(E)}
      \ar[d]^-{\can}_-{\sim} \\
      {\mathcal{M}(c^*(\Omega_{BG}),\Omega_{X_G})} 
      \ar[r]^-{\varepsilon \comp ?} & 
      {\mathcal{M}(c^*(\Omega_{BG}),E)} 
      \gar[r] & 
      {\mathcal{M}(c^*(\Omega_{BG}),E)} 
    }
  \end{equation*}
  Note first that
  this diagram is commutative: For $\omega \in \Gamma(\Omega_{X_G})$ we
  have
  \begin{equation*}
    (\varepsilon \comp \lambda(\omega)\comp c^\sharp)(1) =
    \varepsilon(\omega)= ((\varepsilon \comp \can(\omega))(1).
  \end{equation*}
  Since $E$ is h-injective and $c^\sharp$ is a quasi-isomorphism, the
  right vertical morphism $(? \comp c^\sharp)$ is a quasi-isomorphism.
  So we have
  to show that $\Gamma(\varepsilon)$ is a quasi-isomorphism.
  It is enough to show that 
  $\Gamma(\res (\varepsilon)): \Gamma(\res(\Omega_{X_G})) \ra \Gamma(\res(E))$ is
  a quasi-isomorphism, where
  $\res=\res^{c^*(\Omega_{BG})}_{\ul{\DR}_{X_G}}.$

  The obvious map $\ul{\DR}_{X_G} \ra \res(\Omega_{X_G})$ and its
  composition
  $\ul{\DR}_{X_G} \ra \res(E)$ with $\res(\varepsilon)$ are
  quasi-isomorphisms.
  Since $\Omega_{X_G} =\res(\Omega_{X_G})$ is soft and $X_G$
  is paracompact,
  $\Gamma(\res(\Omega_{X_G}))$ computes the sheaf cohomology 
  $H({X_G}; \ul{\DR}_{X_G}),$
  and so does $\Gamma(\res(E))$ (note that $\res(E)$ is h-injective since
  the left adjoint of $\res$ preserves acyclics).
  This shows that $\Gamma(\varepsilon)$ is a quasi-isomorphism
  and
  proves that
  \eqref{eq:right-action} is a quasi-isomorphism.

  The fact that \eqref{eq:left-action} and \eqref{eq:right-action} are
  quasi-isomorphisms has the following two consequences:
  Firstly,
  the conditions \ref{enum:pullback-forms-qiso} 
  and
  \ref{enum:structure-morph-in-endos-erzeuger-qiso}
  are equivalent.
  Secondly,
  $\Gamma(\Omega_{X_G})$ is
  $\Gamma(\Omega_{BG})$-smooth if and only if 
  $\mathcal{M}(E)$ is $\Gamma(\Omega_{BG})$-smooth: this follows from
  \cite[Lemma~2.14]{lunts-categorical-resolution} and
  Lemma~\ref{l:smoothness-and-quasi-equis} (alternatively,
  it is easy
  to see that $\mathcal{M}(E)$ and $\Gamma(\Omega_{X_G})$ are dg
  Morita equivalent, and then one can use 
  Theorem~\ref{t:smoothness-preserved-by-dg-Morita-equivalence}).
  Hence \ref{enum:GmodH-G-smooth} and
  \ref{enum:smooth-via-pullback-forms} are equivalent.

  It remains to show that \ref{enum:smooth-via-pullback-forms}
  and \ref{enum:pullback-forms-qiso} are equivalent.
  We know that $K:=H(\Gamma(\Omega_{BG}))=H_G(\point)$ is a polynomial
  ring over $\DR$ in finitely many variables of positive even degrees
  (cf.\ proof of Proposition~\ref{p:smoothness-one-orbit-case} below).
  Since $\Gamma(\Omega_{BG})$ is graded commutative there is a
  quasi-isomorphism of dg ($\DR$-)algebras $K
  \ra \Gamma(\Omega_{BG})$ inducing the identity on cohomology.
  By Theorem~\ref{t:smoothness-and-base-change}, 
  part~\ref{enum:smoothness-and-qiso-base-change-res},
  $\Gamma(\Omega_{BG})$-smoothness of  
  $\Gamma(\Omega_{X_G})=\Gamma(\Omega_{BH})$
  is equivalent to
  $K$-smoothness of $\Gamma(\Omega_{BH}).$
  This latter condition is
  equivalent to $K \ra \Gamma(\Omega_{BH})$ being a quasi-isomorphism,
  by 
  Proposition~\ref{p:smoothness-over-local-graded-finite-homological-dim-algebras}
  (the assumptions there are satisfied by the proof of Proposition~\ref{p:smoothness-one-orbit-case}), hence to $\Gamma(\Omega_{BG})
  \ra \Gamma(\Omega_{BH})$ being a quasi-isomorphism.
\end{proof}

\subsection{Reduction to homogeneous spaces}
\label{sec:reduct-homog-spac}

\begin{theorem}
  \label{t:X-G-smooth-iff-all-orbits-G-smooth}
  Let $G$ be a connected complex affine algebraic group
  and $X$ a complex $G$-variety.
  Assume that $X$ consists of finitely many $G$-orbits and that 
  all stabilizers (in $G$ of points in $X$) are 
  connected. The following conditions are equivalent:
  \begin{enumerate}
  \item 
    $X$ is $G$-smooth.
  \item 
    All $G$-orbits in $X$ are $G$-smooth.
  \end{enumerate}
\end{theorem}

\begin{proof}
  Assume that $U$ is an open $G$-orbit in $X,$ and let $F$ be its complement.

  Let $\mathcal{A}_{X_G}=c^*(\Omega_{BG}),$ where $c: X_G \ra BG.$
  Diagram
  \eqref{eq:six-functors-dg-modules-over-dg-sheaf-of-algs}
  for $\groundring=\DR$ yields the following diagram
  (we write $i:F_G \ra X_G$ instead of $i_G$ etc.):
  \begin{equation*}
    \xymatrix@=45pt{
      {D(F_G)} \ar[r]|{i_{\as}} &
      \ar@/_1pc/[l]_{i^*} 
      \ar@/^1pc/[l]^{Ri^!} 
      {D(X_G)} 
      \ar[r]|{j^{\as}} 
      & 
      \ar@/_1pc/[l]_{j_!} 
      \ar@/^1pc/[l]^{Rj_*} 
      {D(U_G)}
      \\
      {D(\mathcal{A}_{F_G})} 
      \ar[u]^-{\res^{\mathcal{A}_{F_G}}_{\ul{\DR}_{F_G}}}_-{\sim}
      \ar[r]|{i_{\as}} &
      \ar@/_1pc/[l]_{i^*} 
      \ar@/^1pc/[l]^{Ri^!} 
      {D(\mathcal{A}_{X_G})} 
      \ar[u]^-{\res^{\mathcal{A}_{X_G}}_{\ul{\DR}_{X_G}}}_-{\sim}
      \ar[r]|{j^{\as}} 
      & 
      \ar@/_1pc/[l]_{j_!} 
      \ar@/^1pc/[l]^{Rj_*} 
      {D(\mathcal{A}_{U_G})}
      \ar[u]^-{\res^{\mathcal{A}_{U_G}}_{\ul{\DR}_{U_G}}}_-{\sim}
    }
  \end{equation*}
  We have explained its properties above. We claim that all functors
  in the upper row induce the following functors between the
  equivariant derived categories (and then it is clear that they
  coincide with the usual functors defined on the equivariant level).
  \begin{equation*}
    \xymatrix{
      {D^b_{G,c}(F)} 
      \ar[r]|{i_{\as}} &
      \ar@/_1pc/[l]_{i^*} 
      \ar@/^1pc/[l]^{Ri^!} 
      {D^b_{G,c}(X)} 
      \ar[r]|{j^{\as}} 
      & 
      \ar@/_1pc/[l]_{j_!} 
      \ar@/^1pc/[l]^{Rj_*}
      {D^b_{G,c}(U)}
    }
  \end{equation*}
  This is obvious for
  $i^*$ and $j^{\as}.$
  Theorem~\ref{t:derived-fibered-base-change} implies that it is
  true for $Rj_*$ and $i_{\as}$ (for constructibility use
  \cite{kaloshin-stratification}, which implies that the decomposition
  of $X$ into $G$-orbits is a Whitney stratification).
  To get the result for $Ri^!$ let $\mathcal{F} \in D^b_{G,c}(X)
  \subset D(X_G),$ and consider the triangle
  \begin{equation*}
    i_{\as}(Ri^!(\mathcal{F})) \ra 
    \mathcal{F} \ra 
    Rj_*(j^{\as}(\mathcal{F})) \ra 
    [1]i_{\as}(Ri^!(\mathcal{F})).
  \end{equation*}
  We already know that $\mathcal{F}$ and $Rj_*(j^{\as}(\mathcal{F}))$ are in
  $D^b_{G,c}(X).$ Hence 
  $i_{\as}(Ri^!(\mathcal{F})) \in D^b_{G,c}(X)$ and 
  $Ri^!(\mathcal{F}) \sila i^*(i_{\as}(Ri^!(\mathcal{F}))) \in
  D^b_{G,c}(F).$ Similarly one gets the result for $j_!.$

  Let 
  $\mathcal{U}:= \ul{\DR}_{U_G}.$ We have seen at
  the beginning of the proof of
  Theorem~\ref{t:one-orbit-case-smoothness-and-quisos} 
  that this is a classical generator of $D^b_{G,c}(U).$
  If $\mathcal{F}$ is a classical generator of $D^b_{G,c}(F),$ then
  $\{ i_{\as} \mathcal{F}, j_! \mathcal{U}\}$ 
  (or $i_{\as} \mathcal{F}\oplus j_! \mathcal{U}$)
  classically generates $D^b_{G,c}(X)$:
  Any object $\mathcal{X} \in D_{G,c}^b(X)$ fits into a
  (distinguished) triangle 
  \begin{equation*}
    j_! j^{\as} \mathcal{X} \ra \mathcal{X} \ra i_{\as}i^* \mathcal{X} \ra
    [1]j_! j^{\as} \mathcal{X},
  \end{equation*}
  and $j^{\as}\mathcal{X} \in D^b_{G,c}(U)$ 
  (resp.\ $i^*\mathcal{X} \in D^b_{G,c}(F)$) is in the subcategory
  classically generated by $\mathcal{U}$ (resp.\ $\mathcal{F}$).
  This argument and an induction on the number of $G$-orbits also
  shows that $D^b_{G,c}(F)$ has a classical generator; we fix such a
  generator $\mathcal{F}.$

  From Lemma~\ref{l:suitable-h-injective-lifts} we obtain (where
  $\res$ denotes the obvious restriction functors): There is an
  h-injective dg $\mathcal{A}_{F_G}$-module $\mathcal{F}'$
  such that $\res(\mathcal{F}') \cong \mathcal{F}$ in $D(F_G)$ and
  such that $i_{\as}\mathcal{F}'$ is h-injective and
  $\res (i_{\as}\mathcal{F}') \cong i_{\as}\mathcal{F}$ in $D(X_G).$
  There is an h-injective dg $\mathcal{A}_{X_G}$-module $\mathcal{V}$
  such that $\res(\mathcal{V}) \cong j_! \mathcal{U}$ in $D(X_G)$ and
  such that $j^{\as}(\mathcal{V})$ is h-injective and satisfies 
  $\mathcal{A}_{U_G} 
  = \mathcal{U} \otimes_{\ul{\DR}_{U_G}} \mathcal{A}_{U_G} 
  \cong j^{\as}(\mathcal{V})$ in
  $D(\mathcal{A}_{U_G})$
  and
  $\res(j^{\as}(\mathcal{V})) \cong \mathcal{U}$ in $D(U_G).$

  Let $\mathcal{E}$ be the full subcategory of
  $\calMod(\mathcal{A}_{X_G})$ whose objects are $i_{\as} \mathcal{F}'$ and
  $\mathcal{V}.$ We write this $\Gamma(\Omega_{BG})$-category
  $\mathcal{E}$ symbolically as 
  \begin{equation*}
    \begin{bmatrix}
      \mathcal{E}(i_{\as} \mathcal{F}')
      & \mathcal{E}(\mathcal{V},
      i_{\as}\mathcal{F}')\\
      \mathcal{E}(i_{\as}\mathcal{F}', \mathcal{V}) &
      \mathcal{E}(\mathcal{V})
    \end{bmatrix}.
  \end{equation*}
  From
  Remark~\ref{rem:dg-algebras-and-bimodule-versus-general-setting}
  and
  Lemma~\ref{l:test-smoothness-on-generator} (applied to
  $i_{\as}\mathcal{F}' \oplus \mathcal{V}$)
  we see that
  $X$ is
  $G$-smooth if and only if $\mathcal{E}$ is 
  $\Gamma(\Omega_{BG})$-smooth.
  
  Define $\mathcal{E}' \subset \mathcal{E}$ to be the 
  subcategory with the same objects and morphisms except that 
  $\mathcal{E}'(\mathcal{V}, i_{\as}\mathcal{F}'):=0,$ symbolically
  \begin{equation*}
    \mathcal{E}'=
    \begin{bmatrix}
      \mathcal{E}(i_{\as} \mathcal{F}')
      & 0 \\
      \mathcal{E}(i_{\as}\mathcal{F}', \mathcal{V}) &
      \mathcal{E}(\mathcal{V})
    \end{bmatrix}.
  \end{equation*}
  Then the obvious morphism $\mathcal{E} \ra \mathcal{E}'$ of
  dg $\Gamma(\Omega_{BG})$-categories is a quasi-equivalence:
  We only need to show that
  $\mathcal{E}(\mathcal{V}, i_{\as}\mathcal{F}')$ is acyclic. But
  \begin{align*}
    H(\mathcal{E}(\mathcal{V}, i_{\as}\mathcal{F}')) 
    & = (\mathcal{H}(\mathcal{A}_{X_G}))(\mathcal{V}, i_{\as}\mathcal{F}') \\
    \text{(since $i_{\as}\mathcal{F}'$ is h-injective)} & \sira
    (D(\mathcal{A}_{X_G}))(\mathcal{V}, i_{\as}\mathcal{F}') \\
    \text{(since $\res$ is an equivalence)} & \sira
    (D({X_G}))(\res \mathcal{V}, \res i_{\as}\mathcal{F}') \\
    & \cong
    (D({X_G}))(j_!\mathcal{U}, i_{\as} \mathcal{F}) \\
    & \cong
    (D({X_G}))(i^*j_!\mathcal{U}, \mathcal{F}) \\
    \text{(since $i^*j_!=0$)} &=0.
  \end{align*}
  
  Lemma~\ref{l:smoothness-and-quasi-equis} and
  Theorem~\ref{t:diagonal-smooth-bimod-smooth-iff-extension-smooth}
  show that $G$-smoothness of $X$
  is equivalent to
  \begin{enumerate}[label=(\alph*')]
  \item 
    \label{enum:end-F-smooth}
    $\mathcal{E}'(i_{\as}\mathcal{F}')$ is
    $\Gamma(\Omega_{BG})$-smooth, and
  \item 
    \label{enum:end-V-smooth}
    $\mathcal{E}'(\mathcal{V})$ is
    $\Gamma(\Omega_{BG})$-smooth, and
  \item 
    \label{enum:bimodule-F-sweet}
    $\mathcal{E}'(i_{\as}\mathcal{F}', \mathcal{V})$ is
    $\Gamma(\Omega_{BG})$-\sweet{} as a dg 
    $\mathcal{E}'(i_{\as}\mathcal{F}')
    \otimes_{\Gamma(\Omega_{BG})}
    \mathcal{E}'(\mathcal{V})^\opp$-module.
  \end{enumerate}
  We claim that these three conditions are equivalent to the following
  two conditions
  \begin{enumerate}[label=(\alph*'')]
  \item 
    \label{enum:F-G-smooth}
    $F$ is $G$-smooth, and
  \item 
    \label{enum:U-G-smooth}
    $U$ is $G$-smooth.
  \end{enumerate}

  \ref{enum:end-F-smooth} $\Leftrightarrow$
  \ref{enum:F-G-smooth}:
  Condition~\ref{enum:end-F-smooth}
  is equivalent to
  $(\calMod(\mathcal{A}_{X_F}))(\mathcal{F}')$ being
  $\Gamma(\Omega_{BG})$-smooth (since $i_{\as}$ is fully faithful), and
  hence to $G$-smoothness of $F,$ by Lemma~\ref{l:test-smoothness-on-generator}.

  \ref{enum:end-V-smooth}
  $\Leftrightarrow$
  \ref{enum:U-G-smooth}:
  Note that 
  \begin{equation}
    \label{eq:j-upper-star-qiso-on-dgend-V}
    j^{\as}:\mathcal{E}'(\mathcal{V})=
    (\calMod(\mathcal{A}_{X_G}))(\mathcal{V}) \ra
    (\calMod(\mathcal{A}_{U_G}))(j^{\as}\mathcal{V})
  \end{equation}
  is a quasi-isomorphism: on the $p$-th cohomology it is given by
  \begin{equation*}
    j^{\as}:
    (D(\mathcal{A}_{X_G}))(\mathcal{V}, [p]\mathcal{V}) \ra
    (D(\mathcal{A}_{U_G}))(j^{\as}\mathcal{V}, [p] j^{\as}\mathcal{V})
  \end{equation*}
  which becomes identified (using the equivalences $\res,$ 
  $\res j^{\as} \cong j^{\as}\res,$ and 
  $\res \mathcal{V}\cong j_! \mathcal{U}$)
  with
  \begin{equation*}
    j^{\as}:
    (D(X_G))(j_! \mathcal{U}, [p]j_!\mathcal{U}) \ra
    (D(U_G))(j^{\as}j_!\mathcal{U}, [p] j^{\as}j_!\mathcal{U})
  \end{equation*}
  and is an isomorphism since $j_!$ is fully faithful.
  Hence 
  condition \ref{enum:end-V-smooth} is equivalent to
  $\Gamma(\Omega_{BG})$-smoothness of
  $(\calMod(\mathcal{A}_{U_G}))(j^{\as}\mathcal{V})$
  (Lemma~\ref{l:smoothness-and-quasi-equis}).
  Now again
  use
  Lemma~\ref{l:test-smoothness-on-generator} and the fact that
  $j^{\as}\mathcal{V}$ is h-injective.

  \ref{enum:U-G-smooth} $\Rightarrow$
  \ref{enum:bimodule-F-sweet}:
  Assume that \ref{enum:U-G-smooth} holds.
  Since $\mathcal{A}_{U_G} \cong j^{\as}(\mathcal{V})$ 
  in $D(\mathcal{A}_{U_G}),$
  Theorem~\ref{t:one-orbit-case-smoothness-and-quisos}
  implies that the structure morphism
  $\Gamma(\Omega_{BG}) \ra
  (\calMod(\mathcal{A}_{U_G}))(j^{\as}(\mathcal{V}))$
  is a quasi-isomorphism.
  By the above quasi-isomorphism 
  \eqref{eq:j-upper-star-qiso-on-dgend-V}
  this is equivalent to 
  $\Gamma(\Omega_{BG}) \ra 
  \mathcal{E}'(\mathcal{V})$ being a
  quasi-isomorphism. 
  Hence
  Corollary~\ref{c:bimodule-smoothness-quasi-isomorphism-invariance}
  shows that 
  condition \ref{enum:bimodule-F-sweet}
  is equivalent to
  \begin{equation*}
    \mathcal{E}'(i_{\as}\mathcal{F}', \mathcal{V}) 
    \in \per(\mathcal{E}'(i_{\as}\mathcal{F}')).
  \end{equation*}
  Applying the adjunction $(i_{\as}, i^!)$ (on the dg level) and using
  that $i_{\as}$ is fully faithful we see
  that this is equivalent to
  \begin{equation}
    \label{eq:connecting-bimodule-sweet-reformulated}
    (\calMod(\mathcal{A}_{X_F}))(\mathcal{F}', i^!\mathcal{V}) 
    \in \per((\calMod(\mathcal{A}_{F_G}))(\mathcal{F}')).
  \end{equation}

  If we view $D^b_{G,c}(F)$ as a full subcategory of
  $D(\mathcal{A}_{F_G}),$ then
  equivalence~\eqref{eq:homotopy-cat-h-injectives-isom-derived-cat-dg-sheaves-version-quasi-inverse} 
  (cf.\ Lemma~\ref{l:test-smoothness-on-generator}) 
  and
  Corollary~\ref{c:test-smoothness-on-classical-generator}
  show that
  \begin{align*}
    D^b_{G,c}(F) & \sira 
    \per((\calMod(\mathcal{A}_{F_G}))(\mathcal{F}')),\\
    \mathcal{G} & \mapsto (\calMod(\mathcal{A}))(\mathcal{F}', \iota(\mathcal{G})),
  \end{align*}
  is an equivalence.

  Since $Ri^!$ preserves the equivariant derived categories and
  $\mathcal{V}$ is h-injective, we have
  $i^!(\mathcal{V}) \cong Ri^!(\mathcal{V}) \in  D^b_{G,c}(F).$
  Since $i_{\as}$ is left adjoint to $i^!$ and preserves acyclics,
  $i^!(\mathcal{V})$ is h-injective. Hence $i^!(\mathcal{V}) \cong
  \iota(i^!(\mathcal{V}))$ already in the homotopy category and therefore
  \begin{equation*}
    (\calMod(\mathcal{A}))(\mathcal{F}', i^!(\mathcal{V})) \cong 
    (\calMod(\mathcal{A}))(\mathcal{F}', \iota(i^!(\mathcal{V})))
    \in \per((\calMod(\mathcal{A}_{F_G}))(\mathcal{F}')).
  \end{equation*}
  This shows \eqref{eq:connecting-bimodule-sweet-reformulated} and hence
  that condition \ref{enum:bimodule-F-sweet} is satisfied.

  These arguments show that $X$ is $G$-smooth if
  and only if 
  $U$ and $F$ are $G$-smooth.
  An induction on the number of $G$-orbits in $X$ finishes the
  proof.
\end{proof}

\subsection{Results concerning the case of a homogeneous space}
\label{sec:results-for-homogeneous-space}

Let $H$ be a closed connected subgroup of a connected (real) Lie group $G.$
The inclusion morphism $H \ra G$ gives rise to the morphism
$H_G(\point) \ra H_H(\point)$ on equivariant cohomology which is a morphism
of (graded commutative) dg $\DR$-algebras (with differential zero).

\begin{proposition}
  \label{p:smoothness-one-orbit-case}
  Let $H$ be a closed connected subgroup of a connected Lie
  group $G.$ Then the following conditions are equivalent:
  \begin{enumerate}
  \item 
    \label{enum:HH-smooth-HG-algebra}
    $H_H(\point)$ is a smooth dg $H_G(\point)$-algebra;
  \item 
    \label{enum:iso-on-equiv-coho-point}
    $H_G(\point) \ra H_H(\point)$ is an isomorphism;
  \item
    \label{enum:max-compacts-coincide}
    a/any maximal compact subgroup of $H$ is a maximal compact subgroup of $G.$
  \item
    \label{enum:quotient-trivial-coho}
    $H(G/H):=H^*(G/H;\DR)=\DR.$ 
  \end{enumerate}
\end{proposition}
 
\begin{proof}
  Recall the following facts 
  (see \cite{hochschild-structure-lie-groups,
    borel-sous-groupes-compact-max}):
  Any connected (real) Lie group $G'$ has maximal compact subgroups and 
  any compact subgroup is contained in one of them; they are connected
  and any two of them are conjugate by an inner automorphism.
  If $K'$ is a maximal compact subgroup, $G'$ is
  homeomorphic to $K' \times \DR^l$ for some $l,$ and
  the quotient $G'/K'$ is homeomorphic to $\DR^l.$

  We write $H_G$ instead of $H_G(\point),$ and similar for other groups.
  Let $M$ be a maximal compact subgroup of $H,$ and $L$ a maximal
  compact subgroup of $G$ containing $M.$
  Let $T$ be a maximal torus in $M,$ and $S$ a maximal torus in $L$
  containing $T$:
  \begin{equation*}
    \xymatrix{
      {S} \ar@{}[r]|-{\subset} &
      {L} \ar@{}[r]|-{\subset} &
      {G} \\
      {T} \ar@{}[r]|-{\subset} \ar@{}[u]|-{\cup} &
      {M} \ar@{}[r]|-{\subset} \ar@{}[u]|-{\cup} &
      {H} \ar@{}[u]|-{\cup} 
    }
  \end{equation*}
  Since $G \cong L \times \DR^g$ and
  $G/L \cong \DR^g$ for some $g \in \DN$ we have
  $H_G=H_L.$
  Let $W_L$ be the Weyl group of $(L,S).$ It acts 
  on 
  $\Sym((\Lie S)^*),$ the space of real valued polynomial functions on
  the Lie algebra $\Lie S$ of $S.$ We have canonically $\Sym((\Lie S)^*)=H_S.$ 
  From \cite[\S27,
  \S28]{borel-cohomologie-espaces-fibres-et-homogenes-de-lie-gp-cpt}
  we know that $H_L=H_S^{W_L}$ and that this is a polynomial ring
  (over $\DR$) in $s:=\dim_{\DR} S$ variables of positive even
  degrees; similarly $H_H=H_M=H_T^{W_M}$ is a polynomial ring 
  in $t:= \dim_{\DR} T$ 
  variables of positive even degrees; furthermore, 
  the inclusion $\Lie T \subset \Lie S$
  induces a morphism
  $H_S^{W_L} \ra H_T^{W_M}$ which coincides with $H_G \ra H_H$ under
  our identifications.

  Now it is clear that \ref{enum:max-compacts-coincide} implies
  \ref{enum:iso-on-equiv-coho-point}. Conversely,
  \ref{enum:iso-on-equiv-coho-point} implies that $S=T$ 
  (if $T \subsetneq S,$ choose $0\not=\chi \in (\Lie S)^*$ such that
  $\chi|_{\Lie T}=0$; then $\prod_{w \in W_L} w(\chi)$ is in
  $H_S^{W_L}$ and nonzero but goes to zero in $H_T^{W_M}$)
  and $W_L=W_M$ (since we know $S=T$ it is clear that $W_M
  \subset W_L$; by \cite[Prop.~3.6]{Humphreysreflect} 
  $H_S$ is a free
  $H_S^{W_L}$-module of rank $|W_L|,$ 
  and
  $H_T$ is a free
  $H_T^{W_M}$-module of rank $|W_M|$; since $H_S=H_T$ and
  $H_S^{W_L}=H_T^{W_M}$ we must have $W_M=W_L$),
  and hence $L=M$ (see e.\,g.\ \cite[Thm.~6.2]{dwyer-wilkerson}); this
  yields 
  \ref{enum:max-compacts-coincide} since $M$ was an arbitrary maximal
  compact subgroup of $H.$

  Proposition~\ref{p:smoothness-over-local-graded-finite-homological-dim-algebras} 
  can be applied to $\groundring=\DR,$ $K=H_G=H_S^{W_L}$
  (which has global dimension $s,$ cf.\ \cite[X.\S8.6,
  Cor.~2]{bourbaki-algebre-chap-10-algebre-homologique}), and 
  the dg $K$-algebra $A=H_H=H_T^{W_M}$ (note that $H(A)=A$ is a
  finitely generated $K$-module since $H_S$ is a finitely generated module over
  the Noetherian ring $K$ and has $A$ as a subquotient).
  This shows that 
  \ref{enum:HH-smooth-HG-algebra} and
  \ref{enum:iso-on-equiv-coho-point} are equivalent.

  If $H(G/H)=\DR,$ then the Leray-Hirsch theorem can be applied to
  $EG/H \ra EG/G$ and shows that $H_G \ra H_H$ is an isomorphism.
  Hence \ref{enum:quotient-trivial-coho} implies
  \ref{enum:iso-on-equiv-coho-point}.

  It remains to show 
  that \ref{enum:max-compacts-coincide}
  implies \ref{enum:quotient-trivial-coho}.
  As above let $M$ be a maximal compact subgroup of $H.$ Assume
  that $M$ is also a maximal compact subgroup of $G.$ 
  Since $H/M \cong \DR^h$ and $G/M \cong \DR^g$ for suitable $g, h \in \DN,$ the long exact
  sequence of homotopy groups of the fiber bundle $G/M \ra G/H$ with
  fiber $H/M$ shows that $\pi_n(G/H)=\{1\}$ for all $n \in \DN.$
  The Hurewicz theorem then shows that $H(G/H)=\DR.$
\end{proof}

\begin{proposition}
  \label{p:smoothness-one-orbit-case-algebraic-groups}
  Let $G$ be a connected complex affine algebraic group and 
  $H \subset G$ a closed connected subgroup.
  Then the equivalent conditions
  of Proposition~\ref{p:smoothness-one-orbit-case}
  are equivalent to the following condition:
  \begin{enumerate}[label=(e)]
  \item 
    \label{ref:GmodH-isomorph-to-Cn}
    $G/H \cong \DC^n$ as complex varieties for some $n \in \DN.$
  \end{enumerate}
\end{proposition}

\begin{proof}
  Obviously   \ref{ref:GmodH-isomorph-to-Cn} implies
  \ref{enum:quotient-trivial-coho}.
  Assume that \ref{enum:max-compacts-coincide} is satisfied.
  We claim that the unipotent radical $U$ of $G$ acts
  transitively on $G/H.$ 
  Let $S \subset G$ be a Levi subgroup, and $L \subset S$ a maximal
  compact subgroup. Our assumption implies that $gLg\inv \subset H$
  for some $g \in G.$ Since $L$ is Zariski-dense in $S$ this implies
  that already $S':=gSg\inv \subset H.$ But then $G=U S' = U H$ and
  hence $U/(U \cap H) \sira G/H.$
  Now use the (presumably well known) Lemma~\ref{l:quotient-unipotent-group}.
\end{proof}

We could not find a good reference for the following result.

\begin{lemma}
  \label{l:quotient-unipotent-group}
  Let $V \subset U$ be a closed subgroup of a unipotent complex affine
  algebraic group. Then $U/V \cong \DC^n$ as complex varieties for
  some $n \in \DN.$
\end{lemma}

\begin{proof}
  (We learned this proof from Hanspeter Kraft
  \cite{kraft-email-2011}.)
  We prove this by an outer induction on $\dim(U)$ and an inner
  induction on $\dim(U/V),$ the cases $\dim(U)=0$ and $\dim(U/V)=0$ being
  trivial. So assume that $\dim(U/V)>0.$
  
  Let $Z$ be the (nontrivial) center of $U.$
  If $Z \subset V,$ then $U/V \cong (U/Z)/(V/Z)$ and we can use induction for
  $V/Z \subset U/Z.$ So assume that $Z \subsetneq V.$

  Claim: There is a closed subgroup  $V \subset V' \subset U$ such
  that $V$ is normal in $V'$ and $V'/V \cong (\DC,+).$
  
  Let $V'':= VZ.$ This is a closed subgroup satisfying $V
  \subsetneq V'' \subset U.$ Note that 
  $Z/(Z\cap V) \cong V''/V,$ and these groups are abelian and
  unipotent. Hence they are isomorphic to a finite product of additive
  groups $(\DC,+).$ Let $V'$ be the inverse image of a one-dimensional
  subgroup of $V''/V.$ This proves the claim.

  The operation of $V'$ on $U$ by right multiplication induces an
  operation of $V'/V \cong (\DC,+)$ on $U/V$ with quotient $U/V',$ and
  $U/V \ra U/V'$ is a principal $(\DC,+)$-bundle. 
  Every such bundle
  over an affine variety is trivial (see
  \cite{serre-espaces-fibres-algebriques}
  or 
  \cite[Ch.~IV]{kraft-schwarz-reductive-group-actions-one-dim-quotient}).
  By induction we know $U/V'\cong \DC^m$ for some $m \in \DN.$
\end{proof}

\appendix

\section{Derived fibered base change}
\label{sec:derived-fibered-base}

We provide a proof of Theorem~\ref{t:derived-fibered-base-change} 
in this appendix.
This theorem is a modification of
\cite[p.~56, Lemma~C1]{BL}; the proof we present is
essentially copied from \cite{soergel-private-note-2010}.
We thank W.~Soergel for giving his consent.

Sheaves are sheaves of $\groundring$-modules, for $\groundring$ an
arbitrary ring. We denote the category of sheaves on a topological
space $X$ by $\Sh(X).$

\begin{lemma}
  \label{l:open-surjective-adjunction-isom}
  Let $f: X \sra Y$ be an open and surjective morphism of topological
  spaces (for example a projection or a locally trivial fiber
  bundle). If all fibers of $f$ are connected, the  
  adjunction morphism
  $\mathcal{G} \ra f_*f^*\mathcal{G}$
  is an isomorphism
  for any sheaf $\mathcal{G}$ on $Y.$
\end{lemma}
  
\begin{proof}
  If $U \subset Y$ is open we have to show that 
  $\mathcal{G}(U) \ra f^*\mathcal{G}(f\inv(U)$ is bijective.
  It is convenient for this to work with the {\'e}tale space associated
  to a sheaf. In this picture we have a pullback diagram
  \begin{equation*}
    \xymatrix{
      {f\inv(U) \times_U \mathcal{G}|_U} \ar@{}[r]|-{=} & 
      {f^*(\mathcal{G}|_U)} \ar[d]^{p'} \ar[r]^-{f'} &
      {\mathcal{G}|_U} \ar[d]^{p} \\
      &
      {f\inv(U)} \ar[r]^-f 
      &
      {U}
    }
  \end{equation*}
  and we need to show that continuous sections of $p$ correspond to
  continuous sections of $p'$ via $t \mapsto (\id_{f\inv(U)}, tf).$
  This map is obviously injective. If $s=(\id, \tau)$ is a continuous
  section of $p',$ then $f=p\tau,$ and the restriction of $\tau$ to
  any fiber of $f$ is constant. Since $Y$ carries the final topology
  with respect to $f: X \ra Y,$ this implies that there is a continuous
  map $s: U \ra \mathcal{G}|_U$ such that $sf=\tau,$ and hence
  $ps=\id_U.$
\end{proof}  

Let (P) be a property of topological spaces.
We say that a topological space $X$ is \define{open-locally (P)} if
any neighborhood of any point $x \in X$ contains an open neighborhood
$U$ of $x$ in $X$ such that $U$ has property (P).

Let
\begin{equation}
  \label{eq:cartesian-derived-fibered-BC}
  \xymatrix{
    {W} \ar[r]^g \ar[d]^q &
    {V} \ar[d]^p \\
    {Y} \ar[r]^f &
    {X}
  }
\end{equation}
be a Cartesian diagram of topological spaces. 

\begin{theorem}
  [{Fibered base change, cf.\ \cite{soergel-private-note-2010}.}]
  \label{t:fibered-base-change}
  In the above setting 
  \eqref{eq:cartesian-derived-fibered-BC}
  assume that $p$ and $q$ are locally trivial
  fiber bundles with open-locally connected fiber. Then the obvious
  natural transformation
  \begin{equation*}
    p^* \comp f_* \ra g_* \comp q^*
  \end{equation*}
  is an isomorphism of functors $\Sh(Y) \ra \Sh(V).$
\end{theorem}

\begin{proof}
  We can assume without loss of generality that $p$ is globally
  trivial. Then 
  \eqref{eq:cartesian-derived-fibered-BC}
  becomes
  \begin{equation}
    \label{eq:cartesian-fibered-BC-trivial}
    \xymatrix@C3cm{
      {W=Y \times Z} \ar[r]^{g=f\times \id_Z} \ar[d]^q &
      {V=X \times Z} \ar[d]^p \\
      {Y} \ar[r]^f &
      {X,}
    }
  \end{equation}
  where $Z$ is some open-locally connected topological space, and
  $p$ and $q$ are the first projections. Let $\mathcal{F} \in \Sh(Y).$
  It is enough to show that for any point $(x, z) \in X \times Z$ the
  morphism
  \begin{equation*}
    (p^* f_* \mathcal{F})_{(x,z)} \ra (g_* q^* \mathcal{F})_{(x,z)}
  \end{equation*}
  is an isomorphism.
  The open neighborhoods of $(x,z)$ in $X \times Z$ have a cofinal
  subsystem formed by neighborhoods of the form $X' \times Z',$ where
  $X'$ is an open neighborhood of $x$ in $X$ and $Z'$ is an open 
  connected
  neighborhood of $z$ in $Z.$ 
  Hence it is enough to show that
  \begin{equation*}
    (p^*f_* \mathcal{F})(X \times Z) \ra 
    (g_*q^*\mathcal{F})(X \times Z)
  \end{equation*}
  is an isomorphism if we assume in addition that $Z$ is connected.
  In this case 
  Lemma~\ref{l:open-surjective-adjunction-isom} shows that
  \begin{equation*}
    (p^*f_* \mathcal{F})(X \times Z) =
    (p_*p^*f_* \mathcal{F})(X) \sila (f_*\mathcal{F})(X)= \mathcal{F}(Y)
  \end{equation*}
  and 
  \begin{equation*}
    (g_*q^*\mathcal{F})(X \times Z) = 
    (q^*\mathcal{F})(Y \times Z) =
    (q_*q^*\mathcal{F})(Y) \sila \mathcal{F}(Y).
  \end{equation*}
  Under these identifications the above morphism corresponds to the
  identity of $\mathcal{F}(Y).$ 
\end{proof}

A map $f: X \ra Y$ is called \define{pre-$\infty$-acyclic} if the
adjunction morphism $\mathcal{F} \ra Rf_*(f^*(\mathcal{F}))$ is an
isomorphism for any any sheaf $\mathcal{F}$ on $Y.$
We say that a topological space is \define{pre-$\infty$-acyclic} if $X \ra
\point$ is pre-$\infty$-acyclic.

\begin{theorem}
  [{Derived fibered base change, cf.\ 
    \cite[p.~56, Lemma~C1]{BL}, \cite{soergel-private-note-2010}}]
  \label{t:derived-fibered-base-change}
  In the above setting 
  \eqref{eq:cartesian-derived-fibered-BC}
  assume that $p$ and $q$ are locally trivial
  fiber bundles with open-locally pre-$\infty$-acyclic fiber. 
  Then the obvious natural transformation
  \begin{equation*}
    p^* \comp Rf_* \ra Rg_* \comp q^*
  \end{equation*}
  is an isomorphism of functors $D^+(Y) \ra D^+(V).$
\end{theorem}

\begin{proof}
  Again we can assume that $p$ is globally trivial, so that
  \eqref{eq:cartesian-derived-fibered-BC} is given by
  \eqref{eq:cartesian-fibered-BC-trivial}
  where $Z$ is now an open-locally pre-$\infty$-acyclic topological space.

  \textbf{Step 1:} 
  Assume in the situation 
  \eqref{eq:cartesian-fibered-BC-trivial}
  that $Y$ has the discrete topology. If $\mathcal{F}$ is a sheaf on
  $Y,$ then 
  $q^*(\mathcal{F})$ is
  $g_*$-acyclic.

  Fix $(x,z) \in X \times Z.$
  It is enough to show that the stalk
  $(R^ig_*(q^*(\mathcal{F})))_{(x,z)}$ vanishes for all $i > 0.$
  We can compute this stalk as the colimit of
  \begin{equation*}
    U \mapsto H^i(g\inv(U); q^*(\mathcal{F})),
  \end{equation*}
  where $U$ ranges over the open neighborhoods of $(x,z)$ in $X \times
  Z$ of the form $U=X' \times Z',$ where
  $X'$ is an open neighborhood of $x \in X$ and $Z'$ is an open 
  pre-$\infty$-acyclic
  neighborhood of $z$ in $Z.$ 
  Fix such a neighborhood $U= X' \times Z',$ and put $Y':= f\inv(X').$
  Then
  \begin{equation*}
    H^i(g\inv(U); q^*(\mathcal{F}))=
    H^i(Y' \times Z'; q^*(\mathcal{F}))
    = \prod_{y' \in Y'} H^i(\{y'\} \times Z'; q^*(\mathcal{F})|_{\{y'\}
      \times Z'})
  \end{equation*}
  If $c: Z' \ra \point$ is the projection, then, for every $y' \in
  Y',$ we have
  \begin{equation*}
     H^i(\{y'\} \times Z'; q^*(\mathcal{F})|_{\{y'\} \times Z'}) 
     = H^i(Z'; c^*(\mathcal{F}_{y'}))
     = H^i(Rc_*(c^*(\mathcal{F}_{y'}))),
  \end{equation*}
  and this vanishes for $i>0$ by the assumption on $Z'.$

  \textbf{Step 2:}
  Let $d:Y' \ra Y$ be the identity map, where $Y'$ is the set $Y$ equipped
  with the discrete topology. We consider the situation
  \eqref{eq:cartesian-fibered-BC-trivial}.
  If $\mathcal{F}$ is a sheaf on
  $Y,$ then $d_*d^*(\mathcal{F})$ is $f_*$-acyclic and $q^*d_*d^*(\mathcal{F})$ is
  $g_*$-acyclic. 

  The first claim is obvious since $d_*d^*(\mathcal{F})$ is the flabby
  sheaf of discontinuous sections of the {\'e}tale space of
  $\mathcal{F}$; this is the first step in the Godement 
  resolution of $\mathcal{F}.$
  To prove the second claim, we expand diagram 
  \eqref{eq:cartesian-fibered-BC-trivial} 
  to
  \begin{equation}
    \label{eq:cartesian-fibered-BC-trivial-extended}
    \xymatrix@C3cm{
      {Y' \times Z} \ar[r]^{d':=d\times \id_Z} \ar[d]^r &
      {W=Y \times Z} \ar[r]^{g=f\times \id_Z} \ar[d]^q &
      {V=X \times Z} \ar[d]^p \\
      {Y'} \ar[r]^d &
      {Y} \ar[r]^f &
      {X}
    }
  \end{equation}
  where $r$ is the projection and $d':= d \times \id_Z.$
  Step 1 applied to $d^*(\mathcal{F})$ shows that the sheaf
  $\mathcal{E}:=r^*d^*(\mathcal{F})$ 
  is acyclic both for $d'_*$ and $(g \comp d')_*.$ 
  The Leray-Grothendieck spectral sequence 
  then shows that the sheaf
  $d'_*(\mathcal{E})$ is $g_*$-acyclic.

  Note that any pre-$\infty$-acyclic space is connected.
  So we can apply Theorem~\ref{t:fibered-base-change}
  to the 
  left square in
  \eqref{eq:cartesian-fibered-BC-trivial-extended}
  and obtain an isomorphism
  \begin{equation*}
    q^*d_*d^*(\mathcal{F}) \sira 
    d'_*r^*d^*(\mathcal{F}) = d'_*(\mathcal{E})
  \end{equation*}
  of sheaves. Hence 
  $q^*d_*d^*(\mathcal{F})$ is $g_*$-acyclic.

  \textbf{Step 3:}
  Let $\mathcal{F}$ be a sheaf on $Y.$ Let $\mathcal{F} \hra
  \mathcal{G}$ be its Godement resolution. Since all components of
  $\mathcal{G}$ are the image under $d_*d^*$ of some sheaf on $Y,$
  Step 2 shows that the morphism
  $p^*(Rf_*(\mathcal{F})) \ra Rg_*(q^*(\mathcal{F}))$ is given by
  $p^*(f_*(\mathcal{G}) \ra g_*(q^*(\mathcal{G})),$ and this morphism
  is an isomorphism by Theorem~\ref{t:fibered-base-change}.

  Now, using suitable truncation functors, it is easy to generalize
  this result from the sheaf $\mathcal{F}$ to arbitrary objects of
  $D^b(Y)$ and $D^+(Y).$
\end{proof}

\def\cprime{$'$} \def\cprime{$'$} \def\cprime{$'$} \def\cprime{$'$}
  \def\Dbar{\leavevmode\lower.6ex\hbox to 0pt{\hskip-.23ex \accent"16\hss}D}
  \def\cprime{$'$} \def\cprime{$'$}
\providecommand{\bysame}{\leavevmode\hbox to3em{\hrulefill}\thinspace}
\providecommand{\MR}{\relax\ifhmode\unskip\space\fi MR }
\providecommand{\MRhref}[2]{%
  \href{http://www.ams.org/mathscinet-getitem?mr=#1}{#2}
}
\providecommand{\href}[2]{#2}


\end{document}